\documentclass[12pt]{article}
\usepackage[latin1]{inputenc}

\usepackage[latin1]{inputenc}
\usepackage{fullpage}
\usepackage{amsfonts}
\usepackage{amssymb}
\usepackage{amsmath}
\usepackage[mathscr]{euscript}
\usepackage{wasysym}
\usepackage{tikz-cd}
\usepackage{stmaryrd}
\usepackage{color}
\usepackage{cite}
\usepackage[colors]{optsys}

\usepackage{algorithm}
\usepackage[algo2e,boxruled]{algorithm2e}

\renewcommand{\depart}{u} 
\renewcommand{\Depart}{U} 
\renewcommand{\DEPART}{\mathcal{U}} 
\renewcommand{\arrivee}{v} 
\renewcommand{\Arrivee}{V} 
\renewcommand{\ARRIVEE}{\mathcal{V}} 

\newcommand{\DEPARTbis}{\mathcal{W}} 
\newcommand{\arriveebis}{w}

\newcommand{\ARRIVEEbis}{\mathcal{W}}

\renewcommand{\arriveeter}{w}

\renewcommand{\ARRIVEEter}{\mathcal{W}}

\newcommand{\LinearQuadratic}{q}
\newcommand{\Quadratic}{Q}
\newcommand{\SquareMapping}{\sigma}
\newcommand{\VecteurObjective}{\pi}
\newcommand{\MatriceObjective}{\Pi}
\newcommand{\VecteurConstraint}{\gamma}
\newcommand{\MatriceConstraint}{\Gamma}
\newcommand{\NonsingularMatrix}{\Sigma}
\renewcommand{\matrice}{\Lambda}
\renewcommand{\vecteur}{\delta}
\newcommand{\scalaire}{\varpi}

\renewcommand{\PRIMAL}{{\mathcal X}}
\renewcommand{\DUAL}{{\mathcal Y}}
\renewcommand{\UNCERTAIN}{{\mathcal W}}
\renewcommand{\uncertainbis}{z}

\newcommand{\perspective}[1]{\widehat{#1}}
 
\newcommand{\constraintdim}{p}
\newcommand{\FONCTIONDEPART}{F}

\title{Conditional Infimum,\\ Hidden Convexity
  and S-Procedure}

\author{Jean-Philippe Chancelier and Michel De Lara, \\ 
 Cermics, \'Ecole nationale des ponts et chauss\'ees, IP Paris, France}

\begin{document}

\maketitle

\begin{abstract}
  Detecting hidden convexity is one of the tools to address nonconvex minimization problems,
  and find global minimizers. We introduce the notion of conditional infimum, develop the theory,
  and establish a tower property, relevant for minimization problems. Then, we illustrate how
  the conditional infimum is instrumental in revealing hidden convexity. Thus equipped, we provide
  a new sufficient condition for hidden convexity in nonconvex quadratic minimization problems,
  that encompasses and goes beyond known results (with the notion of block-signed pair matrix-vector).
  We also show how the conditional infimum is especially adapted to tackle the so-called S-procedure.
\end{abstract}



\section{Introduction}
\label{Introduction}

Convex minimization problems display well-known features that make their
numerical resolution appealing. In particular, convex minimization algorithms
are known to be simpler and less computationally intensive, in comparison with
nonconvex ones.  Thus, it is tempting to ``convexify'' a problem in order to
solve it, rather than to use nonconvex optimization.  More generally, it has
long been searched how to relate a nonconvex minimization problem to a convex
one.  If the original nonconvex minimization problem is formulated on a convex
set, then the convex lower envelope of the objective function has the same
minimum and a solution (argmin) of the original nonconvex problem is solution of
the convex lower envelope problem \cite[Proposition~11]{Horst:1984}.  Needless
to say that computing the lower envelope can be at least as difficult as solving
the original nonconvex problem.  This is why other approaches have been
developed, like convexification by domain or range transformation, as exposed
in~\cite{Horst:1984}, which provides a survey.

The vocable of ``hidden convexity'' covers different approaches: duality and
biduality analysis like in~\cite{BenTal-Ben-Teboulle:1996}; recasting a
nonconvex optimization problem as a convex one in~\cite[\S8.2.7]{Beck:2014};
identifying classes of nonconvex optimization problems whose convex relaxations
have optimal solutions which at the same time are global optimal solutions of
the original nonconvex problems~\cite{Ben-Tal-den-Hertog-Laurent:2011}.  A
survey of hidden convex optimization can be found in~\cite{XiaYong:2020}, with
its focus on three widely used ways to reveal the hidden convex structure for
different classes of nonconvex optimization problems.
In~\cite{Chancelier-DeLara:2021_ECAPRA_JCA}, we dealt with the notion of
\emph{hidden convexity in a function} as one of \emph{convex factorization}\footnote{%
Giving so-called \emph{convex composite} functions.}
  ---
that we characterized by means of one-sided linear conjugacies --- in the
following sense.  We considered a set~$\UNCERTAIN$, a function
\( \fonctionuncertain \colon \UNCERTAIN \to \barRR \), a vector space~$\PRIMAL$
and a mapping \( \theta \colon \UNCERTAIN \to \PRIMAL \).  Then, we said that the
function \( \fonctionuncertain \colon \UNCERTAIN \to \barRR \) displayed hidden
convexity with respect to the mapping~\( \theta \) if there existed a convex function
\( \fonctionprimal \colon \PRIMAL \to \barRR \) such that
\( \fonctionuncertain = \fonctionprimal \circ \theta \).

%

In this paper, we propose a new way to reveal hidden convexity by means of what
we call the ``conditional infimum'', an umbrella notion for well-known
operations in optimization.  It covers the operation of marginalization (that
is, partial minimization as in \cite[Theorem~5.3]{Rockafellar:1970}), widely
used in optimization, especially in the context of studying how optimal values
and optimal solutions depend on the parameters in a given problem.  Another
operation through which new functions are constructed by minimization is the
so-called epi-composition, developed by Rockafellar (see
\cite[p.~27]{Rockafellar-Wets:1998} and the historical note in
\cite[p.~36]{Rockafellar-Wets:1998}); epi-composition is called infimal
postcomposition in \cite[p.~214]{Bauschke-Combettes:2017}.
Notice that the vocable of marginalization hinges at a corresponding operation
(of partial integration) in probability theory.  A nice parallelism between
optimization and probability theories has been pointed out by several authors
\cite{DelMoral:1997,Akian-Quadrat-Viot:1998}.  Following this approach, we have
relabelled epi-composition as \conditionalinfimum\ in
\cite[Definition~2.4]{Chancelier-DeLara:2021_ECAPRA_JCA}, with the notation
\( \ConditionalInfimum{\theta}{\fonctionprimal} \).
The expression ``conditional infimum'' appears in the conclusion part of
\cite{Witsenhausen:1975b}, where it is defined with respect to a partition
field, that is, a subset of the power set which is closed \wrt\ (with respect
to) union and intersection, countable or not; however, the corresponding theory
has not been developed in~\cite{Witsenhausen:1975b}.  Related notions can be
found --- but defined on a measurable space equipped with a unitary Maslov measure
--- in the following works:
in~\cite{DelMoral:1997}, the theory of performance is sketched and the
``conditional performance'' is defined;
in~\cite{Akian-Quadrat-Viot:1998}, the ``conditional cost excess'' is defined;
in~\cite{Barron-Suprema-Jensen:2003}, the ``conditional essential supremum'' is
defined.
In this paper, we define the conditional infimum with respect to a
correspondence between two sets, without requiring measurable structures, and we
study its properties in the perspective of applications to optimization.
\medskip

The paper is organized as follows.
In Sect.~\ref{Conditional_infimum_with_respect_to_a_correspondence}, we provide
a definition of the conditional infimum (and supremum) of a function with
respect to a correspondence (between two sets), followed by examples, main
properties and relations to minimization problems.
%
The paper
\cite{Ben-Tal-den-Hertog-Laurent:2011} claims that the main mathematical tools
to detect hidden convexity are results on the range space of (indefinite)
quadratic forms and the S-procedure (called the S-lemma). Interestingly, these
are the two applications that we develop.
In Sect.~\ref{Hidden_convexity_in_quadratic_optimization_problems}, we provide a
new sufficient condition for detecting hidden convexity in nonconvex
quadratic minimization problems.
In Sect.~\ref{Conditional_infimum_and_the_S-procedure}, we provide new
sufficient conditions --- expressed in terms of conditional infimum --- for the
so-called S-procedure to hold true.
We conclude in Sect.~\ref{Conclusion}.
In Appendix~\ref{Appendix}, we provide additional material and some proofs.

\section{Conditional infimum with respect to a correspondence}
\label{Conditional_infimum_with_respect_to_a_correspondence}

In~\S\ref{Definitions_and_examples_of_the_conditional_infimum}, we provide a
definition of the conditional infimum (and supremum) of a function with respect
to a correspondence between two sets, and we provide examples.  Then, we expose
properties of the conditional infimum
in~\S\ref{Properties_of_the_conditional_infimum}.
In~\S.\ref{Applications_of_the_conditional_infimum_to_minimization_problems} we
develop applications of the conditional infimum to minimization problems.

We use the notation
\( \ic{j,k}=\na{j, j+1,\ldots,k-1,k} \) for any pair of natural numbers such
that \( j \leq k \).
We denote $\barRR = \ClosedIntervalClosed{-\infty}{+\infty} $,
$\RR_{-}=\OpenIntervalClosed{-\infty}{0}$,
$\RR_{+} = \ClosedIntervalOpen{0}{+\infty} $,
$\RR_{++}=\OpenIntervalOpen{0}{+\infty}$.
As we manipulate functions with values in~$\barRR$, we adopt the Moreau
\emph{lower ($\LowPlus$)} and \emph{upper ($\UppPlus$) additions}
\cite{Moreau:1970}, which extend the usual addition~($+$) with
\( \np{+\infty} \LowPlus \np{-\infty}=\np{-\infty} \LowPlus \np{+\infty}=-\infty \) and
\( \np{+\infty} \UppPlus \np{-\infty}=\np{-\infty} \UppPlus \np{+\infty}=+\infty \).
For any set~$\UNCERTAIN$ and function
\( \fonctionuncertain \colon \UNCERTAIN \to \barRR \), its \emph{epigraph} is
\( \epigraph\fonctionuncertain= \defset{
  \np{\uncertain,t}\in\UNCERTAIN\times\RR}%
{\fonctionuncertain\np{\uncertain} \leq t} \), its \emph{strict epigraph} is
\( \epigraph_{s}\fonctionuncertain= \defset{
  \np{\uncertain,t}\in\UNCERTAIN\times\RR}%
{\fonctionuncertain\np{\uncertain} < t} \), its \emph{effective domain} is
\( \dom\fonctionuncertain= \defset{\uncertain\in\UNCERTAIN}{
  \fonctionuncertain\np{\uncertain} <+\infty} \).  A function
\( \fonctionuncertain \colon \UNCERTAIN \to \barRR \) is said to be \emph{convex}
if its epigraph is a convex set, \emph{proper} if it never takes the
value~$-\infty$ and that \( \dom\fonctionuncertain \not = \emptyset \), \emph{lower semi
  continuous (\lsc)} if its epigraph is closed.

\subsection{Definitions and examples of the conditional infimum}
\label{Definitions_and_examples_of_the_conditional_infimum}

In~\S\ref{Background_on_correspondences}, we provide background on
correspondences.  In~\S\ref{Definitions_of_the_conditional_infimum}, we give a
formal definition of the conditional infimum (and supremum) with respect to a
correspondence between two sets.  In~\S\ref{Examples} we give examples, and
in~\S\ref{Special_cases_and_examples} we explore several special cases.

\subsubsection{Background on correspondences}
\label{Background_on_correspondences}
In optimization, one is more familiar with set-valued mappings
\cite[Chapter~5]{Rockafellar-Wets:1998} than with correspondences, though the
two notions are essentially equivalent.  We favor the notion of correspondence
because, regarding conditional infimum, we will obtain nicer formulas with the
composition of correspondences than with the composition of set-valued mappings
(see Footnote~\ref{ft:tower_property}).

We recall that a correspondence~$\correspondence$ between two sets~$\DEPART$
and~$\ARRIVEE$ is a subset \( \correspondence \subset \DEPART\times\ARRIVEE \).
We denote \( \depart\correspondence\arrivee \iff 
\np{\depart,\arrivee} \in \correspondence \).
A \emph{foreset} of a correspondence~$\correspondence$ is
any set of the form \( \correspondence\arrivee = 
\bset{\depart \in \DEPART}{ \depart\correspondence\arrivee } \),
where \( \arrivee \in \ARRIVEE \),
or, by extension, of the form
\( \correspondence\Arrivee = \bset{\depart \in \DEPART}%
{ \exists \arrivee\in\Arrivee, \, \depart\correspondence\arrivee } \),
where \( \Arrivee \subset \ARRIVEE \).
An \emph{afterset} of a correspondence~$\correspondence$ is
any set of the form \( \depart\correspondence=
\bset{\arrivee \in \ARRIVEE}{ \depart\correspondence\arrivee } \),
where \( \depart \in \DEPART \),
or, by extension, of the form
\( \Depart\correspondence= \bset{\arrivee \in \ARRIVEE}%
{ \exists  \depart \in \Depart \eqsepv \depart\correspondence\arrivee } \),
where \( \Depart \subset \DEPART \).
%
%
%
We denote by \( \Converse{\correspondence} \subset \ARRIVEE\times\DEPART \) the
\emph{inverse correspondence}
between the two sets~\( \ARRIVEE \) and \( \DEPART \) given by
\( \arrivee \Converse{\correspondence} \depart \iff
\depart \correspondence \arrivee \).
The \emph{domain} and the \emph{range} of a correspondence~$\correspondence$
are given respectively by 
\( \dom\correspondence = \bset{ \depart \in \DEPART }%
{ \depart\correspondence \not= \emptyset } \) 
and
\( \range\correspondence = \bset{ \arrivee \in \ARRIVEE }%
{ \correspondence\arrivee \not= \emptyset } \).
We have that \( \range\correspondence = \dom\Converse{\correspondence} \)
and \( \dom\correspondence = \range\Converse{\correspondence} \).
%
For any pair of correspondences $\correspondence$ between~$\DEPART$ and~$\ARRIVEE$ 
and $\correspondencebis$ between~$\ARRIVEE$ and $\ARRIVEEter$,
the \emph{composition}
\( \correspondence\correspondencebis \) denotes the correspondence
between the two sets~\( \DEPART \) and \( \ARRIVEEter \) given by, 
for any \( \np{\depart,\arriveeter} \in \DEPART\times\ARRIVEEter \),
\( \depart \np{\correspondence\correspondencebis} \arriveeter
\iff \exists \arrivee \in \ARRIVEE \) such that 
\( \depart \correspondence \arrivee \) and 
\( \arrivee \correspondencebis \arriveeter \).


\subsubsection{Definition of the conditional infimum}
\label{Definitions_of_the_conditional_infimum}

We give a formal definition of the conditional infimum (and supremum) with
respect to a correspondence between two sets.

\subsubsubsection{Optimization over a subset}

First, like in Probability theory\footnote{%
  Even if we draw parallels between optimization and probability theories, we do
  not develop the parallelism to its potential full extent as, for instance, we
  do not consider the equivalent of a generic probability distribution.
  Compared to
  \cite{DelMoral:1997,Akian-Quadrat-Viot:1998,Akian:1999,Barron-Suprema-Jensen:2003}
  which consider Maslov measures and densities --- that is, an analog of
  probability measures --- we could say that, in this paper, we only focus on the
  theory of the conditional infimum/supremum for the analog of the uniform
  probability (see the introduction of~\cite{Akian:1999}).  } where one starts
by defining the conditional probability \wrt\ a subset of the sample space, we
define the conditional infimum (and supremum) \wrt\ a subset.
We adopt the conventions\footnote{%
  Such conventions arise naturally as the mapping
  \( \Depart \in 2^\DEPART \mapsto \inf_{\Depart} \fonctiondepart = \inf_{\depart \in
    \Depart} \fonctiondepart\np{\depart} \) is nonincreasing and as the mapping
  \( \Depart \in 2^\DEPART \mapsto \sup_{\Depart} \fonctiondepart = \inf_{\depart \in
    \Depart} \fonctiondepart\np{\depart} \) is nondecreasing.  However, one has
  to be careful because
  \( \inf_{\Depart} \fonctiondepart \leq \sup_{\Depart} \fonctiondepart \) if
  \( \Depart \neq \emptyset \), but
  \( +\infty = \inf_{\emptyset} \fonctiondepart > \sup_{\emptyset } \fonctiondepart = -\infty \).  }
that 
\cite[p.~1]{Rockafellar-Wets:1998} 
\begin{equation}
  \inf_{\emptyset } \fonctiondepart = 
  \inf_{\depart \in \emptyset } \fonctiondepart\np{\depart} = +\infty 
  \mtext{ and } 
  \sup_{\emptyset } \fonctiondepart =
  \sup_{\depart \in \emptyset } \fonctiondepart\np{\depart} = -\infty \eqfinp
  \label{eq:correspondence_convention}
\end{equation}

\begin{definition}
  Let $ \fonctiondepart \colon \DEPART \to \barRR $ be a function
  and \( \Depart \subset \DEPART \) a subset.
  We define the \emph{conditional infimum}
  (resp. the \emph{conditional supremum})
  of the function~$\fonctiondepart$ with respect to the subset~\( \Depart \)
  by 
  \begin{subequations}
    \begin{equation}
      \InfCond{\fonctiondepart}{\Depart} =
      \inf_{\depart \in \Depart} \fonctiondepart\np{\depart}
      \eqsepv
      \SupCond{\fonctiondepart}{\Depart} =
      \sup_{\depart \in \Depart} \fonctiondepart\np{\depart}
      \eqfinp
      \label{eq:subset_conditional_infimum}  
    \end{equation}
    To make the link with minimization (resp. maximization) under constraint, we also define
    \begin{equation}
      \argminInfCond{\fonctiondepart}{\Depart} =
      \argmin_{\depart \in \Depart} \fonctiondepart\np{\depart}
      \eqsepv
      \argmaxSupCond{\fonctiondepart}{\Depart} =
      \argmax_{\depart \in \Depart} \fonctiondepart\np{\depart}
      \eqfinp
      \label{eq:subset_conditional_infimum_argmin}  
    \end{equation}
  \end{subequations}  
\end{definition}
Thus, this first notion of conditional infimum is related to minimization
problems under constraint.
All properties about the conditional infimum are easily carried to the
conditional supremum (and conversely) because
\begin{equation}
  -\SupCond{\fonctiondepart}{\correspondence} 
  =  \InfCond{-\fonctiondepart}{\correspondence}
  \eqsepv
  -\InfCond{\fonctiondepart}{\correspondence} 
  = \SupCond{-\fonctiondepart}{\correspondence} \eqfinp
  \label{eq:correspondence_conditional_supremum_infimum}
\end{equation}
In the sequel, we will favor the conditional infimum as we are interested in
applications to minimization problems.

\subsubsubsection{Definition of conditional infimum \wrt\ a correspondence} 

In the existing definitions of the conditional infimum of a function in the
literature
\cite{DelMoral:1997,Akian-Quadrat-Viot:1998,Barron-Suprema-Jensen:2003}, both
the original function and its conditional infimum are defined on a measurable
space equipped with a unitary Maslov measure.
By contrast, our definition of the conditional infimum of a function
(Definition~\ref{de:conditional_infimum} below) does not require a measurable
space (nor a Maslov measure) but a correspondence between two sets, a source set
and a target set; what is more, for a function whose domain is the source set,
its conditional infimum is defined on the target set.
The expression~\eqref{eq:correspondence_conditional_infimum} below is called the
marginal function (or minimal value function) in~\cite[Equation~(1.7.3),
page~41]{Pallaschke-Rolewicz:1997} (and noted
\( \overline{\fonctiondepart\correspondence} \) when $\correspondence$ is a
multifunction). However, \( \overline{\fonctiondepart\correspondence} \) is not
interpreted as a conditional infimum, and the theory is not developed in
\cite{Pallaschke-Rolewicz:1997}.

\begin{definition}
  \label{de:conditional_infimum} 
  Let $ \fonctiondepart \colon \DEPART \to \barRR $ be a function
  and $\correspondence$ be a correspondence between the sets~$\DEPART$ and~$\ARRIVEE$.
  We define the \emph{conditional infimum} 
  of the function~$\fonctiondepart$ with respect to the correspondence~$\correspondence$ 
  as the function~\( \InfCond{\fonctiondepart}{\correspondence} \colon \ARRIVEE \to \barRR \)
  given by
  \begin{subequations}
    \begin{equation}
      \InfCond{\fonctiondepart}{\correspondence} \colon 
      \ARRIVEE \to \barRR \eqsepv
      \InfCond{\fonctiondepart}{\correspondence}\np{\arrivee} 
      = 
      \InfCond{\fonctiondepart}{\correspondence\arrivee}
      = \inf_{\depart \in \correspondence\arrivee} \fonctiondepart\np{\depart}
      \eqsepv \forall \arrivee \in \ARRIVEE
      \eqfinv
      \label{eq:correspondence_conditional_infimum}    
    \end{equation}
    where we have used the notation~\eqref{eq:subset_conditional_infimum}.   
    We also define the \emph{conditional argmin} 
    of the function~$\fonctiondepart$ with respect to the correspondence~$\correspondence$ 
    as the set-valued functions \( \argminInfCond{\fonctiondepart}{\correspondence} \colon \ARRIVEE \rightrightarrows \DEPART \)
    given by
    \begin{equation}
      \argminInfCond{\fonctiondepart}{\correspondence} \colon 
      \arrivee \mapsto 
      \argminInfCond{\fonctiondepart}{\correspondence\arrivee}
      =
      \bset{ \depart \in \correspondence\arrivee}{\fonctiondepart(\depart) =
        \InfCond{\fonctiondepart}{\correspondence\arrivee}}
      \eqsepv \forall \arrivee \in \ARRIVEE
      \eqfinp
      \label{eq:correspondence_conditional_infimum_argmin}
    \end{equation}
  \end{subequations}
\end{definition}

The following interpretation of the conditional infimum is a straightforward consequence of
Definition~\ref{de:conditional_infimum}.
\begin{proposition}
  Let $ \fonctiondepart \colon \DEPART \to \barRR $ be a function
  and $\correspondence$ be a correspondence between the sets~$\DEPART$ and~$\ARRIVEE$.
  We have that
  \begin{equation}
    \InfCond{\fonctiondepart}{\correspondence}
    =
    \max\defset{ \phi \colon {\ARRIVEE} \to \barRR }{ \forall \arrivee \in \ARRIVEE\eqsepv 
      \depart \in \correspondence\arrivee \implies \phi\np{{\arrivee}} \leq \fonctiondepart\np{\depart} }
    \eqfinp               
  \end{equation}
\end{proposition}

As a consequence of~\eqref{eq:correspondence_conditional_infimum},
\eqref{eq:correspondence_convention} and of the definition of the range of a
correspondence in~\S\ref{Background_on_correspondences}, we have that
\( \arrivee \not\in \range\correspondence \implies
\InfCond{\fonctiondepart}{\correspondence}\np{\arrivee} = +\infty \), hence we have
the inclusion\footnote{%
  To the left hand side of the inclusion, the notation $\dom$ refers to the
  effective domain of a \emph{function}, whereas to the right hand side, the
  notation $\dom$ refers to the domain of a \emph{correspondence}.}
\begin{equation}
  \dom\bp{\ConditionalInfimum{\correspondence}{\fonctiondepart}} \subset
  \range\correspondence = \dom \Converse{\correspondence}
  \eqfinp
  \label{eq:dom_ConditionalInfimum_subset_range}
\end{equation}

\subsubsection{Examples}
\label{Examples}

The value function of a minimization problem under parametric constraints is an
example of conditional infimum.

\begin{example}[Value function]
  Let $ \fonctiondepart \colon \DEPART \to \barRR $ be a function and
  $\correspondence$ be a correspondence between the sets~$\DEPART$
  and~$\ARRIVEE$.  Recall that the correspondence is a subset
  \( \correspondence \subset \DEPART\times\ARRIVEE \), and let us interpret the foreset
  \( \correspondence\arrivee = \bset{\depart \in \DEPART}{
    \np{\depart,\arrivee}\in\correspondence } \subset\DEPART \) as a set of constraints
  parameterized by~\( \arrivee \in \ARRIVEE \).  Then, the function
  \( \InfCond{\fonctiondepart}{\correspondence} \colon \ARRIVEE \to \barRR \) is
  the value function associated with the family of minimization problems
  \( \inf_{\depart \in \correspondence\arrivee}\fonctiondepart\np{\depart}\),
  where \( \arrivee \in \ARRIVEE \).  Indeed, we have that
  \( \InfCond{\fonctiondepart}{\correspondence}\np{\arrivee} =\inf_{\depart \in
    \correspondence\arrivee}\fonctiondepart\np{\depart}\), for any
  \( \arrivee \in \ARRIVEE \) by~\eqref{eq:correspondence_conditional_infimum}
  and~\eqref{eq:subset_conditional_infimum}.
\end{example}

In mathematical programming, one obtains a more detailed description. 

\begin{example}[Mathematical programming]
  \label{ex:Mathematical_programming}
  Let $ \fonctiondepart_{0}, \fonctiondepart_{1}, \ldots, \fonctiondepart_{\constraintdim} \colon \DEPART
  \to \barRR $ be functions.
  The so-called \emph{$\RR_{+}^{\constraintdim}$-epigraph} of the mapping
  \( \DEPART\ni\depart \mapsto \bp{ \fonctiondepart_{1}\np{\depart}, \ldots,
    \fonctiondepart_{\constraintdim}\np{\depart} } \in \RR^{\constraintdim} \)
  is defined by 
  (see~\cite[page~49]{Pallaschke-Rolewicz:1997},
  \cite[p.~236]{Pennanen:1999},
  \cite[Definition~8]{Gissler-Hoheisel:2023} and also~\S\ref{Application_to_lower_bound_convex_programs})
  \begin{subequations}
    \begin{equation}
      \Epigraph_{\RR_{+}^{\constraintdim}} \np{\fonctiondepart_{1}, \ldots,\fonctiondepart_{\constraintdim} }
      = \defset{ \np{\depart,\alpha_{1},\ldots,\alpha_{\constraintdim}} \in \DEPART{\times}\RR^{\constraintdim}}%
      { \fonctiondepart_{1}\np{\depart} \leq\alpha_{1}, \ldots,
        \fonctiondepart_{\constraintdim}\np{\depart} \leq\alpha_{\constraintdim} }
      \label{eq:RR_+p-epigraph}
    \end{equation}
    and defines a \emph{$\RR_{+}^{\constraintdim}$-epigraphic correspondence} between the sets~$\DEPART$ and~$\RR^{\constraintdim}$.
    Then, the value function of the classic mathematical programming minimization problem is
    \begin{equation}
      \inf_{\substack{\fonctiondepart_{1}\np{\depart} \leq\alpha_{1}\\
          \ldots\\ \fonctiondepart_{\constraintdim}\np{\depart} \leq\alpha_{\constraintdim} }}
      \fonctiondepart_{0}\np{\depart}
      = \InfCond{\fonctiondepart_{0}}{\Epigraph_{\RR_{+}^{\constraintdim}}
        \np{\fonctiondepart_{1}, \ldots,\fonctiondepart_{\constraintdim} }}\np{\alpha_{1},\ldots,\alpha_{\constraintdim}}
      \eqsepv \forall \np{\alpha_{1},\ldots,\alpha_{\constraintdim}} \in \RR^{\constraintdim} 
      \eqfinp
      \label{eq:value_function_of_the_classic_mathematical_programming_minimization_problem}
    \end{equation}
          \label{eq:RR_+p-epigraph_InfCond}
  \end{subequations}
\end{example}

\begin{example}[Conic hull of a function]
  

  Let $\DEPART$ be a real vector space.
  For any function $ \fonctiondepart \colon \DEPART \to \barRR $,
  \begin{itemize}
  \item
    the \emph{perspective function}~$\perspective{\fonctiondepart} \colon
    \RR\times\DEPART \to \barRR $ of the function~$\fonctiondepart$ is defined by
    \begin{subequations}
      \begin{align}
        \mtext{either }   \perspective{\fonctiondepart}(\lambda,\depart) 
        &=
          \begin{cases}
            \lambda \fonctiondepart(\depart/\lambda)
            & \text{ if } (\lambda,\depart) \in \RR_{++}\times\DEPART       \eqfinv
            \\
            +\infty 
            & \text{ else, }
          \end{cases}
          \label{eq:perspective_function}
        \\
        \mtext{or }     \StrictEpigraph\perspective{\fonctiondepart}
        &=
          \RR_{++} \bp{\na{1}\times\StrictEpigraph\fonctiondepart}
          \eqfinv
          \label{eq:StrictEpigraph_perspective_function}
      \end{align}
    \end{subequations}
  \item
    and the \emph{conic hull} \( \fonctiondepart_{c} \colon \DEPART \to \barRR \)
    is defined by \cite[\S8.l]{Moreau:1966-1967} 
    \begin{subequations}
      \begin{align}
        \mtext{either }     \fonctiondepart_{c}\np{\depart}
        &=
          \inf_{\lambda >0} \lambda\fonctiondepart\np{\depart/\lambda}
          =
          \inf_{\lambda >0} \perspective{\fonctiondepart}(\lambda,\depart) 
          \eqsepv \forall \depart\in\DEPART
          \eqfinv
          \label{eq:conic_hull}    
        \\      
        \mtext{or } \StrictEpigraph\fonctiondepart_{c}
        &
          = \RR_{++}\StrictEpigraph\fonctiondepart
          \eqfinv 
          \label{eq:StrictEpigraph_conic_hull}  
      \end{align}
    \end{subequations}
  \end{itemize}
and it is easy to see that \( \fonctiondepart_{c}=
\InfCond{\perspective{\fonctiondepart}}{\RR_{++}\times\Delta_{\DEPART}} \),
    where \( \RR_{++}\times\Delta_{\DEPART} \subset \RR\times\DEPART\times\DEPART \simeq \np{\RR\times\DEPART}\times\DEPART \)
    has to be understood as a correspondence between \( \RR\times\DEPART\) and~\(\DEPART \).
  We will use the following property (whose proof is left to the reader).
  Let $\DEPART$ and $\ARRIVEE$ be two real vector spaces,
  and let $\correspondence$ be a correspondence between~$\DEPART$ and~$\ARRIVEE$.
  For any function $ \fonctiondepart \colon \DEPART \to \barRR $,
  we have that
  \begin{equation}
    \correspondence \mtext{ is a cone of } \DEPART\times\ARRIVEE \implies
    \bp{\InfCond{\fonctiondepart}{\correspondence}}_{c} =\InfCond{\fonctiondepart_{c}}{\correspondence}
    \eqfinp
    \label{eq:conic_hull_of_correspondence_conditional_supremum_infimum}
  \end{equation}
\end{example}

\subsubsection{Special cases of conditional infimum \wrt\ mappings}
\label{Special_cases_and_examples}

After having given a formal definition of the conditional infimum,
we explore several special cases and provide additional examples.

\subsubsubsection{Conditional infimum \wrt\ a 
  correspondence induced by a set-valued mapping
  \( \Theta \colon \DEPART \rightrightarrows \ARRIVEE \)}

\begin{subequations}
  Let \( \Theta \colon \DEPART \rightrightarrows \ARRIVEE \) be a set-valued mapping,
  that is, \( \Theta \colon \DEPART \to  2^\ARRIVEE \).
  We define the \emph{graph} of~$\Theta$ by
  \begin{equation}
    \graph_{\Theta} = \bset{ \np{\depart,\arrivee} \in \DEPART \times \ARRIVEE }%
    { \arrivee \in \Theta\np{\depart} } \subset \DEPART \times \ARRIVEE 
    \eqfinv
    \label{eq:graph_set-valued_mapping}
  \end{equation}
  and \( \Converse{\Theta} \colon \ARRIVEE \rightrightarrows \DEPART$ by $\Converse{\Theta}\np{\arrivee}=
  \nset{\depart \in \DEPART}{ \arrivee \in \Theta\np{\depart}}
  \).
  As the graph~$\graph_{\Theta}$ defines a correspondence between the two sets~\(
  \DEPART \) and \( \ARRIVEE \), we introduce 
  specific definitions and notations. 
  For any function $ \fonctiondepart \colon \DEPART \to \barRR $, 
  we define the \emph{conditional infimum} 
  \( \InfCond{\fonctiondepart}{\Theta} \colon \ARRIVEE \to \barRR \)
  of the function~$\fonctiondepart$ \emph{with respect to the set-valued mapping~$\Theta$}
  by
  \begin{equation}
    \label{eq:ConditionalInfimum_set-valued_mapping}
    \InfCond{\fonctiondepart}{\Theta}\np{\arrivee}
    = \InfCond{\fonctiondepart}{ \graph_{\Theta} }\np{\arrivee}
    = \InfCond{\fonctiondepart}{\Converse{\Theta}\np{\arrivee}}
    \eqsepv \forall \arrivee \in \ARRIVEE
    \eqfinp
  \end{equation}
\end{subequations}

\begin{example}
  In \cite[p.~416]{Rockafellar:1970} (part of \cite[Section~39]{Rockafellar:1970} on convex processes),
  Rockafellar defines \( (Af)(x)= \inf\defset{f(u)}{u\in A^{-1}x} \),
where $f$ is a convex function on $\RR^m$ and $A$ is a so-called convex process from $\RR^m$ to $\RR^n$,
that is, a set-valued mapping from $\RR^m$ to $\RR^n$ whose graph is a convex cone in $\RR^{m+n}$ containing the origin
\cite[p.~413]{Rockafellar:1970}.
Thus, $Af$ can be intepreted as a conditional infimum \wrt\ a correspondence which is 
a convex cone containing the origin.
\label{ex:convex_processes}
\end{example}

\subsubsubsection{Conditional infimum \wrt\ a 
  correspondence induced by a set-valued mapping 
  \( \Phi \colon \ARRIVEE \rightrightarrows \DEPART \) (the other way round)}

Let \( \Phi \colon \ARRIVEE \rightrightarrows \DEPART \) be a set-valued mapping.  Beware that, to
the difference of the set-valued mapping
\( \Theta \colon \DEPART \rightrightarrows \ARRIVEE \) above, the source set
is~$\ARRIVEE$ and the target set is~$\DEPART$.  We consider such set-valued
mappings to make the connection with how they are used in optimization to handle
constraints (see Footnote~\ref{ft:set-valued_mappings_optimization}).
Regarding the graph of~$\Phi$ in~\eqref{eq:graph_set-valued_mapping}, we have
\begin{subequations}
  \begin{align}
    \graph_{ \Phi } 
    &= 
      \bset{ \np{\arrivee,\depart} \in \ARRIVEE \times \DEPART }%
      { \depart \in \Phi\np{\arrivee} } \subset \ARRIVEE \times \DEPART 
      \eqfinv
      \intertext{and, defining \( \Converse{\Phi} \colon \DEPART \rightrightarrows \ARRIVEE \) 
      by \( \Converse{\Phi}\np{\depart} =\nset{\arrivee \in \ARRIVEE }%
      { \depart \in \Phi\np{\arrivee} } \), we get that}
      \npConverse{ \graph_{ \Phi } }
    &= 
      \bset{ \np{\depart,\arrivee} \in \DEPART \times \ARRIVEE }%
      { \depart \in \Phi\np{\arrivee} } 
      = \graph_{ \Converse{\Phi} }
      \subset \DEPART \times \ARRIVEE 
      \eqfinp    
  \end{align}
\end{subequations}
For any function $ \fonctiondepart \colon \DEPART \to \barRR $, we have the
properties\footnote{%
  Set-valued mappings are used in optimization because they offer a handy way to
  denote constraints as in the left hand side expressions
  in~\eqref{eq:ConditionalInfimum_set-valued_mapping_Converse}.  We will explain
  in Footnote~\ref{ft:tower_property} why we have chosen to favor
  correspondences rather than set-valued mappings.
  \label{ft:set-valued_mappings_optimization}}
\begin{equation}
  \inf_{ \depart \in \Phi\np{\arrivee} }\fonctiondepart\np{\depart}
  = 
  \InfCond{\fonctiondepart}{ \Converse{\graph_{ \Phi }} }\np{\arrivee}
  = \InfCond{\fonctiondepart}{ \Converse{\Phi} }\np{\arrivee}
  \eqsepv \forall \arrivee \in \ARRIVEE
  \eqfinp  
  \label{eq:ConditionalInfimum_set-valued_mapping_Converse}
\end{equation}
where we have used the notations in~\eqref{eq:ConditionalInfimum_set-valued_mapping}.

\subsubsubsection{Conditional infimum \wrt\ a 
  correspondence induced by a mapping \( \theta \colon \DEPART \to \ARRIVEE \)}

\begin{subequations}
  Let \( \theta \colon \DEPART \to \ARRIVEE \) be a  mapping.
  The graph of~$\theta$ in~\eqref{eq:graph_set-valued_mapping} is now
  \begin{equation}
    \graph_{\theta} = \bset{ \np{\depart,\arrivee} \in \DEPART \times \ARRIVEE }%
    { \theta\np{\depart} = \arrivee } \subset \DEPART \times \ARRIVEE 
    \eqfinp
    \label{eq:graph}
  \end{equation}
  For any function $ \fonctiondepart \colon \DEPART \to \barRR $, 
  Equation~\eqref{eq:ConditionalInfimum_set-valued_mapping} gives
  (with the mapping~$\theta$ identified with the set-valued mapping $\depart
  \mapsto \ba{\theta(\depart)}$)
  \begin{equation}
    \InfCond{\fonctiondepart}{\theta}\np{\arrivee}
    = \InfCond{\fonctiondepart}{ \graph_{\theta} }\np{\arrivee}
    =  
    \inf_{\depart \in \Converse{\theta}\np{\na{\arrivee}}} \fonctiondepart\np{\depart} 
    =
    \inf_{ \theta\np{\depart} =\arrivee }\fonctiondepart\np{\depart}
    \eqsepv \forall \arrivee \in \ARRIVEE
    \eqfinp
    \label{eq:mapping_ConditionalInfimum}
  \end{equation}
  As \( \theta \colon \DEPART \to \ARRIVEE \) is a mapping,
  it induces a set-valued mapping
  \( \Converse{\theta} \colon  \ARRIVEE \rightrightarrows \DEPART \).
  Using Equations~\eqref{eq:ConditionalInfimum_set-valued_mapping},
  for any function $\fonctionarrivee:\ARRIVEE \to \barRR$, we have the following properties:
  \begin{equation}
    \InfCond{\fonctionarrivee}{ \Converse{\theta} }\np{\depart}
    =    \InfCond{\fonctionarrivee}{\Converse{\graph_{\theta}} }\np{\depart}
    =
    \inf_{\arrivee = \theta\np{\depart} }
    \fonctionarrivee\np{\arrivee}
    =
    \bp{\fonctionarrivee \circ \theta }\np{\depart}
    \eqsepv \forall \depart \in \DEPART
    \eqfinp
    \label{eq:mapping_composition_and_ConditionalInfimum}  
  \end{equation}
\end{subequations}

\begin{example}(Epi-composition/infimal postcomposition) Let
  $ \fonctiondepart \colon \DEPART \to \barRR $ be a function and
  \( \theta \colon \DEPART \to \ARRIVEE \) be a mapping.  Then, the so-called
  \emph{epi-composition} of the function~$\fonctiondepart$ with the
  mapping~\( \theta \) --- developed by Rockafellar (see
  \cite[Equation~1(17), Exercise~1.31]{Rockafellar-Wets:1998} 
  and the historical note in
  \cite[p.~37]{Rockafellar-Wets:1998}), 
  and also called \emph{infimal postcomposition} in \cite[p.~214]{Bauschke-Combettes:2017} --- is exactly the
  conditional infimum \( \InfCond{\fonctiondepart}{\theta} \) in~\eqref{eq:mapping_ConditionalInfimum}.
  \label{epi-composition}
\end{example}

\subsection{Properties of the conditional infimum}
\label{Properties_of_the_conditional_infimum}

In~\S\ref{General_properties_of_the_conditional_infimum}, we expose general
properties of the conditional infimum.
%
%
In~\S\ref{Conditional_infimum_and_convexity}, we provide conditions for the
conditional infimum to be convex.
In~\S\ref{Conditional_infimum_and_one-sided_homogeneous_sublinear_couplings}, we
make the link between the new notion of one-sided homogeneous sublinear
couplings and conditional infimum and supremum.
%

\subsubsection{General properties of the conditional infimum}
\label{General_properties_of_the_conditional_infimum}
%

The $\wedge$ and $\vee$ operations over functions with the same domain, and taking
extended real values, are the infimum and supremum operations.
\begin{proposition}
  \quad
  \begin{enumerate}
  \item
    Nonnegativity:\\ 
    for any function $ \fonctiondepart \colon \DEPART \to \barRR $
    and for any correspondence~$\correspondence$
    between the sets~$\DEPART$ and~$\ARRIVEE$, 
    we have that
    \begin{equation}
      \InfCond{\fonctiondepart}{\correspondence} \geq 0
      \iff
      \bp{ \depart\in\dom\correspondence \implies \fonctiondepart\np{\depart} \geq 0 }
      \eqfinp
      \label{eq:correspondence_conditional_infimum_nonnegativity}
    \end{equation}
  \item 
    Strict epigraph:\\ 
    for any function $ \fonctiondepart \colon \DEPART \to \barRR $
    and for any correspondence~$\correspondence$
    between the sets~$\DEPART$ and~$\ARRIVEE$, 
    we have that
    \begin{equation}
      \StrictEpigraph\InfCond{\fonctiondepart}{\correspondence}=
      \Converse{\correspondence} \bp{\StrictEpigraph\fonctiondepart}
      \eqfinv
      \label{eq:correspondence_conditional_infimum_strict_epigraph} 
    \end{equation}
    where, on the left hand side, \(
    \StrictEpigraph\InfCond{\fonctiondepart}{\correspondence} \) is a subset of \(\ARRIVEE\times\RR\),
    and,  on the right hand side, 
    \( \Converse{\correspondence} \bp{\StrictEpigraph\fonctiondepart} \) is the
    composition of the correspondence~\( \Converse{\correspondence} \)
        between the sets~$\ARRIVEE$ and~$\DEPART$
    with \( \StrictEpigraph\fonctiondepart \), a subset of
    \(\DEPART\times\RR\) hence a correspondence between~$\DEPART$ and $\RR $.

   \item
     Linearity and sublinearity \wrt\ $\wedge$ and $\vee$ operations:\\ 
     for any correspondence~$\correspondence$
     between the sets~$\DEPART$ and~$\ARRIVEE$,
     for any family \( \sequence{\fonctiondepart_i}{i \in I} \) of functions 
        $ \fonctiondepart_i \colon \DEPART \to \barRR $,
        we have that
        \begin{subequations}
          \begin{align}
        \InfCond{ \Bwedge_{i \in I} \fonctiondepart_i }{\correspondence} 
        &=
          \Bwedge_{i \in I} \InfCond{ \fonctiondepart_i }{\correspondence} 
          \eqfinv 
          \label{eq:correspondence_conditional_infimum_properties_wedge}
        \\ 
        \Bvee_{i \in I} \InfCond{ \fonctiondepart_i }{\correspondence} 
        &\leq
          \InfCond{ \Bvee_{i \in I} \fonctiondepart_i }{\correspondence} 
          \eqfinp            
          \end{align}
        \end{subequations}
   \item
     Linearity and sublinearity \wrt\ the $\UppPlus$ operation:\\ 
     for any correspondence~$\correspondence$
     between the sets~$\DEPART$ and~$\ARRIVEE$,
     for any  functions $ \fonctiondepart \colon \DEPART \to \barRR $ 
          and $ \fonctiondepartbis \colon \DEPART \to \barRR $,
    we have that
    \begin{subequations}
      \begin{align}
        \InfCond{\fonctiondepart \UppPlus \fonctiondepartbis}{\correspondence}
        &\geq
          \InfCond{\fonctiondepart}{\correspondence} 
          \UppPlus \InfCond{\fonctiondepartbis}{\correspondence} 
          \eqfinv 
          \intertext{and\footnotemark
          if the function~$ \fonctiondepartbis$ is constant on any foreset~\( \correspondence{\arrivee} \),
          for any \( \arrivee\in\ARRIVEE \) --- or, equivalently,
          if \( \argminInfCond{\fonctiondepartbis}{\correspondence}\np{\arrivee}
          =\correspondence{\arrivee} \), for all \( \arrivee\in\ARRIVEE \), 
          or, equivalently, if \( \InfCond{\fonctiondepartbis}{\correspondence}=\SupCond{\fonctiondepartbis}{\correspondence}\)
          --- then}
          \InfCond{\fonctiondepart \UppPlus \fonctiondepartbis}{\correspondence}
        &=
          \InfCond{\fonctiondepart}{\correspondence} 
          \UppPlus \InfCond{\fonctiondepartbis}{\correspondence} 
          \eqfinp
          \label{eq:correspondence_conditional_infimum_property_UppPlus}
      \end{align}
    \end{subequations}
    \footnotetext{This formula for conditional infima has the flavor of 
      \( \EE\bc{\phi(X)\psi(Y)\mid Y} = \EE\bc{\phi(X)\mid Y}\psi(Y) \) in probability
      theory.}
  \item 
    Monotonicity with respect to functions:\\
    for any correspondence~$\correspondence$
    between the sets~$\DEPART$ and~$\ARRIVEE$, 
    we have that
    \begin{subequations}
      \begin{align}
        \intertext{$\bullet$ for any functions $ \fonctiondepart \colon \DEPART \to \barRR $
        and $ \fonctiondepartbis \colon \DEPART \to \barRR $,}
        \fonctiondepart \leq \fonctiondepartbis \implies
        \InfCond{\fonctiondepart}{\correspondence} 
        &\leq 
          \InfCond{\fonctiondepartbis}{\correspondence} 
          \eqfinv
          \label{eq:correspondence_conditional_infimum_properties_monotonicity}
        \\
        \InfCond{\fonctiondepart}{\DEPART}
        \leq 
        \InfCond{\fonctiondepart}{\correspondence\ARRIVEE} 
        &\leq \InfCond{\fonctiondepart}{\correspondence}
          (\arrivee) \eqsepv \forall \arrivee \in \ARRIVEE
          \eqfinv 
          %
          \intertext{$\bullet$ for any function $ \fonctiondepart \colon \DEPART \to \barRR $ 
          and for any nondecreasing function \( \varphi \colon \barRR \to \barRR \),}
        \varphi \circ \InfCond{\fonctiondepart }{\correspondence} 
        &\leq
          \InfCond{\varphi \circ \fonctiondepart}{\correspondence}
          \eqfinp
      \end{align}
    \end{subequations}
  \item
    Monotonicity with respect to correspondences:\\
    for any pair $\correspondence$, $\correspondencebis$ of correspondences
    between~$\DEPART$ and~$\ARRIVEE$ and for any function
    $ \fonctiondepart \colon \DEPART \to \barRR $, we have that
    \begin{equation}
          \correspondence \subset \correspondencebis \implies
        \InfCond{f}{\correspondence}
        \geq \InfCond{f}{\correspondencebis} 
          \eqfinp
          \label{eq:correspondence_conditional_infimum_properties_inclusion} 
      \end{equation}
  \item 
    Family of correspondences between~$\DEPART$ and~$\ARRIVEE$:\\ 
    for any family $\sequence{\correspondence_i}{i \in I}$ of correspondences
    between~$\DEPART$ and~$\ARRIVEE$ and for any 
    function $ \fonctiondepart \colon \DEPART \to \barRR $,
    we have that
    \begin{subequations}
      \begin{align}
        \bInfCond{f}{\bigcup_{i\in I}\correspondence_{i}}
        &= 
          \bwedge_{i\in I} \InfCond{f}{\correspondence_{i}}
          \eqfinv 
        \\
        \bInfCond{f}{\bigcap_{i\in I}\correspondence_{i}}
        &\geq 
          \bvee_{i\in I} \InfCond{f}{\correspondence_{i}}
          \eqfinp
      \end{align}
    \end{subequations}
  \item 
    Pushforward property:\footnote{%
      This formula for the conditional infimum has the flavor of the
      change of variable formula under pushforward probability, as the set~\( \Arrivee \subset \ARRIVEE \)
      below is pushed onto the set~\( \correspondence\Arrivee \subset \DEPART \).}
    \\
    for any correspondence $\correspondence$ between~$\DEPART$ and~$\ARRIVEE$,
    for any subset \( \Arrivee \subset \ARRIVEE \)
    and for any function $ \fonctiondepart \colon \DEPART \to \barRR $,
    we have that 
    \begin{equation}
      \bInfCond{ \nInfCond{\fonctiondepart}{\correspondence} }{\Arrivee}
      = 
      \bInfCond{\fonctiondepart}{\correspondence\Arrivee}
      \eqfinp
      \label{eq:correspondence_conditional_infimum_properties_pushforward}
    \end{equation}
  \item Tower property:\footnote{%
      This formula for conditional infima has the flavor of the tower property
      for conditional expectations. Had we defined the conditional infimum not
      \wrt\ a correspondence, but \wrt\ a set-valued mapping (like done
      in~\cite[\S~1.7]{Pallaschke-Rolewicz:1997}), the tower property would
      write, in a reverse way, as
      \( \bInfCond{ \nInfCond{\fonctiondepart}{ \Converse{\Phi} } }{ \Converse{\Psi} }
      = \bInfCond{\fonctiondepart}{ \npConverse{ \Psi \circ \Phi} } \), making appear the
      composition~$\Psi \circ \Phi$ of two set-valued mappings
      \( \Psi \colon \ARRIVEEbis \rightrightarrows \ARRIVEE \) and
      \( \Phi \colon \ARRIVEE \rightrightarrows \DEPART \) as in
      \cite[p.~151]{Rockafellar-Wets:1998}.  Indeed, the
      composition~$\Psi \circ \Phi$ satisfies
      \( \graph_{\Psi \circ \Phi}=\graph_{\Phi}\graph_{\Psi} \), hence
      \( \graph_{ \npConverse{ \Psi \circ \Phi} }= \npConverse{ \graph_{\Psi \circ \Phi} }=
      \npConverse{ \graph_{\Phi}\graph_{\Psi} }= \npConverse{ \graph_{\Psi} }
      \npConverse{ \graph_{\Phi} } \).  We prefer the formula
      \( \bInfCond{ \nInfCond{\fonctiondepart}{\correspondence}
      }{\correspondencebis} =
      \bInfCond{\fonctiondepart}{\correspondence\correspondencebis} \) to the
      formula
      \( \bInfCond{ \nInfCond{\fonctiondepart}{ \Converse{\Phi} } }{ \Converse{\Psi} }
      = \bInfCond{\fonctiondepart}{ \npConverse{ \Psi \circ \Phi} } \).
      \label{ft:tower_property} }\\
    For any pair of correspondences $\correspondence$ between~$\DEPART$ and~$\ARRIVEE$ 
    and $\correspondencebis$ between~$\ARRIVEE$ and~$\ARRIVEEter$,
    and for any function $ \fonctiondepart \colon \DEPART \to \barRR $,
    we have that 
    \begin{equation}
      \bInfCond{ \nInfCond{\fonctiondepart}{\correspondence} }{\correspondencebis} 
      = 
      \bInfCond{\fonctiondepart}{\correspondence\correspondencebis} 
      \eqfinp
      \label{eq:tower_property}
    \end{equation}
  \item 
    Right composition with mappings:\\
    for any correspondence~$\correspondence$ between~$\DEPART$ and~$\ARRIVEE$, 
    we have that 

    \noindent $\bullet$ for any function $ \fonctiondepartbis \colon \DEPARTbis \to \barRR $ 
    and for any mapping \( \theta \colon \DEPART \to \DEPARTbis \),
    \begin{subequations}
      \begin{align}
        \InfCond{\fonctiondepartbis \circ \theta}{\correspondence} 
        &= 
          \bInfCond{\fonctiondepartbis}{  \Converse{\graph_{\theta}}\correspondence }
          \eqfinv  
          \label{eq:correspondence_conditional_infimum_right_composition_correspondence}
          \intertext{$\bullet$ for any function $ \fonctiondepart \colon \DEPART \to \barRR $ 
          and for any mapping \( \theta \colon \ARRIVEEbis \to \ARRIVEE \),}
          \InfCond{\fonctiondepart}{\correspondence}  \circ \theta 
        &= 
          \bInfCond{\fonctiondepart}{  \correspondence\Converse{\graph_{\theta}} }
          \eqfinp 
          \label{eq:correspondence_conditional_infimum_correspondence_right_composition}
      \end{align}
    \end{subequations}
    %
  \end{enumerate}
\end{proposition}

\begin{proof} 
  Most of the claims are straightforward 
  consequences of the Definition~\ref{de:conditional_infimum} of the
  conditional infimum, and are left to the reader.

  \noindent $\bullet$
  We prove~\eqref{eq:correspondence_conditional_infimum_strict_epigraph}:
  \begin{align*}
    \np{\arrivee,t} \in \StrictEpigraph\InfCond{\fonctiondepart}{\correspondence}
    & \iff
      \InfCond{\fonctiondepart}{\correspondence}\np{\arrivee} < t      
      \tag{by definition 
      of the strict epigraph}
    \\
    & \iff      
      \exists \depart \in \correspondence\arrivee 
      \eqsepv \fonctiondepart\np{\depart} < t 
      \tag{by definition~\eqref{eq:correspondence_conditional_infimum} of the
      conditional infimum \( \InfCond{\fonctiondepart}{\correspondence} \)}
    \\
    & \iff
      \exists \depart \in \DEPART \eqsepv
      \arrivee \Converse{\correspondence} \depart 
      \text{ and } \np{\depart,t} \in \StrictEpigraph\fonctiondepart
      \tag{by definition 
      of the strict epigraph}
    \\
    & \iff
      \exists \depart \in \DEPART \eqsepv
      \arrivee \Converse{\correspondence} \depart 
      \text{ and } \depart ~\bp{{\StrictEpigraph\fonctiondepart}}~ t 
    \\
    & \iff
      \np{\arrivee,t} \in \Converse{\correspondence} \bp{\StrictEpigraph\fonctiondepart}
      \tag{by definition of the composition of correspondences}
  \end{align*}
\medskip

  \noindent $\bullet$ We prove~\eqref{eq:tower_property} as follows.
  For any pair of correspondences $\correspondence$ between~$\DEPART$ and~$\ARRIVEE$
  and $\correspondencebis$ between~$\ARRIVEE$ and $\ARRIVEEter$,
  any function $ \fonctiondepart \colon \DEPART \to \barRR $
  and any \( \arriveebis \in \ARRIVEEbis \), we have that 
  \begin{align*}
    \bInfCond{ \nInfCond{\fonctiondepart}{\correspondencebis} }{\correspondence}\np{\arriveebis}
    &=
      \inf_{ \arrivee \in \correspondence\arriveebis }
      \InfCond{\fonctiondepart}{\correspondencebis}\np{\arrivee} 
      \tag{by definition~\eqref{eq:correspondence_conditional_infimum} of the 
      conditional infimum}
    \\
    &=
      \inf_{ \arrivee \in \correspondence\arriveebis }
      \inf_{ \depart \in \correspondencebis\arrivee }
      \fonctiondepart\np{\depart}
      \tag{by definition~\eqref{eq:correspondence_conditional_infimum} of the 
      conditional infimum}
    \\
    &=
      \inf_{ \depart \in \correspondencebis\arrivee, \arrivee \in \correspondence\arriveebis }
      \fonctiondepart\np{\depart}
    \\
    &=
      \inf_{ \depart \in \correspondencebis\correspondence\arriveebis }
      \fonctiondepart\np{\depart}
      \tag{by definition of the composition of two correspondences}
    \\
    &=
      \bInfCond{\fonctiondepart}{\correspondence\correspondencebis} 
      \np{\arriveebis}
      \tag{by definition~\eqref{eq:correspondence_conditional_infimum} of the 
      conditional infimum}
  \end{align*}
  
  \noindent $\bullet$
  We prove~\eqref{eq:correspondence_conditional_infimum_right_composition_correspondence}
  as follows.
  For any correspondence~$\correspondence$ between~$\DEPART$ and~$\ARRIVEE$, 
  any function $ \fonctiondepartbis \colon \DEPARTbis \to \barRR $,
  any mapping \( \theta \colon \DEPART \to \DEPARTbis \)
  and any \( \arrivee\in\ARRIVEE \), we have that 
  \begin{align*}
    \nInfCond{\fonctiondepartbis \circ \theta}{\correspondence}
    &=
      \bInfCond{\nInfCond{\fonctiondepartbis}{\Converse{\graph_{\theta}}}}{\correspondence}
      \tag{as $\fonctiondepartbis \circ \theta=\InfCond{\fonctiondepartbis}{\Converse{\graph_{\theta}}}$
      by~\eqref{eq:mapping_composition_and_ConditionalInfimum}}
    \\
    &=\bInfCond{\fonctiondepartbis}{  \Converse{\graph_{\theta}}\correspondence }
      \tag{by the tower property~\eqref{eq:tower_property}}
  \end{align*}
    
  \noindent $\bullet$ We
  prove~\eqref{eq:correspondence_conditional_infimum_correspondence_right_composition}
  as follows.  For any correspondence~$\correspondence$ between~$\DEPART$
  and~$\ARRIVEE$, any function $ \fonctiondepart \colon \DEPART \to \barRR $, any
  mapping \( \theta \colon \ARRIVEEbis \to \ARRIVEE \) and any
  \( \arriveebis \in \ARRIVEEbis \), we have that
  \begin{align*}
    { \nInfCond{\fonctiondepart}{\correspondence}  \circ \theta }
    &=
      \bInfCond{\nInfCond{\fonctiondepart}{\correspondence}}{{\Converse{\graph_{\theta}}}}
      \tag{by~\eqref{eq:mapping_composition_and_ConditionalInfimum}}
    \\
    &=
      \InfCond{\fonctiondepart}{ \correspondence \Converse{\graph_{\theta}} }
      \tag{by the tower property~\eqref{eq:tower_property}}
  \end{align*}
  \smallskip

  This ends the proof.
\end{proof}

\subsubsection{Conditional infimum and convexity}
\label{Conditional_infimum_and_convexity}


\subsubsubsection{Background on vectorial orders}

Let~$\DEPART$ and~$\ARRIVEE$ be two real vector spaces.
Let \( \Cone\subset\ARRIVEE \) be a convex cone containing zero.
The vectorial order~\( \leq_{\Cone} \) on~\( \ARRIVEE \) induced by the cone~\(
\Cone\) is defined by:
\( \arrivee\leq_{\Cone}\arriveebis \iff \arriveebis-\arrivee \in\Cone \).
Let \( \theta \colon \DEPART \to \ARRIVEE \) be a mapping.
We define the
\emph{$\Cone$-epigraph} of \( \theta \) (see~\cite[p.~236]{Pennanen:1999},
\cite[Definition~8 and Equation~(3)]{Gissler-Hoheisel:2023}) by
\begin{equation}
  \Epigraph_{\Cone}\theta = \graph_{\Theta} + \na{0}{\times}\Cone
  = \defset{ \np{\depart,\arrivee} \in \DEPART{\times}\ARRIVEE}%
  { \theta\np{\depart} \leq_{\Cone} \arrivee }
  \eqfinp
  \label{eq:Cone-epigraph}
\end{equation}
We say that \( \theta \colon \DEPART \to \ARRIVEE \) is a \emph{\( \Cone \)-convex mapping} if
\( \Epigraph_{\Cone}\theta \) is a convex subset of~\( \DEPART{\times}\ARRIVEE \).

\subsubsubsection{Conditional infimum and convexity}

It is well known that the epi-composition --- that combines a function with a mapping
(see Example~\ref{epi-composition}) --- 
preserves convexity when the function is convex and the mapping is affine
\cite[Proposition~2.2 (b)]{Rockafellar-Wets:1998} and,  
more generally, that convexity is preserved in inf-projection of a jointly convex function
defined over a product space \cite[Proposition~2.2 (a)]{Rockafellar-Wets:1998} 
(see also \cite[Theorem~5.7]{Rockafellar:1970}),
and also by the transformation recalled in Example~\ref{ex:convex_processes}.
The proofs rely on linear transformations of the convex epigraph of the convex function,
that preserve convexity.
Preservation of convexity also occurs by
pre-composing a convex function with a nondecreasing convex function over the reals
(a result which can be extended using conic orders). 
We obtain the same results, but with a different approach. 
For this purpose, we need the following result that states, for correspondences,
the well-known result that the composition of graph-convex set-valued mappings is graph-convex
(see \cite[Theorem 10.37]{Rockafellar-Wets:1998} and also~\cite[Lemma 2.1]{Pennanen:1999}).
\begin{lemma}
  \label{lem:composition_convexity}
  Let $\DEPART$, $\ARRIVEE$ and $\ARRIVEEter$ be three real vector spaces,
  and let $\correspondence$ be a correspondence between~$\DEPART$ and~$\ARRIVEE$
  and $\correspondencebis$ be a correspondence between~$\ARRIVEE$ and $\ARRIVEEter$.
  If the subsets \( \correspondence \subset \DEPART\times\ARRIVEE \) and
  \( \correspondencebis \subset \ARRIVEE\times\ARRIVEEter \) are convex,
  then the subset \( \correspondence\correspondencebis \subset \DEPART\times\ARRIVEEter \) is convex.
\end{lemma}

\begin{proof}
%
The proof boils down to writing
\(     \correspondence\correspondencebis =
    \Projection_{ \DEPART\times\ARRIVEEter}
    \bp{ \np{\correspondence{\times}\ARRIVEEter} \cap  \np{\DEPART{\times}\correspondencebis}} \),
    which is a projection of an intersection of convex sets, hence is convex.
\end{proof}


As said above, the following Items~\ref{it:value_function_of_the_classic_mathematical_programming_minimization_problem}
and~\ref{it:epi-composition} are well known
(see, for example, \cite[Proposition~1.7.8, page~48]{Pallaschke-Rolewicz:1997}
and \cite[Corollary~1.7.11, page~49]{Pallaschke-Rolewicz:1997}).
\begin{corollary}\quad
  \label{cor:composition_convexity}
  \begin{enumerate}
  \item
    \label{it:infcond_convexity}
    Let~$\DEPART$ and~$\ARRIVEE$ be real vector spaces,
    \( \correspondence \subset \DEPART\times\ARRIVEE \) be a convex subset,
    and $ \fonctiondepart \colon \DEPART \to \barRR $ be a convex function.
    Then, the function \( \InfCond{\fonctiondepart}{\correspondence} \colon \ARRIVEE \to \barRR \)
    is convex.

  \item
    \label{it:composition_convexity}
    Let~$\DEPART$, $\ARRIVEE$ and~$\ARRIVEEbis$ be three real vector spaces,
    \( \Cone\subset\ARRIVEE \) be a convex cone containing zero,
    \( \theta \colon\ARRIVEEbis \to \ARRIVEE \) be a {\( \Cone \)-convex mapping}.
    %
      Let \( F \colon \DEPART\times\ARRIVEE \to \barRR \)
      be a convex function which is {\( \Cone \)-increasing} in the second argument.
      Then, the function \(\DEPART\times\ARRIVEEbis\ni\np{\depart,\arriveebis}\mapsto
\FONCTIONDEPART\bp{\depart,\theta\np{\arriveebis}}\in\barRR\) is convex.    
     
  \item
    \label{it:value_function_of_the_classic_mathematical_programming_minimization_problem}
    Let $\DEPART$ be a real vector space, and 
    $ \fonctiondepart_{0}, \fonctiondepart_{1}, \ldots, \fonctiondepart_{\constraintdim} \colon \DEPART\to\barRR $ be convex functions. Then,
    the value function
    \( \InfCond{\fonctiondepart_{0}}{\Epigraph_{\RR_{+}^{\constraintdim}} \np{\fonctiondepart_{1}, \ldots,\fonctiondepart_{\constraintdim} }}\)
    in~\eqref{eq:value_function_of_the_classic_mathematical_programming_minimization_problem}
    of the classic mathematical programming minimization problem
    (see Example~\ref{ex:Mathematical_programming}) is convex.

  \item
    \label{it:epi-composition}
    Let $\DEPART$ be a real vector space, 
    $ \fonctiondepart_{0}\colon\DEPART\to\barRR $ be a convex function,
    and \( \fonctiondepart_{1}, \ldots, \fonctiondepart_{\constraintdim}\colon\DEPART\to\RR \) be affine functions. 
    Then, the epi-composition (see Example~\ref{epi-composition}) 
    \( \InfCond{\fonctiondepart_{0}}{\np{\fonctiondepart_{1}, \ldots,\fonctiondepart_{\constraintdim} }}\)
    is convex.
  \end{enumerate}
\end{corollary}

\begin{proof}
  \begin{enumerate}
  \item 
    By~\eqref{eq:correspondence_conditional_infimum_strict_epigraph}, we have that
    \(      \StrictEpigraph\InfCond{\fonctiondepart}{\correspondence}=
    \Converse{\correspondence} \bp{\StrictEpigraph\fonctiondepart}\),
    which is a convex set by Lemma~\ref{lem:composition_convexity}.
    We conclude that the function \( \InfCond{\fonctiondepart}{\correspondence} \colon \ARRIVEE \to \barRR \)
    is convex.

      \item     
        For any \(\np{\depart,\arriveebis}\in\DEPART\times\ARRIVEEbis\), we have that 
        \begin{align*}
          \FONCTIONDEPART\bp{\depart,\theta\np{\arriveebis}}
          &=
            \inf\defset{\FONCTIONDEPART\bp{\depart,\arrivee}}%
            {\theta\np{\arriveebis} \leq_{\Cone} \arrivee }
            \tag{since $\FONCTIONDEPART$ is {\( \Cone \)-increasing} in the second argument}
          \\
          &=
  \inf\defset{\FONCTIONDEPART\bp{\depart,\arrivee}}{\arrivee\in\arriveebis~\Epigraph_{\Cone}\theta}
            \tag{by definition~\eqref{eq:Cone-epigraph} of~$\Epigraph_{\Cone}\theta$}
          \\
          &=
            \InfCond{\FONCTIONDEPART}{%
            \underbrace{\Delta_{\DEPART}\times\npConverse{\Epigraph_{\Cone}\theta}}_{\substack{\textrm{identified
            with a subset of}\\ \np{\DEPART\times\ARRIVEE}\times\np{\DEPART\times\ARRIVEEbis}}}
            }\np{\depart,\arriveebis}
            \tag{by definition~\eqref{eq:correspondence_conditional_infimum} of
            the conditional infimum}
        \end{align*}
As \( \Delta_{\DEPART}\subset\DEPART^2 \) is a convex set and
\(\Epigraph_{\Cone}\theta\subset\ARRIVEEbis\times\ARRIVEE\) is a convex set by assumption, 
then \( \Delta_{\DEPART}\times\npConverse{\Epigraph_{\Cone}\theta}\) is also a convex
subset of \(\DEPART^2\times\ARRIVEE\times\ARRIVEEbis\), and so is its
rearrangement as a subset of
\(\np{\DEPART\times\ARRIVEE}\times\np{\DEPART\times\ARRIVEEbis}\).
We conclude
  with Item~\ref{it:infcond_convexity} that the function \(\DEPART\times\ARRIVEEbis\ni
\FONCTIONDEPART\bp{\depart,\theta\np{\arriveebis}}\in\barRR\) is convex.    

  
  \item
    As $ \fonctiondepart_{1}, \ldots, \fonctiondepart_{\constraintdim}$ are convex functions,
    the $\RR_{+}^{\constraintdim}$-epigraph \( \Epigraph_{\RR_{+}^{\constraintdim}} \np{\fonctiondepart_{1}, \ldots,\fonctiondepart_{\constraintdim} }\)
    in~\eqref{eq:RR_+p-epigraph} is convex.
    We conclude with Item~\ref{it:infcond_convexity}.

  \item
    As \( \fonctiondepart_{1}, \ldots, \fonctiondepart_{\constraintdim} \) are affine functions,
    the graph \( \graph_{\np{\fonctiondepart_{1}, \ldots,\fonctiondepart_{\constraintdim} }} \) in~\eqref{eq:graph}
    is an affine subset, hence is convex.
    We conclude with Item~\ref{it:infcond_convexity}.
  \end{enumerate}
  This ends the proof.
\end{proof}

\subsubsection{Conditional infimum and one-sided homogeneous sublinear couplings}
\label{Conditional_infimum_and_one-sided_homogeneous_sublinear_couplings} 

One-sided linear couplings were introduced
in~\cite{Chancelier-DeLara:2021_ECAPRA_JCA}, and the link with conditional
infimum \wrt\ a mapping~\eqref{eq:mapping_composition_and_ConditionalInfimum}
was established.  Here, we generalize them to one-sided homogeneous sublinear
couplings, and we make the link with conditional infimum and supremum (as
defined in~\S\ref{Definitions_of_the_conditional_infimum}).

We consider a pair $(\PRIMAL, \DUAL)$ of vector spaces that are \emph{paired} in
the following sense \cite[p.~13]{Rockafellar:1974}: there exists a bilinear form
\( \nscal{\,}{} \colon \PRIMAL \times \DUAL \to \RR \) and locally convex topologies
that are compatible in the sense that the continuous linear forms on~$\PRIMAL$
are the functions \( \primal \in \PRIMAL \mapsto \nscal{\primal}{\dual} \), for all
\( \dual \in \DUAL \), and that the continuous linear forms on~$\DUAL$ are the
functions \( \dual \in \DUAL \mapsto \nscal{\primal}{\dual} \), for all \( \primal \in \PRIMAL \).
Now, we review concepts and notations related to the Fenchel conjugacy (we refer
the reader to \cite[Sect.~3]{Rockafellar:1974}).  For any functions
\( \fonctionprimal \colon \PRIMAL \to \barRR \) and
\( \fonctiondual \colon \DUAL \to \barRR \), the different conjugates are defined
by\footnote{%
  In convex analysis, one does not use~\( \LFMr{} \) and~\( \LFMbi{} \), but
  simply~\( \LFM{} \) and~\( ^{\Fenchelcoupling\Fenchelcoupling} \).  We
  use~\( \LFMbi{} \) to be consistent with the notation for general conjugacies.
  \label{ft:to_be_consistent_with_the_notation_for_general_conjugacies}}
\begin{subequations}
  \begin{align}
    \LFM{\fonctionprimal}\np{\dual} 
    &= 
      \sup_{\primal \in \PRIMAL} \bp{ \nscal{\primal}{\dual}
      -\fonctionprimal\np{\primal} }
      \eqsepv \forall \dual \in \DUAL
      \eqfinv
      \label{eq:Fenchel_conjugate}
    \\
    \LFMr{\fonctiondual}\np{\primal} 
    &= 
      \sup_{ \dual \in \DUAL} \bp{ \nscal{\primal}{\dual} 
      -\fonctiondual\np{\dual} } 
      \eqsepv \forall \primal \in \PRIMAL
      \eqfinv
      \label{eq:Fenchel_conjugate_reverse}
    \\
    \LFMbi{\fonctionprimal}\np{\primal} 
    &= 
      \sup_{\dual \in \DUAL} \bp{ \nscal{\primal}{\dual} 
      -\LFM{\fonctionprimal}\np{\dual} }
      \eqsepv \forall \primal \in \PRIMAL
      \eqfinp
      \label{eq:Fenchel_biconjugate}
  \end{align}
\end{subequations}
A function \( \fonctionprimal \colon \PRIMAL \to \barRR \), or
\( \fonctiondual \colon \DUAL \to \barRR \), is said to be
\emph{closed}\footnote{%
  We follow the terminology in \cite[p.~15]{Rockafellar:1974}, although it can
  be misleading.  Indeed, a valley function~\cite{Penot:2000} --- a function
  taking the value $-\infty$ on a closed subset (neither the empty set nor the whole
  set) and $+\infty$ outside --- is \lsc\ but not closed.  Some authors
  \cite{Borwein-Lewis:2006} use closed in the sense of \lsc.
  \label{ft:closed_function}}
if it is either \lsc\ and nowhere having the value $-\infty$,
or is the constant function~$-\infty$.
By definition, \emph{closed convex functions} are the two constant
functions~$-\infty$ and~$+\infty$ united with all proper convex \lsc\
functions\footnote{%
  In particular, any closed convex function that takes at least one finite value
  is necessarily proper convex~lsc.  Notice that a function taking the value
  $-\infty$ on a closed convex subset (neither the empty set nor the whole set) and
  $+\infty$ outside is convex~lsc, but is not closed convex (see
  Footnote~\ref{ft:closed_function}).
  \label{ft:closed_convex_function}}.

\begin{definition}
  Let $\UNCERTAIN$ be a nonempty set and
  \( \correspondence \subset \UNCERTAIN \times \PRIMAL \) 
  be a correspondence between $\UNCERTAIN$ and~$\PRIMAL$.
  We define the \emph{one-sided homogeneous sublinear coupling} $\Fenchelcoupling_{\correspondence}$
  between the set~$\UNCERTAIN$ and the vector space~$\DUAL$ by\footnote{%
    In a one-sided homogeneous sublinear coupling, 
    the second set possesses a linear structure 
    (and is even paired with a vector space by means of a bilinear form),
    whereas the first set is not required to carry any structure.}
  \begin{equation}
    \Fenchelcoupling_{\correspondence} \colon \UNCERTAIN \times \DUAL \to \barRR 
    \eqsepv 
    \Fenchelcoupling_{\correspondence}\np{\uncertain, \dual} 
    = \sup_{\primal\in\uncertain\correspondence}\FenchelCoupling{\primal}{\dual} 
    \eqsepv \forall \uncertain \in \UNCERTAIN
    \eqsepv \forall \dual \in \DUAL
    \eqfinp 
    \label{eq:one-sided_homogeneous_sublinear_coupling_second}
  \end{equation}
  For any functions \( \fonctionuncertain \colon \UNCERTAIN \to \barRR \) and
  \( \fonctiondual \colon \DUAL \to \barRR \), the different
  $\Fenchelcoupling_{\correspondence}$-conjugates are defined by
  \begin{subequations}
    \begin{align}
      \SFM{\fonctionuncertain}{\Fenchelcoupling_{\correspondence}}\np{\dual} 
      &= 
        \sup_{\uncertain \in \UNCERTAIN} \Bp{ \Fenchelcoupling_{\correspondence}\np{\uncertain, \dual} 
        \LowPlus \bp{-\fonctionuncertain\np{\uncertain}} } 
        \eqsepv \forall \dual \in \DUAL
        \eqfinv
      \\
      \SFMr{\fonctiondual}{\Fenchelcoupling_{\correspondence}}\np{\uncertain} 
      &= 
        \sup_{ \dual \in \DUAL} \Bp{ \Fenchelcoupling_{\correspondence}\np{\uncertain, \dual} 
        \LowPlus \bp{-\fonctiondual\np{\dual} } } 
        \eqsepv \forall \uncertain \in \UNCERTAIN
        \eqfinv
      \\
      \SFMbi{\fonctionuncertain}{\Fenchelcoupling_{\correspondence}}\np{\uncertain} 
      &= 
        \sup_{\dual \in \DUAL} \Bp{ \Fenchelcoupling_{\correspondence}\np{\uncertain, \dual} 
        \LowPlus \bp{-\SFM{\fonctionuncertain}{\Fenchelcoupling_{\correspondence}}\np{\dual} } }
        \eqsepv \forall \uncertain \in \UNCERTAIN
        \eqfinp
    \end{align}
  \end{subequations}
\end{definition}

Here below are expressions for the different
$\Fenchelcoupling_{\correspondence}$-conjugates of a function, that make use of
conditional infimum and supremum, and of the classic Fenchel conjugacy.
The proof follows the lines of the proofs of \cite[Proposition~2.5,
2.6]{Chancelier-DeLara:2021_ECAPRA_JCA}, and is thus left to the reader.

\begin{proposition}
  \label{pr:one-sided_homogeneous_sublinear_Fenchel-Moreau_conjugate_second}
  For any function \( \fonctiondual \colon \DUAL \to \barRR \), 
  the $\Fenchelcoupling_{\correspondence}'$-Fenchel-Moreau conjugate 
  \( \SFMr{\fonctiondual}{\Fenchelcoupling_{\correspondence}} \colon \UNCERTAIN \to \barRR \) 
  is given by 
  \begin{subequations}
    \begin{equation}
      \SFMr{\fonctiondual}{\Fenchelcoupling_{\correspondence}}=
      \ConditionalSupremum{\Converse{\correspondence}}{\LFMr{ \fonctiondual } }
      \eqfinp
      \label{eq:one-sided_homogeneous_sublinear_c'-Fenchel-Moreau_conjugate}
    \end{equation}
    For any function \( \fonctionuncertain \colon \UNCERTAIN \to \barRR \), 
    the $\Fenchelcoupling_{\correspondence}$-Fenchel-Moreau conjugate 
    \( \SFM{\fonctionuncertain}{\Fenchelcoupling_{\correspondence}} 
    \colon \DUAL \to \barRR \) is given by 
    \begin{equation}
      \SFM{\fonctionuncertain}{\Fenchelcoupling_{\correspondence}}=
      \LFM{ \bp{\ConditionalInfimum{\correspondence}{\fonctionuncertain}} }
      \eqfinv 
      \label{eq:one-sided_homogeneous_sublinear_Fenchel-Moreau_conjugate}
    \end{equation}
    and the $\Fenchelcoupling_{\correspondence}$-Fenchel-Moreau biconjugate 
    \( \SFMbi{\fonctionuncertain}{\Fenchelcoupling_{\correspondence}}
    \colon \UNCERTAIN \to \barRR \)
    is given by
    \begin{equation}
      \SFMbi{\fonctionuncertain}{\Fenchelcoupling_{\correspondence}}
      = 
      \ConditionalSupremum{\Converse{\correspondence}}{
        \LFMr{ \bp{ \SFM{\fonctionuncertain}{\Fenchelcoupling_{\correspondence}} } } }
      =
      \ConditionalSupremum{\Converse{\correspondence}}{
        \LFMbi{ \bp{\ConditionalInfimum{\correspondence}{\fonctionuncertain}} }
      }
      \eqfinp
      \label{eq:one-sided_homogeneous_sublinear_Fenchel-Moreau_biconjugate}
    \end{equation}
  \end{subequations}
  A function \( \fonctionuncertain \colon \UNCERTAIN \to \barRR \) is 
  $\Fenchelcoupling_{\correspondence}$-convex 
  if and only if it is the conditional supremum of
  a closed convex function \( \fonctionprimal \colon \PRIMAL \to \barRR \)
  \wrt\ the inverse correspondence~\( \Converse{\correspondence} \).
  More precisely, for any function \( \fonctionuncertain \colon \UNCERTAIN \to \barRR \),
  we have the equivalences
  \begin{subequations}
    \begin{align}
      &  
        \fonctionuncertain \textrm{ is $\Fenchelcoupling_{\correspondence}$-convex }
        \label{eq:coupling_correspondence-convex_function_a}
      \\
      \iff & 
             \fonctionuncertain =
             \SFMbi{\fonctionuncertain}{\Fenchelcoupling_{\correspondence}}
             \label{eq:coupling_correspondence-convex_function_b}
      \\
      \iff & 
             \fonctionuncertain =
              \ConditionalSupremum{\Converse{\correspondence}}{
    \LFMbi{ \bp{\ConditionalInfimum{\correspondence}{\fonctionuncertain}} } }
      \\
      &
              \textrm{ (where } 
             \LFMbi{ \bp{\ConditionalInfimum{\correspondence}{\fonctionuncertain}} }
        \textrm{ is a closed convex function) }
              \label{eq:coupling_correspondence-convex_function_c}
      \\
      \iff & \textrm{ there exists a closed convex function }
             \fonctionprimal \colon \PRIMAL  \to \barRR 
             \textrm{ such that }
             \fonctionuncertain =
             \ConditionalSupremum{\Converse{\correspondence}}{\fonctionprimal}
            \label{eq:coupling_correspondence-convex_function_d}
             \eqfinp
    \end{align}
    \label{eq:coupling_correspondence-convex_function}
  \end{subequations}
\end{proposition}

\subsection{Conditional infimum in minimization problems}
\label{Applications_of_the_conditional_infimum_to_minimization_problems}

In~\S\ref{Conditional_infimum_and_values/solutions_of_minimization_problems},
we relate the conditional infimum with values of minimization problems
and then with solutions. 
In~\S\ref{Application_to_fractional_programming}, we provide an application to fractional programming
and to lower bound convex programs in~\S\ref{Application_to_lower_bound_convex_programs}. 

\subsubsection{Conditional infimum and values/solutions of minimization problems}
\label{Conditional_infimum_and_values/solutions_of_minimization_problems}

With the conditional infimum, we now establish equalities and inequalities
between two minimization problems, an original problem on the set~$\UNCERTAIN$
and another one on the set~$\PRIMAL$, where the sets $\UNCERTAIN$ and $\PRIMAL$
are possibly different (in particular, $\PRIMAL$ might be a vector space,
whereas $\UNCERTAIN$ is not).
As we deal with optimization problems, we will often resort to the more telling
usage \( \inf_{\uncertain \in \Uncertain} \fonctionuncertain\np{\uncertain} \) or
\( \min_{\uncertain \in \Uncertain} \fonctionuncertain\np{\uncertain} \), rather
than \( \InfCond{\fonctionuncertain}{\Uncertain} \) as
in~\eqref{eq:subset_conditional_infimum}.

\subsubsubsection{Conditional infimum and values of minimization problems}

The following Proposition~\ref{pr:inf_ConditionalInfimum_inf_original} relates
the \emph{infimum} of an original function with the \emph{infimum} of its
conditional infimum.
\begin{proposition}
  \label{pr:inf_ConditionalInfimum_inf_original}
  We consider two sets $\UNCERTAIN$ and $\PRIMAL$, 
  a correspondence~$\correspondence$ between~$\UNCERTAIN$ and~$\PRIMAL$, 
  and a function \( \fonctionuncertain \colon \UNCERTAIN \to \barRR \).
  For any subset \( \Primal \subset \PRIMAL \), 
  we have the equality
  \begin{equation}
    \inf_{\uncertain \in \correspondence\Primal} \fonctionuncertain\np{\uncertain} 
    =
    \inf_{\primal \in \Primal}
    \bp{\nInfCond{\fonctionuncertain}{\correspondence}\np{\primal}}
    \eqfinp
    \label{eq:ConditionalInfimum_tower_property_equality_subset} 
  \end{equation}  
  For any subset \( \Uncertain \subset \UNCERTAIN \), we have the implications
  \begin{subequations}
    \begin{align}
      \Uncertain \subset \domain\correspondence
      & \implies
        \inf_{\uncertain \in \Uncertain} \fonctionuncertain\np{\uncertain} 
        \geq
        \inf_{\primal \in \Uncertain\correspondence} 
        \bp{\nInfCond{\fonctionuncertain}{\correspondence}\np{\primal}}
        \eqfinv
        \label{eq:ConditionalInfimum_tower_property_ineq}
      \\
      \correspondence\Converse{\correspondence}\Uncertain \subset 
      \Uncertain \subset \domain\correspondence
      & \implies
        \inf_{\uncertain \in \Uncertain} \fonctionuncertain\np{\uncertain} 
        =
        \inf_{\primal \in \Uncertain\correspondence} 
        \bp{\nInfCond{\fonctionuncertain}{\correspondence}\np{\primal}}
        \eqfinp
        \label{eq:ConditionalInfimum_tower_property}
    \end{align}
  \end{subequations}
\end{proposition}

When the set~\( \Primal \) is convex and the
function~\( \InfCond{\fonctionuncertain}{\correspondence}\) is convex, then the
right hand side in
Equation~\eqref{eq:ConditionalInfimum_tower_property_equality_subset} is a
convex program.

\begin{proof}
  The equality~\eqref{eq:ConditionalInfimum_tower_property_equality_subset}
  is proved as follows: 
  \begin{align*}
    \inf_{\uncertain \in \correspondence\Primal} \fonctionuncertain\np{\uncertain} 
    &=
      \InfCond{\fonctionuncertain}{\correspondence\Primal}
      \tag{by definition~\eqref{eq:subset_conditional_infimum}}
    \\
    &=
      \bInfCond{ \nInfCond{\fonctionuncertain}{\correspondence} }{\Primal}
      \tag{by the pushforward property~\eqref{eq:correspondence_conditional_infimum_properties_pushforward}}
    \\
    &=
      \inf_{\primal \in \Primal}
      \bp{\InfCond{\fonctionuncertain}{\correspondence}\np{\primal}}
      \eqfinp
      \tag{by definition~\eqref{eq:subset_conditional_infimum}}
  \end{align*}

  We suppose that \( \Uncertain \subset \domain\correspondence\) and we prove the
  right hand side inequality
  in~\eqref{eq:ConditionalInfimum_tower_property_ineq}.  First, we prove that
  $\Uncertain \subset \correspondence\Converse{\correspondence}\Uncertain$.  Indeed,
  if $\uncertain \in \Uncertain$ we have that
  $\uncertain \in \domain\correspondence$ as
  \( \Uncertain \subset \domain\correspondence\).  Therefore, there exists
  $\primal \in \PRIMAL $ such that $\uncertain \correspondence \primal$ or, equivalently, that
  $\primal \Converse{\correspondence} \uncertain$.  Now,
  $\uncertain \correspondence \primal$ and
  $\primal \Converse{\correspondence} \uncertain$ imply that
  $\uncertain \correspondence \Converse{\correspondence} \uncertain$ and thus
  $\uncertain \in \correspondence \Converse{\correspondence} \Uncertain$. Second,
  we obtain that
  \begin{align*}
    \inf_{\uncertain \in \Uncertain} \fonctionuncertain\np{\uncertain}
    &\ge 
      \inf_{\uncertain \in \correspondence \Converse{\correspondence} \Uncertain}
      \fonctionuncertain\np{\uncertain}
      \tag{since \( \Uncertain \subset \correspondence\Converse{\correspondence}\Uncertain \)}
    \\
    &=
      \inf_{\primal \in\Converse{\correspondence}\Uncertain}
      \bp{\InfCond{\fonctionuncertain}{\correspondence}\np{\primal}}
      \eqfinp
      \tag{by Equation~\eqref{eq:ConditionalInfimum_tower_property_equality_subset}
      with $\Primal=\Converse{\correspondence}\Uncertain$}
    \\ 
  \end{align*}

  When
  \( \correspondence\Converse{\correspondence}\Uncertain \subset \Uncertain \subset
  \domain\correspondence \), the right hand side equality
  in~\eqref{eq:ConditionalInfimum_tower_property} comes from the fact that the
  inequality above is an equality using that
  \( \correspondence\Converse{\correspondence}\Uncertain \subset \Uncertain \).
  
  \medskip
  This ends the proof. 
\end{proof}

\subsubsubsection{Conditional infimum and solutions of minimization problems}
\label{Conditional_infimum_and_solutions_of_minimization_problems}

With the conditional infimum, we now state sufficient conditions to relate the
optimal \emph{solutions} of two minimization problems.  The following
Proposition~\ref{pr:argmin_ConditionalInfimum_argmin_original} relates the
\emph{argmin} of an original function with the \emph{argmin} of its conditional
infimum.

\begin{proposition}
  \label{pr:argmin_ConditionalInfimum_argmin_original}
  We consider a function \( \fonctionuncertain \colon \UNCERTAIN \to \barRR \), 
  a subset \( \Uncertain \subset \UNCERTAIN \) 
  and the minimization problem 
  \begin{equation}
    \min_{ \uncertain\in\Uncertain } \fonctionuncertain\np{\uncertain} 
    \eqfinp  
    \label{eq:minimization_problem_Uncertain}
  \end{equation}
  Assume that there exists 
  \begin{subequations}
    \begin{enumerate}
    \item 
      a set~$\PRIMAL$,
      a correspondence~$\correspondence$ between~$\UNCERTAIN$ and~$\PRIMAL$, 
      and a function 
      \( \fonctionprimal \colon \PRIMAL \to \barRR \) such that 
      \begin{equation}
        \fonctionprimal\np{\primal} \leq 
        \InfCond{\fonctionuncertain}{\correspondence}\np{\primal}
        \eqsepv \forall \primal \in \PRIMAL
        \eqfinv
        \label{eq:argmin_ConditionalInfimum_argmin_original_fonctionprimal}
      \end{equation}
      a subset 
      \( \Primal  \subset \PRIMAL \) such that 
      \begin{equation}
        \Uncertain \subset \correspondence \Primal  
        \eqfinv
        \label{eq:argmin_ConditionalInfimum_argmin_original_subsets}
      \end{equation}
      and an optimal solution~\( \primal\opt \in \PRIMAL \) to the 
      auxiliary minimization problem
      \( \min_{ \primal \in \Primal } \fonctionprimal\np{\primal} \), 
      that is, 
      \begin{equation}
        \primal\opt
        \in 
        \argmin_{ {\primal \in \Primal} } \fonctionprimal\np{\primal}
        \eqfinv
        \label{eq:argmin_ConditionalInfimum_argmin_original_primal}
      \end{equation}
    \item 
      an element \( \uncertain\opt \in \UNCERTAIN \) such that 
      \begin{align}
        \fonctionuncertain\np{\uncertain\opt}
        &=
          \fonctionprimal\np{\primal\opt}
          \eqfinv
          \label{eq:argmin_ConditionalInfimum_argmin_original_=}
        \\
        \uncertain\opt
        &\in 
          \Uncertain
          \eqfinp
          \label{eq:argmin_ConditionalInfimum_argmin_original_in_Uncertain=}
      \end{align}
    \end{enumerate}
  \end{subequations}
  Then, $\uncertain\opt$ is an optimal solution to the original
  minimization problem~\eqref{eq:minimization_problem_Uncertain}, 
  that is, 
  \begin{equation}
    \uncertain\opt \in 
    \argmin_{ \uncertain \in\Uncertain } \fonctionuncertain\np{\uncertain}
    \eqfinp
    \label{eq:argmin_ConditionalInfimum=argmin_original}
  \end{equation}
\end{proposition}

\begin{proof}
  The equality~\eqref{eq:argmin_ConditionalInfimum=argmin_original}
  between solutions of minimization problems follows from
  \begin{align*}
    \fonctionuncertain\np{\uncertain\opt}
    &=
      \fonctionprimal\np{\primal\opt}
      \tag{by assumption~\eqref{eq:argmin_ConditionalInfimum_argmin_original_=}}
    \\
    &=
      \min_{ \primal \in \Primal } \fonctionprimal\np{\primal}
      \tag{by assumption~\eqref{eq:argmin_ConditionalInfimum_argmin_original_primal}}
    \\
    &\leq
      \inf_{ \primal \in \Primal } \InfCond{\fonctionuncertain}{\correspondence}\np{\primal}
      \tag{because \( \fonctionprimal \leq 
      \InfCond{\fonctionuncertain}{\correspondence}\)
      by assumption~\eqref{eq:argmin_ConditionalInfimum_argmin_original_fonctionprimal}}
    \\
    &=
      \inf_{ \uncertain \in \correspondence\Primal } \fonctionuncertain\np{\uncertain}
      \tag{by the equality~\eqref{eq:ConditionalInfimum_tower_property_equality_subset}}
    \\
    & \leq 
      \inf_{ \uncertain \in \Uncertain} \fonctionuncertain\np{\uncertain}
      \eqfinp
      \tag{because \( \correspondence \Primal \supset \Uncertain \)
      by assumption~\eqref{eq:argmin_ConditionalInfimum_argmin_original_subsets}}
    \\
  \end{align*}
  As \( \uncertain\opt \in \Uncertain \)
  by assumption~\eqref{eq:argmin_ConditionalInfimum_argmin_original_in_Uncertain=},
  this ends the proof.
\end{proof}

\subsubsection{Application to fractional programming}
\label{Application_to_fractional_programming}

Regarding fractional programming, we refer the reader to
\cite{Charnes-Cooper:1962,Schaible:1974} and \cite[\S2.6]{XiaYong:2020}.

Let ${\spacedim} \in \NN^*$ be a positive integer.  Let
$ \fonctiondepart \colon \RR^{\spacedim} \to \barRR $ and
$\fonctiondepartbis \colon \RR^{\spacedim} \to \RR$ be two functions and
$\Convex \subset \RR^{\spacedim} $ be a nonempty set.
When $\depart \in \Convex \implies \fonctiondepartbis(\depart) > 0 $, we consider
the following minimization problem
\begin{align}
  \inf_{\depart \in
  \Convex}\frac{\fonctiondepart(\depart)}{\fonctiondepartbis(\depart)}
  =
  \InfCond{\frac{\fonctiondepart}{\fonctiondepartbis}}{\Convex}
  \label{eq:ratio-pb}
  \eqfinp
\end{align}
We recall that the {perspective function}~$\perspective{\fonctiondepart} \colon
\RR_{++}\times\RR^{\spacedim} \to \barRR $ of the function~$\fonctiondepart$ has
been defined in~\eqref{eq:perspective_function}.

\begin{lemma}
  \label{le:infcond_frac_and_persp}
  Let $ \fonctiondepart \colon \RR^{\spacedim} \to \barRR $ and
  $\fonctiondepartbis \colon \RR^{\spacedim} \to \RR$ be two functions.
  Consider the mapping\footnote{This mapping is
    reminiscent of the generalized Charnes-Cooper-variable transformation
    C.C.T. in \cite[Sect.~1]{Schaible:1974}} \( \theta \colon \RR^{\spacedim} \to \RR{\times}\RR^{\spacedim}$ defined by
  ${ \theta(\depart) } = \bp{\theta_{1}(\depart),\theta_2(\depart)} = \bp{
    \frac{1}{\fonctiondepartbis(\depart)}
    ,
    \frac{\depart}{\fonctiondepartbis(\depart)}
  } \).
  For any nonempty set $\Convex\subset \nset{\depart\in \RR^{\spacedim}}{\fonctiondepartbis(\depart)> 0}$, we have that
  \begin{subequations}
    \begin{align}
      \label{eq:theta_C}
      \theta(\Convex)
      &=
        \bset{(\lambda,\arrivee)\in \RR_{++}{\times}\RR^{\spacedim}}{ \arrivee/\lambda \in \Convex
        \text{ and } \fonctiondepartbis(\arrivee/\lambda)=1/\lambda }
        \eqfinv
      \\        
      \label{eq:infconf_frac_theta}
      \InfCond{\frac{f}{\fonctiondepartbis}}{\theta}
      &=
        \perspective{\fonctiondepart} \UppPlus \Indicator{ \theta(\Convex)} 
        \eqfinv
      \\
      \label{eq:fractional_to_convex}
      \inf_{\depart\in\Convex}
      \frac{\fonctiondepart(\depart)}{\fonctiondepartbis(\depart)}
      &=
        \inf_{(\lambda,\arrivee) \in \theta(\Convex)}
        \perspective{\fonctiondepart}(\lambda,\arrivee)
        \eqfinp
    \end{align}
  \end{subequations}
\end{lemma}

\begin{proof}
  Using that $\Convex \subset \nset{\depart\in \RR^{\spacedim}}{\fonctiondepartbis(\depart)> 0}$,
  it is straightforward to obain that the mapping $\theta: \Convex \to \theta(\Convex)$ is an injection, hence a bijection.
  The inverse mapping $\Converse{\theta} \colon \theta(\Convex) \to \Convex$ is given by 
  $(\lambda,\arrivee) \mapsto \arrivee/\lambda$, and we easily get~\eqref{eq:theta_C}.
  \medskip

  Now, we prove~\eqref{eq:infconf_frac_theta}.      
  If $(\lambda,\arrivee) \not\in \theta(\Convex)$, it is immediate that
  $\InfCond{\frac{\fonctiondepart}{\fonctiondepartbis}}{\theta}(\lambda,\arrivee) = +\infty$
  by definition~\eqref{eq:mapping_ConditionalInfimum} of $\InfCond{\frac{\fonctiondepart}{\fonctiondepartbis}}{\theta}$.
  Now, for any $(\lambda,\arrivee) \in \theta(\Convex)$, we have
  \begin{align*} 
    \InfCond{\frac{\fonctiondepart}{\fonctiondepartbis}}{\theta}(\lambda,\arrivee)
    &=
      \InfCond{\frac{\fonctiondepart(\depart)}{\fonctiondepartbis(\depart)}}%
      { \theta_{1}(\depart)=\lambda, \theta_2(\depart)=\arrivee}
      \tag{by definition~\eqref{eq:mapping_ConditionalInfimum} of $\InfCond{\frac{\fonctiondepart}{\fonctiondepartbis}}{\theta}$}
    \\
    &=
      \InfCond{\frac{\perspective{\fonctiondepart}(1,\depart)}{\fonctiondepartbis(\depart)}}%
      { \theta_{1}(\depart)=\lambda, \theta_2(\depart)=\arrivee}
      \tag{by~\eqref{eq:perspective_function}}
    \\
    &=
      \InfCond{\frac{\perspective{\fonctiondepart}(1,\depart)}{\fonctiondepartbis(\depart)}}%
      { \frac{1}{\fonctiondepartbis(\depart)}=\lambda, \frac{\depart}{\fonctiondepartbis(\depart)}=\arrivee }
      \tag{by definition of the mapping~$\theta$}
    \\
    &=
      \InfCond{\perspective{\fonctiondepart}(\frac{1}{\fonctiondepartbis(\depart)},
      \frac{\depart}{\fonctiondepartbis(\depart)})}{
      { \frac{1}{\fonctiondepartbis(\depart)}=\lambda, \frac{\depart}{\fonctiondepartbis(\depart)}=\arrivee }}
      \intertext{as the perspective function~$\perspective{\fonctiondepart}$ is 1-homogeneous,
      and as \( \fonctiondepartbis(\depart) > 0 \) for \( \depart\in\Convex \) by assumption,}
    &=
      \perspective{\fonctiondepart}(\lambda,\arrivee)
      \eqfinv
  \end{align*}
  which gives Equation~\eqref{eq:infconf_frac_theta}.
  %

  \medskip

  Finally, we prove~\eqref{eq:fractional_to_convex}.
  We have that
  \begin{align*} 
    \inf_{(\lambda,\arrivee) \in \theta(\Convex)}  \perspective{\fonctiondepart}(\lambda,\arrivee) 
    &=
      \inf_{(\lambda,\arrivee) \in \theta(\Convex)} \InfCond{\frac{\fonctiondepart}{\fonctiondepartbis}}{\theta}(\lambda,\arrivee)
      \tag{by~\eqref{eq:infconf_frac_theta}}
    \\
    &=
      \InfCond{\InfCond{\frac{\fonctiondepart}{\fonctiondepartbis}}{\theta}}{\theta(\Convex)}
      \tag{by~\eqref{eq:subset_conditional_infimum}}
    \\
    &=
      \InfCond{\InfCond{\frac{\fonctiondepart}{\fonctiondepartbis}}{\graph_{\theta} }}{\theta(\Convex)}
      \tag{by~\eqref{eq:mapping_ConditionalInfimum}}
    \\
    &=
      \InfCond{ \frac{\fonctiondepart}{\fonctiondepartbis} }%
      {
      \graph_{\theta} \theta(\Convex)
      }
      \tag{by the pushforward property~\eqref{eq:correspondence_conditional_infimum_properties_pushforward}}
    \\
    &=
      \InfCond{ \frac{\fonctiondepart}{\fonctiondepartbis} }%
      {
      \Converse{\theta}(\theta(\Convex))
      }
      \tag{by~\eqref{eq:graph}}
    \\
    &=
      \InfCond{ \frac{\fonctiondepart}{\fonctiondepartbis} }%
      {\Convex}
      \intertext{as we have shown at the beginning that $\theta: \Convex \to \theta(\Convex)$ is a bijection}
    &= { \inf_{\depart\in\Convex}
      \frac{\fonctiondepart(\depart)}{\fonctiondepartbis(\depart)} }
      \tag{by~\eqref{eq:subset_conditional_infimum}}
      \eqfinv
  \end{align*}
  which gives Equation~\eqref{eq:fractional_to_convex}, and ends the proof.
\end{proof}
  
According to the previous Lemma~\ref{le:infcond_frac_and_persp}, if the function
$\fonctiondepart$ is convex and if the set $\theta(\Convex)$ is convex then the
fractional minimization problem is equivalent to a convex problem as given by
Equation~\eqref{eq:fractional_to_convex}.  The special case where
$\Convex = \nset{\depart \in \RR^{\spacedim}}{ A x \le b }$ and
$\fonctiondepart$ and $\fonctiondepartbis$ are affine functions easily satisfies
the above assumptions: it is known in the literature under the name of
\emph{fractional linear programming} \cite{Charnes-Cooper:1962}.  In the
equivalent formulation, the function~$\perspective{\fonctiondepart}$ is linear
and the equivalent minimization problem is a linear optimization problem.

As in~\cite{Schaible:1974}, we go beyond linear programming and address
fractional \emph{convex} programming in the next
Proposition~\ref{pr:convex_frac}.  However, to the difference of
\cite[Proposition~3]{Schaible:1974}, we treat general convex constraints (not
necessarily given by inequalities).
  
\begin{proposition}
  \label{pr:convex_frac}
  Let $ \fonctiondepart \colon \RR^{\spacedim} \to \barRR $ and
  $\fonctiondepartbis \colon \RR^{\spacedim} \to \RR$ be two convex functions.
  Let $\Convex \subset \RR^{\spacedim} $ be a nonempty convex set such that
  $\depart \in \Convex \implies \fonctiondepartbis(\depart) > 0 $ and
  \( \fonctiondepart(\depart) \leq 0 \).
  Then, the perspective function~$\perspective{\fonctiondepart}$ is convex,
  the set
  \begin{subequations}
    \begin{equation}
      \label{eq:perspective_Convex}
      \perspective{\Convex} = \bset{(\lambda,\arrivee)\in \RR_{++}{\times}\RR^{\spacedim}}%
      { \arrivee/\lambda \in \Convex
        \text{ and } \fonctiondepartbis(\arrivee/\lambda) \leq 1/\lambda}
    \end{equation}
    is convex, and we have that
    \begin{equation}
      \inf_{\depart\in\Convex}
      \frac{\fonctiondepart(\depart)}{\fonctiondepartbis(\depart)}
      =
      \overbrace{
        \inf_{(\lambda,\arrivee) \in \perspective{\Convex}}
        \perspective{\fonctiondepart}(\lambda,\arrivee)
      }^{\textrm{convex program}}
      \eqfinp
    \end{equation}  
  \end{subequations}
\end{proposition}

\begin{proof}
  By definition~\eqref{eq:perspective_function} of the perspective function, 
  the set~\( \perspective{\Convex} \) in~\eqref{eq:perspective_Convex}
  can be expressed as
  \begin{equation*}
    \perspective{\Convex}=\bset{(\lambda,\arrivee)\in \RR_{++}{\times}\RR^{\spacedim}}%
    { \perspective{\Indicator{\Convex}}(\lambda,\arrivee) \le 0
      \text{ and } \perspective{\fonctiondepartbis}(\lambda,\arrivee) \le 1}
    \eqfinp 
  \end{equation*}
  Under this form, it is immediate to prove that the set~$\perspective{\Convex}$ is convex as
  $\Indicator{\Convex}$ and $\fonctiondepartbis$ are both
  convex functions and their perspective functions, \(
  \perspective{\Indicator{\Convex}} \) and \( \perspective{\fonctiondepartbis}
  \), are therefore also convex ---
  as is well known (see the relationship between epigraphs in~\eqref{eq:StrictEpigraph_perspective_function}). 

  Moreover, comparing Equation~\eqref{eq:theta_C} with Equation~\eqref{eq:perspective_Convex},
  we get the inclusion
  \begin{align*}
    \theta(\Convex)
    &=
      \bset{(\lambda,\arrivee)\in \RR_{++}{\times}\RR^{\spacedim}}%
      { \arrivee/\lambda \in \Convex
      \text{ and } \fonctiondepartbis(\arrivee/\lambda) = 1/\lambda}\\
    \\
    &\subset
      \bset{(\lambda,\arrivee)\in \RR_{++}{\times}\RR^{\spacedim}}%
      { \arrivee/\lambda \in \Convex
      \text{ and } \fonctiondepartbis(\arrivee/\lambda) \leq 1/\lambda} 
      = \perspective{\Convex}
      \eqfinp             
  \end{align*}
  Thus, using~\eqref{eq:fractional_to_convex}, we get the equality and inequality
  \begin{equation}
    \label{ineq:frac}
    \inf_{\depart\in\Convex}
    \frac{\fonctiondepart(\depart)}{\fonctiondepartbis(\depart)}
    =
    \inf_{(\lambda,\arrivee) \in \theta(\Convex)}
    \perspective{\fonctiondepart}(\lambda,\arrivee)
    \geq
    \inf_{(\lambda,\arrivee) \in \perspective{\Convex}}
    \perspective{\fonctiondepart}(\lambda,\arrivee)
    \eqfinp
  \end{equation}
  Using the assumption that the function~$\fonctiondepart$ is nonpositive on~$\Convex$, we are going
  to show that we have an equality in Equation~\eqref{ineq:frac}.
  Indeed, consider $(\lambda,\arrivee)\in \perspective{\Convex}$.
  First, we set
  \( \tau =\perspective{\fonctiondepartbis}(\lambda,\arrivee) = 
  \lambda\fonctiondepartbis(\arrivee/\lambda) \), which is such 
  that $\tau \in \OpenIntervalClosed{0}{1}$,
  by definition~\eqref{eq:perspective_Convex} of~\( \perspective{\Convex} \).
  Second, we set 
  \( (\lambda',\arrivee')= (\lambda/\tau,\arrivee/\tau) \), which is such 
  that $(\lambda',\arrivee')= (\lambda/\tau,\arrivee/\tau) \in \theta(\Convex)$
  by~\eqref{eq:theta_C} as, on the one hand,
  \( \arrivee'/\lambda'= \arrivee/\lambda \in \Convex \) and, on the other hand,
  \( \fonctiondepartbis(\arrivee'/\lambda') = \fonctiondepartbis(\arrivee/\lambda) 
  = \tau/\lambda= 1/\lambda' \).
  Third, as $\tau \in \OpenIntervalClosed{0}{1}$ and $\fonctiondepart$ is nonpositive, we get that
  \begin{align*}
    \perspective{\fonctiondepart}\underbrace{(\lambda,\arrivee)}_{\in \perspective{\Convex}}
    = \perspective{\fonctiondepart}(\tau\lambda',\tau\arrivee')
    = \tau \lambda'\perspective{\fonctiondepart}(1,\arrivee'/\lambda')
    = \tau \lambda'\fonctiondepart(\arrivee'/\lambda')
    \ge  \lambda'\fonctiondepart(\arrivee'/\lambda')
    = \perspective{\fonctiondepart}\underbrace{(\lambda',\arrivee')}_{\in \theta(\Convex)}
    \eqfinp
    \label{eq:vplp}
  \end{align*}
  Thus, we have shown that, for any $(\lambda,\arrivee)\in
  \perspective{\Convex}$, there exists $(\lambda',\arrivee')\in
  \theta(\Convex)$
  such that \( \perspective{\fonctiondepart}(\lambda,\arrivee) \geq
  \perspective{\fonctiondepart}(\lambda',\arrivee') \). From this, we 
  finally deduce that \( \inf_{(\lambda,\arrivee) \in \perspective{\Convex}}
  \perspective{\fonctiondepart}(\lambda,\arrivee) \geq
  \inf_{(\lambda,\arrivee) \in \theta(\Convex)}
  \perspective{\fonctiondepart}(\lambda,\arrivee) \), hence 
  an equality in Equation~\eqref{ineq:frac}.
\end{proof}
  
\subsubsection{Application to lower bound convex programs}
\label{Application_to_lower_bound_convex_programs}

\subsubsubsection{Lower bound convex programs in sparse optimization}

We can use the conditional infimum in
Proposition~\ref{pr:inf_ConditionalInfimum_inf_original} to obtain lower bound
convex programs for nonconvex problems as follows.  Consider, on the one hand, a
function \( \fonctionuncertain \colon \UNCERTAIN \to \barRR \)
and, on the other hand, a correspondence~$\correspondence$ between~$\UNCERTAIN$
and $\PRIMAL$, and a subset \( \Primal \subset \PRIMAL \).
By~\eqref{eq:ConditionalInfimum_tower_property_equality_subset} followed by
well-known inequalities in convex analysis, we have that
\begin{subequations}
  \begin{equation}
    \underbrace{ \inf_{\primal \in \closedconvexhull{\Primal}}
      \LFMbi{\bp{\nInfCond{\fonctionuncertain}{\correspondence}}}\np{\primal} }%
    _{\textrm{lower bound convex program}}
    \leq
    \inf_{\primal \in \Primal}
    \bp{\nInfCond{\fonctionuncertain}{\correspondence}\np{\primal}}
    =
    \inf_{\uncertain \in \correspondence\Primal} \fonctionuncertain\np{\uncertain} 
    \eqfinp      
  \end{equation}
  This may be interesting when \( \LFMbi{\fonctionuncertain} \) is trivial, but
  \( \LFMbi{\bp{\nInfCond{\fonctionuncertain}{\correspondence}}} \) is not, like
  when \( \fonctionuncertain \) is the \lzeropseudonorm\ on~$\RR^{\spacedim}$
  and $\correspondence$ is given by the normalization mapping onto the Euclidean
  sphere \cite[Sect.~3]{Chancelier-DeLara:2021_ECAPRA_JCA}.  In that case,
  \( \LFMbi{\bp{\nInfCond{\fonctionuncertain}{\correspondence}}} \) is a proper
  convex \lsc\ function~ \( {\cal L}_{0}\) defined and studied in
  \cite[\S~4.1]{Chancelier-DeLara:2021_ECAPRA_JCA}, giving
  \begin{equation}
    \underbrace{ \inf_{\primal \in \closedconvexhull{\Primal}}
      {\cal L}_{0}\np{\primal} 
    }%
    _{\textrm{lower bound convex program}}
    \leq
    \inf_{\uncertain \in \RR_{++}\Primal} \lzero\np{\uncertain} 
    \eqfinp      
  \end{equation}
\end{subequations}

\subsubsubsection{Lower bound convex programs with \( \Cone \)-convexity}

\begin{proposition}
  Let~$\DEPART$ and~$\ARRIVEE$ be two real vector spaces.
  Let \( \Cone\subset\ARRIVEE \) be a convex cone containing zero,
  and  \( \theta \colon \DEPART \to \ARRIVEE \) be a \( \Cone \)-convex mapping.   
  Let $ \FONCTIONDEPART \colon \DEPART\times\ARRIVEE \to \barRR $ be a convex function.
  Then, we have that
  \begin{equation}
    \underbrace{ \inf_{\np{\depart,\arrivee} \in \Epigraph_{\Cone}\theta}
      \FONCTIONDEPART\np{\depart,\arrivee} }_{\textrm{convex program}}
    \leq 
    \inf_{\depart \in \DEPART} \FONCTIONDEPART\bp{\depart,\theta\np{\depart}}
    \eqfinv
    \label{eq:lower_bound_convex_programs_with_Cone-convexity}
  \end{equation}
  where the right hand side minimization problem is generally not a convex
  program\footnote{%
It is a convex program if the mapping~\( \theta \) is linear, or under the assumptions
  of Item~\ref{it:composition_convexity} in Corollary~\ref{cor:composition_convexity}.}.
\end{proposition}

\begin{proof}
  We have that
  \begin{align*}
    \inf_{\depart \in \DEPART }\FONCTIONDEPART\bp{\depart,\theta\np{\depart}}
    &=
      \inf_{\depart \in \graph_{\np{\textrm{Id}_{\DEPART},\theta}}\DEPART\times\ARRIVEE }
      \FONCTIONDEPART\bp{\depart,\theta\np{\depart}}
      \tag{since \( \DEPART=\graph_{\np{\textrm{Id}_{\DEPART},\theta}}\np{\DEPART\times\ARRIVEE} \)}
    \\
    &=
      \inf_{\np{\depart,\arrivee} \in \DEPART\times\ARRIVEE }
      \Bp{\nInfCond{\FONCTIONDEPART\bp{\cdot,\theta\np{\cdot}}}{\np{\textrm{Id}_{\DEPART},\theta}}}\np{\depart,\arrivee} 
      \tag{by~\eqref{eq:ConditionalInfimum_tower_property_equality_subset}}
    \\
    &=
      \inf_{\np{\depart,\arrivee} \in \DEPART\times\ARRIVEE }
      \FONCTIONDEPART\np{\depart,\arrivee} + \Indicator{\graph_{\theta}}\np{\depart,\arrivee}
      \intertext{by definition~\eqref{eq:mapping_ConditionalInfimum} of the
      conditional infimum \wrt\ a mapping \( \theta \colon \DEPART \to \ARRIVEE \),       
      and by definition~\eqref{eq:graph} of~\( \graph_{\theta} \)}
    &\geq
      \inf_{\np{\depart,\arrivee} \in \DEPART\times\ARRIVEE }
      \FONCTIONDEPART\np{\depart,\arrivee} + \Indicator{\graph_{\theta}+\Cone}\np{\depart,\arrivee}
      \tag{as \( 0\in\Cone\) by assumption, hence \( \graph_{\theta} \subset \graph_{\theta}+\Cone \)}
    \\
    &=
      \inf_{\np{\depart,\arrivee} \in \Epigraph_{\Cone}\theta}
      \FONCTIONDEPART\np{\depart,\arrivee}
      \tag{as \( \Epigraph_{\Cone}\theta = \graph_{\Theta} + \na{0}{\times}\Cone \) by~\eqref{eq:Cone-epigraph}}
      \eqfinp 
  \end{align*}
\end{proof}

With this approach, we can revisit how hidden convexity is obtained
in~\cite{Ben-Tal-den-Hertog-Laurent:2011}.  We consider \( \DEPART=\RR^{n}\),
\( \ARRIVEE=\RR^{N} \), \( \Cone=\RR_{+}^{N} \),
\( \theta=\np{f_{1},\cdots,f_{N}} \),
\( \FONCTIONDEPART=F_{0}+ \sum_{l\in L}\Indicator{\na{F_{l}\leq \rho_{l}}} \), where all
functions \( f_{1},\cdots,f_{N} \colon \RR^{n} \to \RR\), \( F_{0} \) and
\( F_{l} \colon \RR^{N} \to\RR \), \( l\in L \), are convex.  The inequality between
(NCP) and (CP) in~\cite[Section~1]{Ben-Tal-den-Hertog-Laurent:2011} is exactly
the inequality~\eqref{eq:lower_bound_convex_programs_with_Cone-convexity}
between
\( \inf_{\np{\depart,\arrivee} \in \Epigraph_{\Cone}\theta}
\FONCTIONDEPART\np{\depart,\arrivee} \) and
\( \inf_{\depart \in \DEPART} \FONCTIONDEPART\bp{\depart,\theta\np{\depart}} \).  Then,
\cite[Section~2]{Ben-Tal-den-Hertog-Laurent:2011} provides suitable assumptions
regarding the functions \( f_{1},\cdots,f_{N} \), \( F_{0} \) and \( F_{l}\),
\( l\in L \), under which a solution of (CP) is one of~(NCP), hence under which
the inequality~\eqref{eq:lower_bound_convex_programs_with_Cone-convexity} is an
equality.

\section{Conditional infimum and hidden convexity in quadratic optimization}
\label{Hidden_convexity_in_quadratic_optimization_problems}

Let ${\spacedim} \in \NN^*$ be a positive integer.  A \emph{quadratic function} is
a polynomial of degree~2 on~\( \RR^{\spacedim} \), whereas a \emph{quadratic
  form} is a quadratic function without linear and constant terms.
In this Sect.~\ref{Hidden_convexity_in_quadratic_optimization_problems}, we
contribute to the analysis of minimization problems of the form
\begin{equation}
  \min\bset{ \textrm{quadratic function} }%
  { \text{ quadratic functions belong to a convex set} }
  \eqfinv
  \label{eq:LQ}
\end{equation}
and to identify conditions under which such problems are equivalent to convex
minimization problems (hidden convexity).  


In~\S\ref{Conditional_infimum_of_quadratic_functions}, we provide necessary and
sufficient conditions under which the conditional infimum of a quadratic
function, \wrt\ the square mapping, is convex.
In~\S\ref{Hidden_convexity_in_quadratic_minimization_problems}, we deduce a new
sufficient condition for hidden convexity in quadratic minimization problems,
which covers, and goes beyond, existing conditions, as revealed by comparison
with the literature.

\subsection{Conditional infimum of quadratic functions}
\label{Conditional_infimum_of_quadratic_functions}

We define the \emph{square mapping}
\( \SquareMapping \colon \RR^{\spacedim} \to \RR^{\spacedim} \) by
\begin{equation}
  \SquareMapping\np{\uncertain}=\SquareMapping\np{\uncertain_{1},\ldots,\uncertain_{\spacedim}}
  = \np{\uncertain_{1}^2,\ldots,\uncertain_{\spacedim}^2}
  \eqsepv \forall \uncertain \in \RR^{\spacedim}
  \eqfinp
  \label{eq:SquareMapping}
\end{equation}
We consider a vector~$\vecteur \in \RR^{\spacedim}$, a
$\spacedim{\times}\spacedim$ symmetric matrix~$\matrice$,
and a \emph{quadratic function} \( \LinearQuadratic \colon \RR^{\spacedim} \to \barRR \) given
by (where~$\transp$ denotes transposition)
\begin{equation}
  \LinearQuadratic\np{\uncertain}=
  \uncertain\transp\matrice\uncertain + \uncertain\transp\vecteur
  \eqsepv \forall  \uncertain \in \RR^{\spacedim}
  \eqfinp
  \label{eq:LinearQuadraticFonctionuncertain}
\end{equation}
The function \( \ConditionalInfimum{\SquareMapping}{\LinearQuadratic} \colon \RR^{\spacedim}
\to \barRR \), defined in~\eqref{eq:mapping_ConditionalInfimum}, is given by
\begin{equation}
  \begin{split}
    \ConditionalInfimum{\SquareMapping}{\LinearQuadratic}\np{\primal_{1},\ldots,\primal_{\spacedim}}
    =\inf\bset{ \uncertain\transp\matrice\uncertain + \uncertain\transp\vecteur }%
    { \uncertain_{1}^2=\primal_{1},\ldots,\uncertain_{\spacedim}^2=\primal_{\spacedim} }
    \eqsepv       \\
    \forall \primal=\np{\primal_{1},\ldots,\primal_{\spacedim}} \in \RR^{\spacedim}
    \eqfinp
  \end{split}
  \label{eq:LinearQuadraticFonctionuncertain_wrt_SquareMapping}
\end{equation}

In~\S\ref{Convex_lower_bound_for_the_conditional_infimum_of_a_quadratic_function},
we provide a convex lower bound for the conditional infimum of the quadratic
function~\eqref{eq:LinearQuadraticFonctionuncertain} \wrt\ the square mapping~\eqref{eq:SquareMapping}. 
In~\S\ref{Block-signed_form_of_a_pair_matrix-vector}, we define the notion of
{block-signed form of a pair matrix-vector}.
In~\S\ref{Equivalence_between_block-signed_form_and_convex_conditional_infimum},
we prove the equivalence between block-signed form of a pair matrix-vector and convexity of the
conditional infimum of the corresponding quadratic function.

\subsubsection{Convex lower bound for the conditional infimum of a quadratic
  function}
\label{Convex_lower_bound_for_the_conditional_infimum_of_a_quadratic_function}

\begin{proposition}
  \label{pr:ConditionalInfimum_quadratic_convex_LB_inequality}
    Let 
  $\vecteur \in \RR^{\spacedim}$ be a vector, and $\matrice$ be a
  $\spacedim{\times}\spacedim$ symmetric matrix.
  Let the quadratic
  function \( \LinearQuadratic \colon \RR^{\spacedim} \to \barRR \) be given
  by~\eqref{eq:LinearQuadraticFonctionuncertain}.

  Then, the function~\( \widetilde{\LinearQuadratic} \colon \RR^{\spacedim} \to \RR \cup
  \na{+\infty} \), defined by
  \begin{equation}
      \widetilde{\LinearQuadratic}\np{\primal_{1},\ldots,\primal_{\spacedim}} =
    \begin{cases}
      +\infty & \text{ if } \np{\primal_{1},\ldots,\primal_{\spacedim}} \not\in \RR_+^{\spacedim}
           \eqfinv
      \\
      \displaystyle
      \sum_{i=1}^{\spacedim} \matrice_{i,i} \primal_i
      -
      \sum_{i\neq j} \module{\matrice_{i,j}}\sqrt{\primal_i \primal_j}
      -
      \sum_{i=1}^{\spacedim} \module{\vecteur_i} \sqrt{\primal_i }
         & \text{ if } \np{\primal_{1},\ldots,\primal_{\spacedim}} \in \RR_+^{\spacedim} 
           \eqfinv
    \end{cases}
    \label{eq:ConditionalInfimum_quadratic_convex_LB}
  \end{equation}
  is proper convex \lsc\  with effective domain \(
  \domain\widetilde{\LinearQuadratic} =\RR_+^{\spacedim}\), and we have that
  \begin{equation}
    \ConditionalInfimum{\SquareMapping}{\LinearQuadratic} \geq \widetilde{\LinearQuadratic}
    \eqfinp
    \label{eq:ConditionalInfimum_quadratic_convex_LB_inequality}
  \end{equation}
\end{proposition}

  \begin{proof}
        As \( \SquareMapping\np{\RR^{\spacedim}}= \RR_+^{\spacedim}\) by
    definition~\eqref{eq:SquareMapping} of the square mapping, we have
    \(  \bp{\ConditionalInfimum{\SquareMapping}{\LinearQuadratic}}\np{\primal}
    =+\infty \) for any \( \primal \not\in \RR_+^{\spacedim}\), by~\eqref{eq:dom_ConditionalInfimum_subset_range}.
    Then, we have, for any \( \primal=\np{\primal_{1},\ldots,\primal_{\spacedim}}
    \in \RR_+^{\spacedim} \),
    \begin{align}
         \bp{\ConditionalInfimum{
        \SquareMapping}{\LinearQuadratic}}
        \np{\primal_{1},\ldots,\primal_{\spacedim}}
       &=
        \inf\bset{\uncertain\transp\matrice\uncertain + \uncertain\transp\vecteur }{%
        \uncertain\in\RR^{\spacedim} \eqsepv
        \np{\uncertain_{1}^2,\ldots,\uncertain_{\spacedim}^2} = \np{\primal_{1},\ldots,\primal_{\spacedim}} }
        \nonumber
        \intertext{by definition~\eqref{eq:mapping_ConditionalInfimum} of the conditional
        infimum \wrt\ a mapping, and by definition~\eqref{eq:SquareMapping} of the square mapping}
      &=
        \inf\Bset{ \sum_{i=1}^{\spacedim} \matrice_{i,i} \uncertain_i^2 +
        \sum_{i\neq j} \matrice_{i,j} \uncertain_i\uncertain_j
        + \sum_{i=1}^{\spacedim} \vecteur_i \uncertain_i }{%
        \uncertain_{1}^2=\primal_{1},\ldots,\uncertain_{\spacedim}^2=\primal_{\spacedim}}
        \nonumber
      \\
      &=
        \sum_{i=1}^{\spacedim} \matrice_{i,i} \primal_i +
        \inf\Bset{
        \sum_{i\neq j} \matrice_{i,j} \uncertain_i\uncertain_j
        + \sum_{i=1}^{\spacedim} \vecteur_i \uncertain_i }{%
        \uncertain_{1}=\pm\sqrt{\primal_{1}},\ldots,\uncertain_{\spacedim}=\pm\sqrt{\primal_{\spacedim}} }
        \nonumber
      \\
      &=
        \sum_{i=1}^{\spacedim} \matrice_{i,i} \primal_i +
        \min\Bset{
        \sum_{i\neq j} \matrice_{i,j}\varepsilon'_i\varepsilon'_j \sqrt{\primal_i\primal_j}
        + \sum_{i=1}^{\spacedim} \vecteur_i\varepsilon'_i \sqrt{\primal_i} }{%
        \varepsilon' \in \na{-1,1}^{\spacedim} }
        \label{eq:ConditionalInfimum_quadratic-example_varepsilon'_inproof}
      \\
      & \geq
        \sum_{i=1}^{\spacedim} \matrice_{i,i} \primal_i +
        \sum_{i\neq j} \module{\matrice_{i,j}} \np{-\sqrt{\primal_i \primal_j} }
        + \sum_{i=1}^{\spacedim} \module{\vecteur_i} \np{-\sqrt{\primal_i }}
        \eqfinv
        \label{eq:ConditionalInfimum_quadratic-example_inproof}
    \end{align}
    by taking term by term inequality, and
    where we recognize, in this last expression~\eqref{eq:ConditionalInfimum_quadratic-example_inproof},
    the expression~\eqref{eq:ConditionalInfimum_quadratic_convex_LB} of the
    function~\( \widetilde{\LinearQuadratic}  \).
    This gives the inequality~\eqref{eq:ConditionalInfimum_quadratic_convex_LB_inequality}.

    Now, it is easily checked (by computing the Hessian) that the functions
    \( \np{\primal_i,\primal_j} \in \RR_+^2 \mapsto \np{-\sqrt{\primal_i
        \primal_j} } \) are convex, for all $i\neq j$.
    Therefore, it is easily deduced that the function \( \widetilde{\LinearQuadratic} \colon \RR^{\spacedim} \to \barRR \)
    in~\eqref{eq:ConditionalInfimum_quadratic_convex_LB} is convex \lsc\
    with effective domain \( \domain\widetilde{\LinearQuadratic} =\RR_+^{\spacedim}\),
    hence is proper convex \lsc. 
  \end{proof}

\subsubsection{Block-signed form of a pair matrix-vector}
 \label{Block-signed_form_of_a_pair_matrix-vector}

As a second step towards tackling problems of the form~\eqref{eq:LQ}, we will
provide a necessary and sufficient condition under which the conditional infimum~\eqref{eq:LinearQuadraticFonctionuncertain_wrt_SquareMapping}
of a quadratic function~\eqref{eq:LinearQuadraticFonctionuncertain}, \wrt\ the square
mapping~\eqref{eq:SquareMapping}, is convex.
For this purpose, we define the notion of \emph{block-signed form of a pair matrix-vector}.

The function~$\sign{\cdot}$ returns the sign$~\sign{z} \in\na{-1,0,1}$ of a real
number~$z$ and, by extension, the sign of a vector (resp. of a matrix) is the
vector (resp. the matrix) whose entries are the signs of the original vector
(resp. matrix) entries.

\begin{definition}
  \label{de:block-signed form}
  Let 
  $\vecteur \in \RR^{\spacedim}$ be a vector,
  and $\matrice$ be a $\spacedim{\times}\spacedim$ symmetric matrix.
  %
  We say that \emph{the pair \( \np{\matrice,\vecteur} \)}
  can be put in \emph{block-signed form} or, equivalently, that the
quadratic function~\eqref{eq:LinearQuadraticFonctionuncertain}
  can be put in \emph{block-signed form}
  if there exists two (possibly empty) subsets~\( P \) (plus), $M$ (minus) \( \subset \ic{1,\spacedim} \) such that
  \begin{subequations}
    \begin{align}
      &
        {P \cap M = \emptyset \eqsepv P \cup M = \ic{1,\spacedim} }
      \\
      \sign{\vecteur_{P}}
      &\in \na{+1,0}^{P}
      \\
      \sign{\vecteur_M}
      &\in \na{-1,0}^M
      \\
      \sign{\matrice_{PP}}
      &\in \na{-1,0}^{P\times P} \qquad\text{ outside of the diagonal}
      \\
      \sign{\matrice_{MM}}
      &\in \na{-1,0}^{M\times M} \qquad\text{ outside of the diagonal}
      \\
      \sign{\matrice_{PM}}
      &\in \na{+1,0}^{P\times M}
      \\
      \sign{\matrice_{MP}}
      &\in \na{+1,0}^{M\times P}
    \end{align}
    in which case we set \( \varepsilon \in \na{-1,+1}^\spacedim \) defined by
    \begin{equation}
      \varepsilon_i=-1 \eqsepv \forall i\in P \text{ and }
      \varepsilon_i=+1 \eqsepv \forall i\in M
      \eqfinp
    \end{equation}
    \label{eq:block-signed_form_math}
  \end{subequations}
  Thus, after rearranging the indices, the pair {\( \np{\matrice,\vecteur} \)}
  is in {block-signed form}~\eqref{eq:block-signed_form}
  \begin{equation}
    \matrice=
    \left(
      \begin{array}{c|c}
        {
        \begin{array}{ccc}
          \matrice_{1,1}&& \matrice_{i,j}\leq 0 \\
                        &\ddots& \\
          \matrice_{i,j}\leq 0 && \matrice_{\cardinality{P},\cardinality{P}}
        \end{array}
        }
        &
          \matrice_{i,j}\geq 0
        \\ \hline
        \matrice_{i,j}\geq 0
                        &   \begin{array}{ccc}
                          \matrice_{\cardinality{P}+1,\cardinality{P}+1}
                          && \matrice_{i,j}\leq 0 \\ &\ddots& \\ \matrice_{i,j}\leq 0 && \matrice_{\spacedim,\spacedim}
                        \end{array}
      \end{array}
    \right)
    \eqsepv
    \vecteur=
    \left(
      \begin{array}{c}
        \vecteur_{i}\geq 0 \\ \hline \vecteur_{i}\leq 0
      \end{array} \right)
    \eqfinp
    \label{eq:block-signed_form}
  \end{equation}
  In the two extreme cases \( P=\emptyset\) or \( M=\emptyset\),
  the matrix~\( \matrice \) in Equation~\eqref{eq:block-signed_form} has nonpositive elements
  outside of the diagonal.
\end{definition}
Notice that, if the square matrix~\( \matrice\) is diagonal, the pair
\( \np{\matrice,\vecteur} \) can be put in {block-signed} form by simply
rearranging the entries of the vector~\( \vecteur \) in nonnegative values
followed by nonpositive ones.
Also, for any vector~\( \vecteur \) and scalar~\(\lambda\geq 0\), the pair
\( \np{\vecteur\vecteur',\lambda\vecteur} \) can be put in {block-signed} form.

\subsubsection{Equivalence between block-signed form and convex conditional
  infimum}
\label{Equivalence_between_block-signed_form_and_convex_conditional_infimum}

The following equivalence will be our main tool in deriving sufficient
conditions for hidden convexity in quadratic minimization problems.

\begin{theorem}
  Let 
  $\vecteur \in \RR^{\spacedim}$ be a vector, and $\matrice$ be a
  $\spacedim{\times}\spacedim$ symmetric matrix.
  Let the quadratic
  function \( \LinearQuadratic \colon \RR^{\spacedim} \to \barRR \) be given
  by~\eqref{eq:LinearQuadraticFonctionuncertain}.
  The following assertions are equivalent.
  \begin{enumerate}
  \item
    \label{it:ConditionalInfimum_SquareMapping_LinearQuadratic_convex}
The conditional infimum~\eqref{eq:LinearQuadraticFonctionuncertain_wrt_SquareMapping}
of the quadratic function~\eqref{eq:LinearQuadraticFonctionuncertain}, \wrt\ the square
mapping~\eqref{eq:SquareMapping}, is convex.
  \item
    \label{it:block-signed_form}
    The pair \( \np{\matrice,\vecteur} \) can be put in {block-signed} form
    as in Definition~\ref{de:block-signed form} 
    (or, equivalently, the quadratic function \( \LinearQuadratic \colon \RR^{\spacedim} \to \barRR \) given
    by~\eqref{eq:LinearQuadraticFonctionuncertain} can be put in {block-signed} form).
  \end{enumerate}

    In these two equivalent cases, the function~\( \ConditionalInfimum{\SquareMapping}{\LinearQuadratic}\)
  in~\eqref{eq:LinearQuadraticFonctionuncertain_wrt_SquareMapping}
  is proper convex \lsc\
  with effective domain \( \domain\ConditionalInfimum{\SquareMapping}{\LinearQuadratic} =\RR_+^{\spacedim}\),
  and has the expression~\eqref{eq:ConditionalInfimum_quadratic_convex_LB}.
  \label{th:Conditional_infimum_of_a_quadratic_function_knowing_squares}
\end{theorem}

\begin{proof}
  \begin{subequations}
    %


    \noindent$\bullet$ Suppose that the pair
    \( \np{\matrice,\vecteur} \) can be put in {block-signed} form or,
    equivalently by Proposition~\ref{pr:varepsilon_sign_appendix} (in
    Appendix~\ref{Appendix}), suppose that there exists
    \( \varepsilon=\np{\varepsilon_{1},\ldots,\varepsilon_{\spacedim}} \in \na{-1,1}^{\spacedim} \) such
    that~\eqref{eq:varepsilon} below holds true:
    \begin{equation}
       \begin{cases}
        \varepsilon_i\vecteur_i \leq 0 \eqsepv & \forall i=1,\ldots,\spacedim
        \eqfinv
        \\
        \text{and} &
        \\
        \varepsilon_i\varepsilon_j\matrice_{i,j}  \leq 0 \eqsepv
        &
        \forall i,j=1,\ldots,\spacedim \eqsepv i\neq j
        \eqfinp
      \end{cases}
      \label{eq:varepsilon}
    \end{equation}
    Then, it is easy to check that \( \varepsilon \in \na{-1,1}^{\spacedim}  \)
    satisfying~\eqref{eq:varepsilon}
    provides an equality in the inequality
    between~\eqref{eq:ConditionalInfimum_quadratic-example_varepsilon'_inproof}
    and~\eqref{eq:ConditionalInfimum_quadratic-example_inproof}.
    Thus, we get that the function
    \( \ConditionalInfimum{\SquareMapping}{\LinearQuadratic} \)
    is the function \( \widetilde{\LinearQuadratic} \)
    given by~\eqref{eq:ConditionalInfimum_quadratic_convex_LB}.
    Then, we apply Proposition~\ref{pr:ConditionalInfimum_quadratic_convex_LB_inequality}.
     \medskip

    \noindent$\bullet$
    Suppose that the function
    \( \fonctionprimalbis=\ConditionalInfimum{\SquareMapping}{\LinearQuadratic} \colon \RR^{\spacedim} \to \barRR \)
    is convex.
    For any \( \varepsilon \in \na{-1,1}^{\spacedim} \),
    the following subset \( \Primal_{\varepsilon} \) of \( ]0,+\infty[^{\spacedim} \) is
    closed (as easily follows from its second expression below)
    \begin{align*}
      \Primal_{\varepsilon}
      =&
         \Bset{ \primal\in ]0,+\infty[^{\spacedim} }{ \varepsilon \in \argmin\defset{
         \sum_{i\neq j} \matrice_{i,j}\varepsilon'_i\varepsilon'_j \sqrt{\primal_i\primal_j}
         + \sum_{i=1}^{\spacedim} \vecteur_i\varepsilon'_i \sqrt{\primal_i} }{%
         \varepsilon' \in \na{-1,1}^{\spacedim} \} } }
      \\
      =&
         \bigg\{ \primal\in ]0,+\infty[^{\spacedim}
         \; \bigg\vert \;
         \sum_{i\neq j} \matrice_{i,j}\varepsilon_i\varepsilon_j \sqrt{\primal_i\primal_j}
         + \sum_{i=1}^{\spacedim} \vecteur_i\varepsilon_i \sqrt{\primal_i}
      \\
       & \qquad\qquad\qquad\qquad \leq
         \sum_{i\neq j} \matrice_{i,j}\varepsilon'_i\varepsilon'_j \sqrt{\primal_i\primal_j}
         + \sum_{i=1}^{\spacedim} \vecteur_i\varepsilon'_i \sqrt{\primal_i} \eqsepv
         \forall \varepsilon' \in \na{-1,1}^{\spacedim} \bigg\}
         \eqfinp
    \end{align*}
    We are going to show that one of the subsets \( \Primal_{\varepsilon} \), when
    \( \varepsilon \in \na{-1,1}^{\spacedim} \), has nonempty interior.
    As \( \bigcup_{\varepsilon' \in \na{-1,1}^{\spacedim} } \Primal_{\varepsilon'}  = ]0,+\infty[^{\spacedim}
    \), there is at least one subset \( L \subset \na{-1,1}^{\spacedim} \) such that
    \( \bigcup_{\varepsilon' \in L} \Primal_{\varepsilon'}  = ]0,+\infty[^{\spacedim} \)
    and the subset~$L$ has the smallest possible cardinal.
    If \( \cardinal{L}=1 \), then there is one \( \varepsilon \in \na{-1,1}^{\spacedim} \) such that \(
    \Primal_{\varepsilon} = ]0,+\infty[^{\spacedim}\), and this \( \Primal_{\varepsilon} \) obviously has nonempty interior.
    If \( \cardinal{L} \geq 2 \), then
    \( \bigcup_{\varepsilon' \in \na{-1,1}^{\spacedim} } \Primal_{\varepsilon'}  = ]0,+\infty[^{\spacedim} \)
    implies that, for any \( \varepsilon \in L \),
    we have that \( \emptyset \subsetneq \Bp{ \bigcup_{\varepsilon' \in L\setminus \na{\varepsilon}}
      \Primal_{\varepsilon'} }^c \subset \Primal_{\varepsilon} \).
    Therefore, the subset \( \Primal_{\varepsilon} \) has nonempty interior since
    it contains the nonempty set  \( \Bp{ \bigcup_{\varepsilon' \in L\setminus \na{\varepsilon}}
      \Primal_{\varepsilon'} }^c \), which is open as the complementary set of
    a finite union of closed subsets.

    As a consequence
    of~\eqref{eq:ConditionalInfimum_quadratic-example_varepsilon'_inproof},
    there is one \( \varepsilon \in \na{-1,1}^{\spacedim} \)
    and there is a ball~$B$ in~\( ]0,+\infty[^{\spacedim} \) such that
    \[
      \fonctionprimalbis\np{\primal_{1},\ldots,\primal_{\spacedim}}
      =
      \sum_{i\neq j} \matrice_{i,j}\varepsilon_i\varepsilon_j \sqrt{\primal_i\primal_j}
      + \sum_{i=1}^{\spacedim} \vecteur_i\varepsilon_i \sqrt{\primal_i}
      \eqsepv \forall \primal \in B
      \eqfinp
    \]
    As the function \( \fonctionprimalbis \) is convex,
    so is the function \( k \colon B \ni \primal \mapsto \sum_{i\neq j} \matrice_{i,j}\varepsilon_i\varepsilon_j \sqrt{\primal_i\primal_j}
    + \sum_{i=1}^{\spacedim} \vecteur_i\varepsilon_i \sqrt{\primal_i} \), and so are
    the restrictions
    \begin{align*}
      \primal_i \mapsto k\np{0,\ldots,0,\primal_i,0,\ldots,0}
      &=
        \matrice_{i,i}\varepsilon_i^2\primal_i + \vecteur_i\varepsilon_i \sqrt{\primal_i}
        \eqfinv
      \\
      \np{\primal_i,\primal_j} \mapsto
      k\np{0,\ldots,0,\primal_i,0,\ldots,0,\primal_j,0,\ldots,0}
      &=
        \matrice_{i,j}\varepsilon_i\varepsilon_j \sqrt{\primal_i\primal_j}
        + \vecteur_i\varepsilon_i \sqrt{\primal_i} + \vecteur_j\varepsilon_j \sqrt{\primal_j}
        \eqfinp
    \end{align*}
    We deduce that, for any~$i$, the functions \( \primal_i \mapsto \vecteur_i\varepsilon_i \sqrt{\primal_i} \)
    are convex in a neighborhood of a point in \( ]0,+\infty[ \), hence that
    \( \vecteur_i\varepsilon_i \leq 0 \) (by differentiating twice in~\( \primal_i \) for example).
    We also deduce that, for any \( i \neq j \), the functions \( \np{\primal_i,\primal_j} \mapsto
    \matrice_{i,j}\varepsilon_i\varepsilon_j \sqrt{\primal_i\primal_j}
    + \vecteur_i\varepsilon_i \sqrt{\primal_i} + \vecteur_j\varepsilon_j \sqrt{\primal_j} \)
    are convex in a neighborhood of a point in \( ]0,+\infty[^{2} \), hence
    --- with what we have just obtained ---
    so are the functions \( \np{\primal_i,\primal_j} \mapsto
    \matrice_{i,j}\varepsilon_i\varepsilon_j \sqrt{\primal_i\primal_j} \).
    In the same way, we get that \( \matrice_{i,j}\varepsilon_i\varepsilon_j \leq 0 \).
    Thus, we have obtained that \( \varepsilon \in \na{-1,1}^{\spacedim} \) satifies~\eqref{eq:varepsilon},
    hence that the pair \( \np{\matrice,\vecteur} \) can be put in {block-signed} form
    by Proposition~\ref{pr:varepsilon_sign_appendix} in Appendix~\ref{Appendix}.
    This ends the proof.
  \end{subequations}
\end{proof}

Under an additional assumption, we obtain a sufficient condition for
the function \(
\ConditionalInfimum{\Epigraph_{\RR_{+}^{\spacedim}}\SquareMapping}{\LinearQuadratic}
\) to be convex.

\begin{proposition}
  Let 
  $\vecteur \in \RR^{\spacedim}$ be a vector, and $\matrice$ be a
  $\spacedim{\times}\spacedim$ symmetric matrix.
  Let the quadratic
  function \( \LinearQuadratic \colon \RR^{\spacedim} \to \barRR \) be given
  by~\eqref{eq:LinearQuadraticFonctionuncertain}.
If the pair \( \np{\matrice,\vecteur} \) can be put in {block-signed} form
    (or, equivalently, the quadratic function \( \LinearQuadratic \colon \RR^{\spacedim} \to \barRR \) given
    by~\eqref{eq:LinearQuadraticFonctionuncertain} can be put in {block-signed} form),
    and all the diagonal terms in~\eqref{eq:block-signed_form} are nonpositive,
    then
    the function \( \ConditionalInfimum{\Epigraph_{\RR_{+}^{\spacedim}}\SquareMapping}{\LinearQuadratic} \colon \RR^{\spacedim}
    \to \RR\cup\na{+\infty} \), defined by, for all \( \np{\primal_{1},\ldots,\primal_{\spacedim}} \in
      \RR^{\spacedim} \), 
    \begin{equation}
      \begin{split}
\ConditionalInfimum{\Epigraph_{\RR_{+}^{\spacedim}}\SquareMapping}{\LinearQuadratic}\np{\primal_{1},\ldots,\primal_{\spacedim}}
        =\inf\bset{ \uncertain\transp\matrice\uncertain + \uncertain\transp\vecteur }%
      { \uncertain_{1}^2\leq\primal_{1},\ldots,\uncertain_{\spacedim}^2\leq \primal_{\spacedim} }
      \eqfinv
      \end{split}
    \end{equation}
is proper convex \lsc\ and satisfies
\(    \ConditionalInfimum{\SquareMapping}{\LinearQuadratic}=
    \ConditionalInfimum{\Epigraph_{\RR_{+}^{\spacedim}}\SquareMapping}{\LinearQuadratic}
    \), that is, 
  \begin{align}
& 
      \inf\bset{ \uncertain\transp\matrice\uncertain + \uncertain\transp\vecteur }%
 { \uncertain_{1}^2=\primal_{1},\ldots,\uncertain_{\spacedim}^2=\primal_{\spacedim} }
\nonumber \\
    =&
 \inf\bset{ \uncertain\transp\matrice\uncertain + \uncertain\transp\vecteur }%
 { \uncertain_{1}^2\leq\primal_{1},\ldots,\uncertain_{\spacedim}^2\leq\primal_{\spacedim} }
      \eqsepv
      \forall \np{\primal_{1},\ldots,\primal_{\spacedim}} \in \RR^{\spacedim}
       \eqfinp
  \end{align}
\label{pr:Conditional_infimum_of_a_quadratic_function_knowing_squares}
\end{proposition}

\begin{proof}
    Suppose that the pair \( \np{\matrice,\vecteur} \) can be put in {block-signed} form
     where all the diagonal terms in~\eqref{eq:block-signed_form} are nonpositive.
    The functions \( \ConditionalInfimum{\SquareMapping}{\LinearQuadratic} \) and
    \( \ConditionalInfimum{\Epigraph_{\RR_{+}^{\spacedim}}\SquareMapping}{\LinearQuadratic} \)
    coincide outside of~\( \RR_+^{\spacedim} \) as they both take the value~\( +\infty \).
We have, for any \( \primal=\np{\primal_{1},\ldots,\primal_{\spacedim}}
    \in \RR_+^{\spacedim} \),
    \begin{align*}
 &        \bp{\ConditionalInfimum{\SquareMapping}{\LinearQuadratic}}
   \np{\primal_{1},\ldots,\primal_{\spacedim}}
           \nonumber
      \\
      \geq&
            \bp{\ConditionalInfimum{\Epigraph_{\RR_{+}^{\spacedim}}\SquareMapping}{\LinearQuadratic}}
            \np{\primal_{1},\ldots,\primal_{\spacedim}}
            \tag{by monotonicity~\eqref{eq:correspondence_conditional_infimum_properties_inclusion} as
            \( \graph_{\SquareMapping} \subset \Epigraph_{\RR_{+}^{\spacedim}}\SquareMapping \)
            by~\eqref{eq:Cone-epigraph}}
      \\
      =&
        \inf\bset{\uncertain\transp\matrice\uncertain + \uncertain\transp\vecteur }{%
        \uncertain\in\RR^{\spacedim} \eqsepv
        \np{\uncertain_{1}^2,\ldots,\uncertain_{\spacedim}^2} \leq_{\RR_{+}^{\spacedim}} \np{\primal_{1},\ldots,\primal_{\spacedim}} }
        \nonumber
        \intertext{by definition~\eqref{eq:correspondence_conditional_infimum} of the conditional
        infimum \wrt\ the \( \RR_{+}^{\spacedim} \)-epigraphic correspondence
~\( \Epigraph_{\RR_{+}^{\spacedim}}\SquareMapping \) as defined
  in~\eqref{eq:RR_+p-epigraph}, and by definition~\eqref{eq:SquareMapping} of the square mapping}
      =&
        \inf\Bset{ \sum_{i=1}^{\spacedim} \matrice_{i,i} \uncertain_i^2 +
        \sum_{i\neq j} \matrice_{i,j} \uncertain_i\uncertain_j
        + \sum_{i=1}^{\spacedim} \vecteur_i \uncertain_i }{%
        \uncertain_{1}^2\leq\primal_{1},\ldots,\uncertain_{\spacedim}^2\leq\primal_{\spacedim}}
        \nonumber
      \\
      =&
         \inf\Bset{ \sum_{i=1}^{\spacedim} \matrice_{i,i} \uncertain_i^2 +
        \sum_{i\neq j} \matrice_{i,j} \uncertain_i\uncertain_j
        + \sum_{i=1}^{\spacedim} \vecteur_i \uncertain_i }{%
        \uncertain_{1}\in\ClosedIntervalClosed{-\sqrt{\primal_{1}}}{\sqrt{\primal_{1}}},
        \ldots,\uncertain_{\spacedim}\in\ClosedIntervalClosed{-\sqrt{\primal_{\spacedim}}}{\sqrt{\primal_{\spacedim}}}
}
        \nonumber
      \\
      =&
        \min\Bset{
       \sum_{i=1}^{\spacedim} \matrice_{i,i} \np{\varepsilon'_i}^2 \primal_i +
         \sum_{i\neq j} \matrice_{i,j}\varepsilon'_i\varepsilon'_j \sqrt{\primal_i\primal_j}
        + \sum_{i=1}^{\spacedim} \vecteur_i\varepsilon'_i \sqrt{\primal_i} }{%
        \varepsilon' \in \ClosedIntervalClosed{-1}{1}^{\spacedim} }
      \\
      \geq&
        \sum_{i=1}^{\spacedim} \min\na{0,\matrice_{i,i}} \primal_i +
        \sum_{i\neq j} \module{\matrice_{i,j}} \np{-\sqrt{\primal_i \primal_j} }
            + \sum_{i=1}^{\spacedim} \module{\vecteur_i} \np{-\sqrt{\primal_i }}
\tag{by taking term by term inequality}
      \\
      =&
        \sum_{i=1}^{\spacedim} \matrice_{i,i} \primal_i +
        \sum_{i\neq j} \module{\matrice_{i,j}} \np{-\sqrt{\primal_i \primal_j} }
            + \sum_{i=1}^{\spacedim} \module{\vecteur_i} \np{-\sqrt{\primal_i }}
\tag{as \( \min\na{0,\matrice_{i,i}}=\matrice_{i,i} \) since \( \matrice_{i,i}\leq 0 \) by assumption}
      \\
      =&
      \bp{\ConditionalInfimum{\SquareMapping}{\LinearQuadratic}}
         \np{\primal_{1},\ldots,\primal_{\spacedim}}
     \end{align*}
     by~\eqref{eq:ConditionalInfimum_quadratic_convex_LB},
     since the pair \( \np{\matrice,\vecteur} \) can be put in {block-signed} form by assumption,
 hence Item~\ref{it:block-signed_form} of Theorem~\ref{th:Conditional_infimum_of_a_quadratic_function_knowing_squares}
 holds true.
 We conclude that
\( \ConditionalInfimum{\Epigraph_{\RR_{+}^{\spacedim}}\SquareMapping}{\LinearQuadratic}
= \ConditionalInfimum{\SquareMapping}{\LinearQuadratic} \)
    is a proper convex \lsc\ function.
  \end{proof}

\subsection{Hidden convexity in quadratic minimization problems}
\label{Hidden_convexity_in_quadratic_minimization_problems}

In~\S\ref{New_sufficient_conditions}, we derive from
Theorem~\ref{th:Conditional_infimum_of_a_quadratic_function_knowing_squares} a
new sufficient condition for hidden convexity in quadratic minimization problems
(Theorem~\ref{th:ConditionalInfimum_quadratic-example}).
Then, we provide several corollaries 
that we compare with existing results on quadratic minimization under quadratic
constraints in~\S\ref{quadratic_minimization_under_quadratic_constraints}, the
determination of trust regions in optimization in~\S\ref{Trust_region_problems},
and weighted max-cut problems in~\S\ref{Weighted_max-cut_problems}.

\subsubsection{New sufficient condition for hidden convexity in quadratic minimization problems}
\label{New_sufficient_conditions}

We provide now a sufficient condition under which the minimization of a
quadratic function plus a convex function of the square
mapping~\eqref{eq:SquareMapping} can be turned into a convex minimization
problem.

\begin{theorem}
  \label{th:ConditionalInfimum_quadratic-example}
  Let \( \phi\colon \RR_+^{\spacedim} \to \RR\cup\na{+\infty} \) be a convex function (possibly incorporating constraints).
  Let $\matrice$ be a $\spacedim\times\spacedim$ symmetric matrix,
  $\vecteur \in \RR^{\spacedim}$ be a vector.
  If the pair \( \np{\matrice,\vecteur} \) can be put in {block-signed} form
  (see Definition~\ref{de:block-signed form}),
  then we have that
  \begin{equation}
    \inf_{\uncertain\in\RR^{\spacedim}}\uncertain\transp\matrice\uncertain + \uncertain\transp\vecteur
    + \phi\np{\uncertain_{1}^2,\ldots,\uncertain_{\spacedim}^2}         
    =
    \overbrace{ 
      \inf_{ \primal \in \RR_+^{\spacedim} }\widetilde{\LinearQuadratic}\np{\primal}+\phi\np{\primal}
    }^{\textrm{convex program}} 
    \eqfinv
    \label{eq:ConditionalInfimum_quadratic_problem}
  \end{equation}
  where the function \( \widetilde{\LinearQuadratic} \colon \RR^{\spacedim} \to \RR\cup\na{+\infty} \)
  is proper convex \lsc\ with effective domain \( \domain\widetilde{\LinearQuadratic}=\RR_+^{\spacedim}\),
  and is given by~\eqref{eq:ConditionalInfimum_quadratic_convex_LB}.

  Let \( \varepsilon=\np{\varepsilon_{1},\ldots,\varepsilon_{\spacedim}} \in
  \na{-1,1}^{\spacedim} \) be given by Definition~\ref{de:block-signed form}.
  Then, regarding argmin, we have the following implication
  \begin{equation}
    \begin{split}
      \primal\opt \in \argmin_{ \primal \in \RR_+^{\spacedim} }\widetilde{\LinearQuadratic}\np{\primal}+\phi\np{\primal}
      \implies \\
      \varepsilon \cdot \sqrt{\primal\opt} \in
      \argmin_{\uncertain\in\RR^{\spacedim}}\uncertain\transp\matrice\uncertain + \uncertain\transp\vecteur
      + \phi\np{\uncertain_{1}^2,\ldots,\uncertain_{\spacedim}^2}         
      \eqfinv
    \end{split}
    \label{eq:argmin}
  \end{equation}
  where the vector \( \varepsilon \cdot \sqrt{\primal\opt} \in \RR^{\spacedim} \)
  has components
  \( \varepsilon_i\sqrt{\primal\opt_i} \), for \( i=1,\ldots,\spacedim \).
  %
\end{theorem}

\begin{proof}
  Equation~\eqref{eq:ConditionalInfimum_quadratic_problem} is
  a straightforward application of
  Proposition~\ref{pr:inf_ConditionalInfimum_inf_original},
  as we have that
  \begin{subequations}
    \begin{align*}
      \inf_{\uncertain\in\RR^{\spacedim}}\uncertain\transp\matrice\uncertain + \uncertain\transp\vecteur
      + \phi\np{\uncertain_{1}^2,\ldots,\uncertain_{\spacedim}^2}
      &=
        \inf_{\uncertain \in \graph_{\SquareMapping}\RR_{+}^{\spacedim}}
        \LinearQuadratic\np{\uncertain}+ \phi\bp{\SquareMapping\np{\uncertain}}
        \intertext{as \( \RR^{\spacedim}=\graph_{\SquareMapping}\RR_{+}^{\spacedim} \)
        by definition~\eqref{eq:graph} of the graph of the square mapping~\eqref{eq:SquareMapping},
        and by definition~\eqref{eq:LinearQuadraticFonctionuncertain} of~\( \LinearQuadratic\),}
      &=
        \inf_{\primal \in \RR_{+}^{\spacedim}}
        \bp{\nInfCond{\LinearQuadratic + \phi\circ\SquareMapping}{\graph_{\SquareMapping}}\np{\primal}}
        \intertext{by Equation~\eqref{eq:ConditionalInfimum_tower_property_equality_subset}
        with function \( \fonctionuncertain=\LinearQuadratic+ \phi\circ\SquareMapping\),
        correspondence \( \correspondence=\graph_{\SquareMapping} \) given
        by the graph of the square mapping~\eqref{eq:SquareMapping},
        and (convex) subset \( \Primal=\RR_+^{\spacedim} \),}
      &=
        \inf_{\primal \in \RR_{+}^{\spacedim}}
        \bp{\nInfCond{\LinearQuadratic}{\SquareMapping}\np{\primal} + \phi\np{\primal} }
        \tag{by~\eqref{eq:correspondence_conditional_infimum_property_UppPlus}}
      \\
      &=
        \inf_{ \primal \in \RR_{+}^{\spacedim}} \widetilde{\LinearQuadratic}\np{\primal}+ \phi\np{\primal} 
        \tag{as \( \widetilde{\LinearQuadratic} = \ConditionalInfimum{\SquareMapping}{\LinearQuadratic}\)
        by~\eqref{eq:LinearQuadraticFonctionuncertain_wrt_SquareMapping}}
        \eqfinp
    \end{align*}
  \end{subequations}

  The function \( \widetilde{\LinearQuadratic} \colon \RR^{\spacedim} \to \barRR \) is proper
  convex \lsc\ with effective domain
  \( \domain\widetilde{\LinearQuadratic}=\RR_+^{\spacedim}\), and is given
  by~\eqref{eq:ConditionalInfimum_quadratic_convex_LB}, by
  Theorem~\ref{th:Conditional_infimum_of_a_quadratic_function_knowing_squares}.

  The correspondence~\eqref{eq:argmin} between argmins follows from
  Proposition~\ref{pr:argmin_ConditionalInfimum_argmin_original} 
  by using the vector
  \( \uncertain\opt = \varepsilon \cdot \sqrt{\primal\opt} \in \RR^{\spacedim} \) where
  \( \varepsilon \) is given by Definition~\ref{de:block-signed form}.  \medskip

  This ends the proof.
\end{proof}

%


Because Equation~\eqref{eq:ConditionalInfimum_quadratic_problem} is a
straightforward application of
Proposition~\ref{pr:inf_ConditionalInfimum_inf_original} --- which itself is an
equivalence between {block-signed} form and convexity of a conditional infimum ---
we suspect that Theorem~\ref{th:ConditionalInfimum_quadratic-example} provides
the most general sufficient condition to detect hidden convexity in quadratic
minimization problems under square constraints.  To illustrate this point, we
will now provide several corollaries derived fom
Theorem~\ref{th:ConditionalInfimum_quadratic-example}, and compare them with
prominent results in the literature.

\subsubsection{Quadratic minimization under quadratic constraints}
\label{quadratic_minimization_under_quadratic_constraints}

First, we state a Corollary of
Theorem~\ref{th:ConditionalInfimum_quadratic-example}, then show that it covers
the result in \cite[Theorem~3]{Ben-Tal-den-Hertog:2014}, and that it allows to
detect hidden convexity in the Celis-Dennis-Tapia subproblem
\cite[Sect.~5]{XiaYong:2020}.

\subsubsubsection{Corollary of Theorem~\ref{th:ConditionalInfimum_quadratic-example}}

The first corollary is an answer to the question raised at the beginning of this
Sect.~\ref{Hidden_convexity_in_quadratic_optimization_problems}, as we provide a
new sufficient condition to obtain hidden convexity in problems of minimization
of a quadratic function under quadratic constraints like in~\eqref{eq:LQ}.

\begin{corollary}
  \label{cor:LinearQuadraticProblem}
  Let \( p \in \NN^* \) and let
  \( \varphi\colon \RR^{\constraintdim} \to \barRR \) be a convex function (possibly incorporating constraints).
  Let $ \Quadratic, \Quadratic^{1}, \ldots, \Quadratic^{\constraintdim} \colon \RR^{\spacedim}
  \to \RR $ be quadratic functions, that is, there exist
  $\spacedim\times\spacedim$ (nonzero) symmetric matrices
  $\MatriceObjective, \MatriceConstraint^{1}, \ldots, \MatriceConstraint^{\constraintdim}$,
  vectors $\VecteurObjective, \VecteurConstraint^{1}, \ldots, \VecteurConstraint^{\constraintdim} \in \RR^{\spacedim}$,
  and scalars $\scalaire, \scalaire^{1}, \ldots, \scalaire^{\constraintdim} \in \RR$,
  such that
  \begin{subequations}
    \begin{align}
      \Quadratic\np{\uncertainbis}
      &=
        \uncertainbis\transp\MatriceObjective\uncertainbis + \uncertainbis\transp\VecteurObjective
        + \scalaire
        \eqsepv \forall \uncertainbis \in \RR^{\spacedim}
        \eqfinv
      \\
      \Quadratic^{\LocalIndex}\np{\uncertainbis}
      &=
        \uncertainbis\transp\MatriceConstraint^{\LocalIndex}\uncertainbis + \uncertainbis\transp\VecteurConstraint^{\LocalIndex}
        + \scalaire^{\LocalIndex}
        \eqsepv \forall \uncertainbis \in \RR^{\spacedim}
        \eqsepv \forall \LocalIndex\in\ic{1,\constraintdim}
        \eqfinp
    \end{align}
  \end{subequations}

  Suppose that there exist a non singular $\spacedim{\times}\spacedim$ matrix~$\NonsingularMatrix$ and
  a vector $\bar{\uncertain} \in \RR^{\spacedim}$ such that 
  \begin{subequations}
    \begin{enumerate}
    \item
      ({block-signed} form of the quadratic objective function)
      
      denoting
      \begin{equation}
        \Quadratic\bp{\NonsingularMatrix\np{\uncertain-\bar{\uncertain}}}
        =\uncertain\transp\bar{\MatriceObjective}\uncertain + \uncertain\transp\bar{\VecteurObjective}  + \bar{\scalaire}
        \eqsepv \forall \uncertain \in \RR^{\spacedim}
        \eqfinv
      \end{equation}
      the pair \( \np{\bar{\MatriceObjective},\bar{\VecteurObjective}} \) can be put in {block-signed} form
      (see Definition~\ref{de:block-signed form}),
    \item
      (joint diagonalization of ``constraint'' matrices)
      
      there exist $\spacedim\times\spacedim$ diagonal matrices $\bar{\MatriceConstraint}^{1}, \ldots, \bar{\MatriceConstraint}^{\constraintdim}$
      and scalars $\bar{\scalaire}^{1}, \ldots, \bar{\scalaire}^{\constraintdim} \in \RR$,
      such that
      \begin{equation}
        \Quadratic^{\LocalIndex}\bp{\NonsingularMatrix\np{\uncertain-\bar{\uncertain}}}
        = \uncertain\transp\bar{\MatriceConstraint}^{\LocalIndex}\uncertain
        + \bar{\scalaire}^{\LocalIndex}
        \eqsepv \forall \uncertain \in \RR^{\spacedim}
        \eqsepv \forall \LocalIndex\in\ic{1,\constraintdim}
        \eqfinv
      \end{equation}
    \end{enumerate}
  \end{subequations}
  Then, the minimization problem
  \begin{equation}
    \inf_{\uncertain \in \RR^{\spacedim}} \Quadratic\np{\uncertain}
    + \varphi \bp{\Quadratic^{1}\np{\uncertain}, \ldots, \Quadratic^{\constraintdim}\np{\uncertain} }
    \label{eq:LinearQuadraticProblem}
  \end{equation}
  is equivalent to a convex minimization problem.
\end{corollary}

\begin{proof}
  We just provide a sketch of proof.
  For this purpose, we introduce
the $p\times\spacedim$ matrix~\( \bar{\MatriceConstraint}
=  \sequence{\bar{\MatriceConstraint}^{\LocalIndex}_{\LocalIndexbis,\LocalIndexbis}}%
{\LocalIndex\in\ic{1,\constraintdim}, \LocalIndexbis\in\ic{1,\spacedim}} \), and get that
\begin{align*}
&  \inf_{\uncertain \in \RR^{\spacedim}} \Quadratic\np{\uncertain}
    + \varphi \bp{\Quadratic^{1}\np{\uncertain}, \ldots, \Quadratic^{\constraintdim}\np{\uncertain} }
  \\
  =&
     \inf_{\primal\in \RR^{\constraintdim}}
     \InfCond{\Quadratic + \varphi \np{\Quadratic^{1}, \ldots, \Quadratic^{\constraintdim}}}%
     {\np{\Quadratic^{1}, \ldots, \Quadratic^{\constraintdim}}}\np{\primal}
  \\
  =&
     \inf_{\primal\in \RR^{\constraintdim}}
     \InfCond{\Quadratic}%
     {\np{\Quadratic^{1}, \ldots, \Quadratic^{\constraintdim}}}\np{\primal}
     \UppPlus \varphi \np{\primal^{1}, \ldots, \primal^{\constraintdim}}
  \\
  =&
     \inf_{\primal\in \RR^{\constraintdim}}
     \ConditionalInfimum%
    {\uncertain\transp\bar{\MatriceConstraint}^{\LocalIndex}\uncertain
      + \bar{\scalaire}^{\LocalIndex}=\primal_{\LocalIndex} \eqsepv
      \LocalIndex\in\ic{1,\constraintdim} }%
    {\uncertain\transp\bar{\MatriceObjective}\uncertain +
    \uncertain\transp\bar{\VecteurObjective} + \bar{\scalaire}}%
     \UppPlus \varphi \np{\primal^{1}, \ldots, \primal^{\constraintdim}}
  \\
  =&
          \inf_{\primal\in \RR^{\constraintdim}}
    \ConditionalInfimum%
     {\bar{\MatriceConstraint}\SquareMapping\np{\uncertain}=
\primal_{\LocalIndex}-\bar{\scalaire}^{\LocalIndex} \eqsepv
      \LocalIndex\in\ic{1,\constraintdim} }%
    {\uncertain\transp\bar{\MatriceObjective}\uncertain +
      \uncertain\transp\bar{\VecteurObjective} + \bar{\scalaire}}%
     \UppPlus \varphi \np{\primal^{1}, \ldots, \primal^{\constraintdim}}
  \\
  =&
          \inf_{\primal\in \RR^{\constraintdim}}
    \underbrace{ \BInfCond{%
    \underbrace{ \InfCond{\uncertain\transp\bar{\MatriceObjective}\uncertain +
     \uncertain\transp\bar{\VecteurObjective} + \bar{\scalaire}}%
     {\SquareMapping}}_{\textrm{convex function}} }%
     {\underbrace{\bar{\MatriceConstraint}}_{\substack{\textrm{linear}\\ \textrm{mapping}}}}%
     }_{\textrm{convex function}}%
     \np{\primal_{\LocalIndex}-\bar{\scalaire}^{\LocalIndex} \eqsepv     
     \LocalIndex\in\ic{1,\constraintdim} }     
     \UppPlus \underbrace{ \varphi }_{\substack{\textrm{convex}\\ \textrm{function}}}\np{\primal^{1}, \ldots, \primal^{\constraintdim}}
  \eqfinp
\end{align*}
\end{proof}

To our knowledge, the conditions that we propose in
Corollary~\ref{cor:LinearQuadraticProblem} go beyond existing sufficient
conditions.  Indeed, as far as we know, \emph{all known results} rely on
\emph{joint diagonalization} of \emph{all} the symmetric \emph{matrices
  $\MatriceObjective$ (objective) and
  $\MatriceConstraint^{1}, \ldots, \MatriceConstraint^{\constraintdim}$ (constraints)}
in~\eqref{eq:LinearQuadraticProblem}.  By contrast, our result in
Corollary~\ref{cor:LinearQuadraticProblem} only requires, on the one hand, joint
diagonalization of the symmetric matrices
$\MatriceConstraint^{1}, \ldots, \MatriceConstraint^{\constraintdim}$ that appear in the constraint
and, on the other hand, that the pair
\( \np{\bar{\MatriceObjective},\bar{\VecteurObjective}} \) can be put in
{block-signed} form.  Thus, the symmetric matrix~$\bar{\MatriceObjective}$ can
have, in the diagonalizing basis and after rearrangement of indices, an
expression like in~\eqref{eq:block-signed_form}, hence we do not require the
symmetric matrix~$\bar{\MatriceObjective}$ in the objective function to be
diagonal, but we leave room for off-diagonal terms as far as these terms have
signs like in~\eqref{eq:block-signed_form}. 

\subsubsubsection{Parallel between probability theory and optimization}

Let us point an analogy between probability theory
and optimization.
The so-called \emph{linearity of gaussian regression} is well known:
if \( \np{ \va{X}_0, \va{X}_1, \ldots, \va{X}_{\constraintdim} } \) is a Gaussian random
vector in~$\RR^\spacedim$, then
\[
 \espe\conditionaly{\va{X}_0}{ \va{X}_1, \ldots, \va{X}_{\constraintdim} }
  =\textrm{affine function of~}\np{ \va{X}_1, \ldots, \va{X}_{\constraintdim} } \eqfinp
\]
In optimization, this becomes the \emph{convexity of quadratic forms infimal
  regression} as follows. 
\begin{theorem}
  If \( \Quadratic_{0}, \Quadratic_{1}, \ldots, \Quadratic_{\constraintdim} \) are
  diagonal quadratic forms on~$\RR^\spacedim$, then
    \[
\InfCond{\Quadratic_{0}}{\Quadratic_{1}, \ldots, \Quadratic_{\constraintdim} }
    =\textrm{convex function of~}\np{\Quadratic_{1}, \ldots, \Quadratic_{\constraintdim} }  \eqfinp
  \]
\end{theorem}

\subsubsubsection{Comparison with \cite[Theorem~3]{Ben-Tal-den-Hertog:2014}}

We state a version of \cite[Theorem~3]{Ben-Tal-den-Hertog:2014} without
\cite[Assumption~2]{Ben-Tal-den-Hertog:2014}.  The proof shows that it is a
consequence of Corollary~\ref{cor:LinearQuadraticProblem}.  \medskip

\begin{subequations}
  \noindent\emph{\textbf{\cite[Theorem~3]{Ben-Tal-den-Hertog:2014} without \cite[Assumption~2]{Ben-Tal-den-Hertog:2014}} 
    Let $D$ and $A$ be symmetric $\spacedim\times\spacedim$ matrices,
    $b, e\in\RR^{\spacedim}$ be vectors, and $c$ be a real number.
    Suppose that there exists a non singular matrix $S$
    such that
    \begin{equation}
      S \transp D S = \mathrm{diag}(\delta_{1},\ldots,\delta_{\spacedim})
      \eqsepv
      S \transp A S =\mathrm{diag}(\alpha_{1},\ldots,\alpha_{\spacedim})
      \eqfinp
      \label{eq:BenTal-Hertog:diag_assumption}
    \end{equation}
    Then, the minimization problem
    \begin{equation}
      \inf\bset{\frac{1}{2} z \transp D z  + e\transp z }{%
        z\in\RR^{\spacedim} \eqsepv \frac{1}{2} z \transp A z  + b\transp z +c \leq 0}
      \label{eq:Ben-Tal-den-Hertog:2014}
    \end{equation}
    is equivalent to a convex minimization problem and,
      when there are optimal solutions, they are in one-to-one correspondence.
  }
  \medskip

  \begin{proof} It is a simple consequence of Corollary~\ref{cor:LinearQuadraticProblem} but we detail the steps
    to follow~\cite[p.~4]{Ben-Tal-den-Hertog:2014}.
    Using assumption~\eqref{eq:BenTal-Hertog:diag_assumption}, and as done
  in \cite[p.~4]{Ben-Tal-den-Hertog:2014}, the minimization
  problem~\eqref{eq:Ben-Tal-den-Hertog:2014} is equivalent to the minimization problem
  \begin{equation}
    \inf\Bset{ \sum_{i=1}^{\spacedim} \frac{1}{2} \delta_i z_i^2 + e_i z_i }{%
      z\in\RR^{\spacedim} \eqsepv \sum_{i=1}^{\spacedim}
      \frac{1}{2} \alpha_i z_i^2   + b_i z_i +c \leq 0}
    \eqfinv
  \end{equation}
  which is also equivalent --- up to the affine transformation $y= z + \rho$,
  whith $\rho_i = b_i/\alpha_i$ when $\alpha_i \not=0$ and $\rho_i = 0$ when $\alpha_i =0$ --- to
  \begin{equation}
    \inf\Bset{ \sum_{i=1}^{\spacedim} \frac{1}{2} \delta_i y_i^2 + (e_i - \delta_i \rho_i) y_i -
       \frac{3}{2} \delta_i \rho_i ^2 - e_i \rho_i
     }{%
      y\in\RR^{\spacedim} \eqsepv \sum_{i=1}^{\spacedim}
      \frac{1}{2} \alpha_i y_i^2  + (c - \alpha_i\rho_i) \leq 0}
    \eqfinp
  \end{equation}
  We conclude that this latter problem is equivalent to a convex minimization problem
  using Theorem~\ref{th:ConditionalInfimum_quadratic-example} with
  \( \varphi=\Indicator{\RR_{-}} \) the indicator function of~\( \RR_{-} \). 
\end{proof}
\end{subequations}

\subsubsubsection{Hidden convexity in the Celis-Dennis-Tapia (CDT) subproblem \cite[Sect.~5]{XiaYong:2020}}

It is said in \cite[Sect.~5]{XiaYong:2020} that hidden convex reformulation remains unknown
for the so-called \emph{Celis-Dennis-Tapia (CDT) subproblem}, which consists in
minimizing a quadratic function under two quadratic constraints.
From Corollary~\ref{cor:LinearQuadraticProblem}, we easily deduce the following convex reformulation result.

\begin{corollary}
  If \( \MatriceConstraint^{1} \), \( \MatriceConstraint^{2} \) are diagonal matrices, and if
  the pair \( \np{\MatriceObjective,\VecteurObjective} \) can be put in {block-signed} form
  (see Definition~\ref{de:block-signed form}),
  there is a convex reformulation for the minimization problem
  \begin{subequations}
    \begin{align}
      \min \quad &
           \uncertainbis\transp\MatriceObjective\uncertainbis + \uncertainbis\transp\VecteurObjective
      \\
             &
               \uncertainbis\transp\MatriceConstraint^{1}\uncertainbis  + \scalaire^{1} \leq 0
      \\
             &
               \uncertainbis\transp\MatriceConstraint^{2}\uncertainbis + \scalaire^{2} \leq 0
               \eqfinp
    \end{align}  
  \end{subequations}  
\end{corollary}


\subsubsection{Determination of trust regions in optimization} 
\label{Trust_region_problems}

In optimization, a trust region is a subset of the domain of the objective
function where this latter is suitably approximated using a model function
(often a quadratic).  First, we state a Corollary of
Theorem~\ref{th:ConditionalInfimum_quadratic-example}, then show that it covers
the results in \cite[Theorem~7]{BenTal-Ben-Teboulle:1996}, \cite[TRS
\S8.2.7]{Beck:2014} and \cite[Proposition~2.4]{Xia-Sheu-Yuan:2017}.

\subsubsubsection{Corollary of Theorem~\ref{th:ConditionalInfimum_quadratic-example}}

\begin{subequations}
  \begin{corollary}
    \label{cor:trust_phi}
    Let $\MatriceObjective, \MatriceConstraint$ be two $\spacedim\times\spacedim$ symmetric matrices,
    $\VecteurObjective \in \RR^{\spacedim}$ be a vector,
    and \( \varphi\colon \RR \to \barRR \) be a convex function (possibly incorporating constraints).
    Suppose that there exists a nonsingular matrix~$\NonsingularMatrix$ such that
    \begin{enumerate}
    \item
      \label{it:cor:trust_phi_block-signed}
      ({block-signed} form of the quadratic objective function)
      
      the pair \( \np{ \bar{\MatriceObjective}, \bar{\VecteurObjective}}
      =\np{ \NonsingularMatrix\transp \MatriceObjective \NonsingularMatrix,
        \NonsingularMatrix\transp\VecteurObjective } \) can be put in {block-signed} form
      (see Definition~\ref{de:block-signed form}),
    \item
      (diagonalization of the ``constraint'' matrix)
      
      the matrix \( \bar{\MatriceConstraint}= \NonsingularMatrix\transp \MatriceConstraint \NonsingularMatrix \)
      is diagonal.    
    \end{enumerate}
    Then, the solutions of the minimization problem
    \begin{equation}
      \min_{\uncertainbis\in\RR^{\spacedim}} \uncertainbis\transp\MatriceObjective\uncertainbis
      + \VecteurObjective\transp\uncertainbis + \varphi\np{ \uncertainbis\transp \MatriceConstraint \uncertainbis }         
      \label{eq:symmetric_matrix_nondecreasing_convex_function}
    \end{equation}
    can be obtained from those of the convex minimization problem
    \begin{equation}
      \min_{ \np{\primal_{1},\ldots,\primal_{\spacedim}} \in \RR_+^{\spacedim} }
      \sum_{i=1}^{\spacedim} \bar{\MatriceObjective}_{ii} \primal_i
      - \sum_{i=1}^{\spacedim} \module{\bar{\VecteurObjective}_i} \sqrt{\primal_i}
      +\varphi\np{\sum_{i=1}^{\spacedim} \bar{\MatriceConstraint}_{ii} \primal_i}
      \eqfinp
    \end{equation}
  \end{corollary}

  \begin{proof}
    By setting \( \uncertainbis=\NonsingularMatrix\uncertain \)
    (that is, by the change of variable \( \uncertainbis \mapsto \Converse{\NonsingularMatrix}\uncertainbis=\uncertain \))
    --- and using the definitions
    \( \bar{\MatriceConstraint}= \NonsingularMatrix\transp \MatriceConstraint \NonsingularMatrix \),
    \( \bar{\MatriceObjective}=  \NonsingularMatrix\transp \MatriceObjective \NonsingularMatrix \),
    \( \bar{\VecteurObjective}= \NonsingularMatrix\transp\VecteurObjective \) --- 
    the minimization problem~\eqref{eq:symmetric_matrix_nondecreasing_convex_function} is equivalent to
    \begin{equation}
      \min_{\uncertain\in\RR^{\spacedim}} \sum_{i=1}^{\spacedim} \bar{\MatriceObjective}_{ii} \uncertain_i^2
      + \sum_{i=1}^{\spacedim} \bar{\VecteurObjective}_{i}\uncertain_{i}
      + \varphi\np{ \sum_{i=1}^{\spacedim} \bar{\MatriceConstraint}_{ii} \uncertain_i^2}         
      \eqfinp 
      \label{eq:symmetric_matrix_nondecreasing_convex_function_bis}
    \end{equation}
    Then, it suffices to apply
    Theorem~\ref{th:ConditionalInfimum_quadratic-example}
    with \( \matrice=\bar{\MatriceObjective} \), \( \vecteur=\bar{\VecteurObjective}\),
    \( \varepsilon =- \textrm{sign} \np{\vecteur} =- \textrm{sign} \np{\bar{\VecteurObjective}} \),
    and with the convex function~\( \phi \) defined by  \( \phi\np{\primal_{1},\ldots,\primal_{\spacedim}}
    = \varphi\np{\sum_{i=1}^{\spacedim} \bar{\MatriceConstraint}_{ii} \primal_i} \)
    for all \( \np{\primal_{1},\ldots,\primal_{\spacedim}} \in \RR_+^{\spacedim} \).
  \end{proof}
\end{subequations}

\subsubsubsection{Comparison with \cite[Theorem~7]{BenTal-Ben-Teboulle:1996}}
  
We state \cite[Theorem~7]{BenTal-Ben-Teboulle:1996} and then show that it is a
consequence of Corollary~\ref{cor:trust_phi}. 
\medskip

\noindent\emph{\textbf{\cite[Theorem~7]{BenTal-Ben-Teboulle:1996}}
  Let $Q$ and $M$ be symmetric $\spacedim\times\spacedim$ matrices,
  $g\in\RR^{\spacedim}$ be a vector, and $l$, $u$ be real numbers.
  Suppose that the following minimization problem (\cite[Problem~Q, p.~4]{BenTal-Ben-Teboulle:1996}) 
  \begin{subequations}
    \begin{equation}
      \min\bset{z\transp Q z  - 2 g\transp z }{%
        z\in\RR^{\spacedim} \eqsepv l \leq z\transp M z  \leq u}
      \label{eq:BenTal-Ben-Teboulle:1996}
    \end{equation}
    is feasible (\cite[blanket Assumption~(a)]{BenTal-Ben-Teboulle:1996}), and that
    (\cite[blanket Assumption~(b)]{BenTal-Ben-Teboulle:1996})
    \begin{equation}
      \exists \alpha\in\RR \eqsepv Q + \alpha M > 0
      \eqfinv
      \label{eq:BenTal-Ben-Teboulle:1996_assumption}
    \end{equation}
    where the notation $Q + \alpha M > 0$ means that the real symmetric matrix $Q + \alpha M $ is positive definite.
    Then, there exists a nonsingular matrix~$\NonsingularMatrix $ such that
    (\cite[Equations~(1)-(2)]{BenTal-Ben-Teboulle:1996})
    \begin{equation}
      \NonsingularMatrix \transp Q\NonsingularMatrix =\mathrm{diag}(d_{1},\ldots,d_{\spacedim}) \eqsepv
      \NonsingularMatrix \transp M\NonsingularMatrix = \mathrm{diag}(s_{1},\ldots,s_{\spacedim}) \eqfinv
      \label{eq:BenTal-Ben-Teboulle:1996_assumption_implied}
    \end{equation}
    and Problem~\eqref{eq:BenTal-Ben-Teboulle:1996} is equivalent to
    \begin{equation}
      \min\Bset{ \sum_{i=1}^{\spacedim} d_i \primal_i
        -
        \sum_{i=1}^{\spacedim} \module{c_i} \sqrt{\primal_i } }%
      { \np{\primal_{1},\ldots,\primal_{\spacedim}} \in \RR_+^{\spacedim} \eqsepv
        l \leq s\transp  \primal \leq u }
      \eqfinv
    \end{equation}
    with $c = \NonsingularMatrix \transp g$.
  \end{subequations}
}
\medskip

\begin{proof}
  This is a straightforward application of Corollary~\ref{cor:trust_phi}
  with
  $\MatriceObjective=Q, \MatriceConstraint=M$, $\VecteurObjective=-2g$ and
  $\bar{\MatriceObjective}=\mathrm{diag}(d_{1},\ldots,d_{\spacedim})$,
  $\bar{\MatriceConstraint}=\mathrm{diag}(s_{1},\ldots,s_{\spacedim})$,
  $\bar{\VecteurObjective}=c = \NonsingularMatrix \transp g$
  and \( \varphi=\Indicator{[l,u]} \) the indicator function of the segment~\( [l,u] \).
  Indeed, as the square matrix~\( \bar{\MatriceObjective}\) is diagonal,
  the pair \( \np{\bar{\MatriceObjective},\bar{\VecteurObjective}} \) can be put in {block-signed} form
  by simply rearranging the vector~\( \bar{\VecteurObjective} \)
  (see Definition~\ref{de:block-signed form} and especially
  Equation~\eqref{eq:block-signed_form}).
\end{proof}

Corollary~\ref{cor:trust_phi} is more general than
\cite[Theorem~7]{BenTal-Ben-Teboulle:1996}.  Indeed, our proof does not require
the Assumption~\eqref{eq:BenTal-Ben-Teboulle:1996_assumption}, but the weaker
Assumption~\eqref{eq:BenTal-Ben-Teboulle:1996_assumption_implied}, and the even
weaker assumption in Item~\ref{it:cor:trust_phi_block-signed} of
Corollary~\ref{cor:trust_phi}.

\subsubsubsection{Comparison with \cite[TRS \S8.2.7]{Beck:2014}}

We state \cite[TRS \S8.2.7]{Beck:2014}
and then show that it is a consequence of Corollary~\ref{cor:trust_phi}.
\medskip

Let $\norm{\cdot}_2$ denote the Euclidean norm on~\( \RR^{\spacedim} \).

\noindent\emph{\textbf{\cite[TRS \S8.2.7]{Beck:2014}}
  For any symmetric $\spacedim{\times}\spacedim$ matrix $M$, vector $\beta \in \RR^{\spacedim}$ and $c \in \RR$,
  the (possibly nonconvex) \emph{trust region subproblem}
  \begin{equation}
    \min\bset{z\transp M z  + 2 \beta\transp z + c }{%
      z\in\RR^{\spacedim} \eqsepv
      \norm{z}_2 \leq 1}
    \label{eq:trs}
  \end{equation}
  is equivalent to a convex optimization problem.
}
\medskip

\begin{proof}
  The minimization problem~\eqref{eq:trs} is in the format of Corollary~\ref{cor:trust_phi} with
  $\MatriceObjective=M$, \( \MatriceConstraint = \textrm{Id}\),
  $\VecteurObjective=2\beta$, \( \varphi=c+\Indicator{[0,1]} \), where
  \( \Indicator{[0,1]} \) is the indicator function of the segment~\( [0,1] \).
  Let $\NonsingularMatrix$ be an orthogonal matrix such that
  \( \bar{\MatriceObjective}=  \NonsingularMatrix\transp \MatriceObjective \NonsingularMatrix \) is diagonal.
  Then, we get that 
  \( \bar{\MatriceConstraint}= \NonsingularMatrix\transp \MatriceConstraint \NonsingularMatrix = \textrm{Id}\),
  and we set \( \bar{\VecteurObjective}= \NonsingularMatrix\transp\VecteurObjective \).
  As the square matrix~\( \bar{\MatriceObjective}\) is diagonal,
  the pair \( \np{\bar{\MatriceObjective},\bar{\VecteurObjective}} \) can be put in {block-signed} form
  by simply rearranging the vector~\( \bar{\VecteurObjective} \)
  (see Definition~\ref{de:block-signed form} and especially
  Equation~\eqref{eq:block-signed_form}).
  The assumptions of Corollary~\ref{cor:trust_phi} are satisfied, and we conclude. 
\end{proof}

\subsubsubsection{Comparison with \cite[Proposition~2.4]{Xia-Sheu-Yuan:2017}}

We state \cite[Proposition~2.4]{Xia-Sheu-Yuan:2017} 
and then show that it is a consequence of Corollary~\ref{cor:trust_phi}.
\medskip

\noindent\emph{\textbf{\cite[Proposition~2.4]{Xia-Sheu-Yuan:2017}}
  For any symmetric\footnote{%
    \cite[Proposition~2.4]{Xia-Sheu-Yuan:2017} is stated in the case where $M$ is a diagonal matrix,
    but there is no loss of generality since
    there exists an orthogonal matrix~$\NonsingularMatrix$
    such that \( \bar{\MatriceObjective}=\NonsingularMatrix\transp M \NonsingularMatrix \) is diagonal,
    and since \( \norm{z}_2 = \norm{\NonsingularMatrix z}_2 \).}
  $\spacedim{\times}\spacedim$ matrix $M$, vector $\beta \in \RR^{\spacedim}$,
  \( \rho > 0 \) and \( p > 2 \), the (possibly nonconvex) \emph{$p$-regularized subproblem}
  \begin{equation}
    \min_{z\in\RR^{\spacedim}} z\transp M z  + 2 \beta\transp z + \rho\norm{z}_2^{\constraintdim}
    \label{eq:p-regularized_subproblem}
  \end{equation}
  is equivalent to a convex optimization problem.
}
\medskip


\begin{proof}
  The proof is the same as right above, except for
  \( \varphi(t)=t^{p/2} \). 
  
  
\end{proof}

\subsubsection{Weighted max-cut problems}
\label{Weighted_max-cut_problems}

Let $\GRAPH = (\VERTEX, \EDGE)$ be an undirected graph with
${\spacedim} = \cardinal{\VERTEX}$ and $\EDGE \subset \VERTEX{\times}\VERTEX$.  From now
on, we suppose that the nodes are indexed by $\na{1,\ldots,{\spacedim}}$ and we
identify $\VERTEX=\na{1,\ldots,{\spacedim}}$ and
$\EDGE \subset \na{1,\ldots,{\spacedim}}{\times}\na{1,\ldots,{\spacedim}}$.
We suppose given a ${\spacedim}{\times}{\spacedim}$ symmetric matrix~$\matrice$ such
that $\matrice_{i,j} \not=0$ iff there is an arc between the two nodes~$i$
and~$j$ of~$\VERTEX$, that is $\na{i,j}\in E$.
A \emph{cut} in the weighted graph
$(\VERTEX,\matrice)$ is a partition\footnote{%
  Being a partition, one should have $\emptyset\subsetneq S \subsetneq \VERTEX$,
  but the literature is not totally clear on that point.}
$(S, \VERTEX\backslash S)$ of the node
set~$\VERTEX$; the so-called \emph{weighted max-cut
  problem} is to find a cut of maximum total weight or, equivalently, to find a
subset $S\subset \VERTEX$ such that
$\sum_{i\in S, j\in \VERTEX\backslash S} \matrice_{i,j}$ is maximal.  As a cut can be
equivalently given by a vector $\uncertain \in \na{-1,1}^{\spacedim}$, the weighted max-cut problem can
be reformulated as\footnote{%
  Here, there is a subtlety which is not totally clear in the literature.
  Indeed, if a cut is a partition $(S, \VERTEX\backslash S)$ of the node
  set~$\VERTEX$, it should not cover the two polar cases $S=\emptyset$ and $S=\VERTEX$.
  If it does, what is the meaning given to the expression
  $\sum_{i\in S, j\in \VERTEX\backslash S} \matrice_{i,j}$ when $S=\emptyset$ or $S=\VERTEX$?
  If it does not, a cut should be identified with a vector in
  $\na{-1,1}^{\spacedim}\setminus\na{-\1,\1}$, where
\( \1\in\na{-1,1}^{\spacedim} \) is made of ones, hence the
maximum~\eqref{eq:weighted_max-cut_problem} should be changed.
}
\begin{equation}
  \max\bset{
    \frac{1}{4} \sum_{i,j} \matrice_{i,j}(1- \uncertain_i\uncertain_j)}%
  { \uncertain\in \na{-1,1}^{\spacedim} } 
  \eqfinv
  \label{eq:weighted_max-cut_problem}
\end{equation}
with $S=\nset{i \in \VERTEX }{w_i=-1}$.
The max-cut problem is one of the central problems of combinatorial
optimization. This problem was shown to be NP-complete~\cite{Karp:1972};
however, various special cases with polynomial solvability have been
identified~\cite{Boros:1991} as for example when all weights are nonpositive.
%
Using the unconstrained quadratic 0-1 programming formulation, an $O(n)$
algorithm was proposed in~\cite{Barahona:1986} for the case of series-parallel
graphs and in~\cite{Chakradha-Bushnell:1992} when the associated graph is
transformable into a combinational circuit of logic gates. More generally, an
$O(n)$ algorithm was presented in~\cite{Crama-et-al:1990} for graphs of constant
tree width.

Our contribution, a corollary of
Theorem~\ref{th:Conditional_infimum_of_a_quadratic_function_knowing_squares}, is
the following.

\begin{corollary}[Weighted max-cut problem]
  \label{co:weighted-max-cut}
  Let $\spacedim \in \NN^*$ be a positive integer and
  \( \matrice \) be a $\spacedim\times \spacedim$ symmetric matrix.
  If the pair $(\matrice, 0)$ can be put in block-signed form with $(P,M)$
  to be found in Definition~\ref{de:block-signed form},
  we have that 
  \begin{equation}
    \max\bset{
      \frac{1}{4} \sum_{i,j} \matrice_{i,j}(1- \uncertain_i\uncertain_j)}%
    { \uncertain\in \na{-1,1}^{\spacedim} }
    = \sum_{i\in P, j \in M} \matrice_{i,j}
    \label{eq:maxcut_val}
    \eqfinv
  \end{equation}
with the convention that \( \sum_{i\in P, j \in M} \matrice_{i,j} =0\) if one of
the sets~$P$ or~$M$ is empty,
and the solution of the weighted max-cut
problem~\eqref{eq:weighted_max-cut_problem} is given by~\(S=M\).
\end{corollary}

\begin{proof} We have that
  \begin{align*}
    \max_{\uncertain\in \na{-1,1}^{\spacedim}}
    \frac{1}{4} \sum_{i,j} \matrice_{i,j}(1- \uncertain_i\uncertain_j)
    &= 
      \frac{1}{4} \sum_{i,j} \matrice_{i,j}
      - \frac{1}{4} \min\bset{ \uncertain\transp\matrice\uncertain }%
      { \uncertain_{1}^2=1,\ldots,\uncertain_{\spacedim}^2=1 }
     \\
    &= 
      \frac{1}{4} \sum_{i,j} \matrice_{i,j} - \frac{1}{4} \widetilde{\LinearQuadratic}\bp{\np{1,\ldots,1}}
      \intertext{where the function ~\( \widetilde{\LinearQuadratic}\) is given by~\eqref{eq:ConditionalInfimum_quadratic_convex_LB}
      by Theorem~\ref{th:ConditionalInfimum_quadratic-example} since
      the pair $(\matrice, 0)$ can be put in block-signed form}      
    &= 
      \frac{1}{4} \sum_{i,j} \matrice_{i,j}
      - \frac{1}{4} \bp{ \sum_{i=1}^{\spacedim} \matrice_{i,i} - \sum_{i\neq j} \module{\matrice_{i,j}} }
      \tag{by~\eqref{eq:ConditionalInfimum_quadratic_convex_LB}}
    \\
    &=
      \frac{1}{4} \sum_{i\neq j}\bp{ \matrice_{i,j} + \module{\matrice_{i,j}}}
      \nonumber
    \\
    &=
      \frac{1}{4} \sum_{i\in P, j \in M} 2\matrice_{i,j}
      +\frac{1}{4} \sum_{j\in P, i \in M} 2\matrice_{i,j}
      \intertext{as $\matrice_{i,j}\leq 0$, hence \(\matrice_{i,j} + \module{\matrice_{i,j}}=0\), when both indices are in either~$P$
      or~$M$ by~\eqref{eq:block-signed_form_math} or
      \eqref{eq:block-signed_form}, and also in the case where either
      \(P=\emptyset\), \(M=\ic{1,\spacedim}\) or \(P=\ic{1,\spacedim}\),
      \(M=\emptyset\), in which case \( \sum_{i\in P, j \in M} \) and \(
      \sum_{j\in P, i \in M} \) sum to zero}
    &=
      \sum_{i\in P, j \in M} \matrice_{i,j}
 \tag{as \( \matrice \) is a symmetric matrix}
      \eqfinv
  \end{align*}
with the convention that \( \sum_{i\in P, j \in M} \matrice_{i,j} =0\) if one of
the sets~$P$ or~$M$ is empty.

  This ends the proof.  
\end{proof}

Thus, when the pair $(\matrice, 0)$ can be put in block-signed form, we obtain
with Corollary~\ref{co:weighted-max-cut} and explicit solution of the weighted
max-cut problem.
As checking that a pair can be put in block-signed form is obtained by a
$O(n +m)$~algorithm (see Algorithm~\ref{tts:alg:2tssdp}) --- where $m$ is the
number of arcs in the graph associated with~$\matrice$ --- we have obtained a new
class for which the weighted max-cut problem is solved in $O(n +m)$, which gives
$O(n^2)$ in the worst case.

Assuming that the pair $(\matrice, 0)$ can be put in block-signed form, the
weighted max-cut problem for the matrix $\matrice$ has a value given by
Equation~\eqref{eq:maxcut_val} which only depends on the nonnegative value of
the matrix $\matrice$. Thus we obtain, that the weighted max-cut problem for the
matrix $\matrice_{+}= \max(\matrice,0)$ --- which is easily seen to be also
block-signed with the same~$(P,M)$ ---
has the same value. Moreover, the weighted graph $(\VERTEX,\matrice_{+})$ is bipartite and a
possible partition of vertices is given by $(P,M)$.

However, proving that the graph $(\VERTEX,\matrice_{+})$ is bipartite and finding a
partition~$(P,M)$ is not enough to solve the weighted max-cut problem for the matrix $\matrice$ when the matrix
$\matrice$ is not block-signed. Consider $\matrice$ given by
\begin{equation}
  \matrice=
  \left(
    \begin{array}{cccc}
      0 & 1 & -1 & -1 \\
      1 & 0 &  0 &  0 \\
      -1& 0 &  0 &  1 \\
      -1& 0 &  1 &  0
    \end{array}
    \right)\eqfinp
\end{equation}
We obtain that the graph $(\VERTEX,\matrice_{+})$ is bipartite and its associated max-cut value
is $2$ obtained by the partition $\ba{\na{1,3},\na{2,4}}$. However for the graph $(\VERTEX,\matrice)$ we
obtain a max-cut value of $1$ which is obtained by the three possible partitions
$\ba{\na{1,3},\na{2,4}}$, $\ba{\na{1,4},\na{2,3}}$ and $\ba{\na{2},\na{1,3,4}}$.

Finally, it is worth noting that it is easy to generate random instances of
block-signed matrices and thus, using Corollary~\ref{co:weighted-max-cut}, it is easy to
generate large weighted max-cut instances for which the optimal value is known. This could be of
interest for numerically evaluating the performances of max-cut heuristics.




\section{Conditional infimum and S-procedure}
\label{Conditional_infimum_and_the_S-procedure}


In~\S\ref{The_S-procedure_revisited_with_the_conditional_infimum}, we briefly
present the so-called S-procedure (see the survey
paper~\cite{Polik-Terlaky:2007}).  Then, we reformulate it by means of the
conditional infimum.
Finally, in~\S\ref{The_S-procedure_for_quadratic_functions}, we provide
conditions for the S-procedure to be valid for quadratic functions,
using results obtained in Sect.~\ref{Hidden_convexity_in_quadratic_optimization_problems}.

\subsection{The S-procedure revisited with the conditional infimum}
\label{The_S-procedure_revisited_with_the_conditional_infimum}



\subsubsubsection{Definition of the classic S-procedure}

Let $\DEPART$ be a nonempty set, 
$ \fonctiondepart^{0}, \fonctiondepart^{1}, \ldots, \fonctiondepart^{\constraintdim} \colon \DEPART
\to \RR $ be real-valued functions, and consider the statements
\begin{subequations}
  \begin{align}
    \textrm{(I)}\qquad
    &
      \bp{ \fonctiondepart^{\LocalIndex}\np{\depart} \geq 0 \eqsepv \forall \LocalIndex\in\ic{1,\constraintdim} }
      \implies
      \fonctiondepart^{0}\np{\depart} \geq 0 
      \eqfinv
      \label{eq:classic(I)}
    \\
    \textrm{(C)}\qquad
    &
      \exists\, \alpha^{1} \geq 0, \ldots, \alpha^{\constraintdim} \geq 0
      \mtext{ such that }\fonctiondepart^{0} -
      \sum_{\LocalIndex=1}^{\constraintdim}  \alpha^{\LocalIndex}\fonctiondepart^{\LocalIndex} \geq 0
      \eqfinp
      \label{eq:classic(C)}
  \end{align}
  \label{eq:classic_S-procedure}
\end{subequations}
It is obvious that (C) $\implies$ (I). The classic S-procedure consists in finding
sufficient conditions to ensure that (I) $\implies$ (C), that is, to obtain 
(I) + \emph{suitable conditions} $\implies$ (C).

\subsubsubsection{The classic S-procedure revisited with the conditional infimum}

The proof of the following
Proposition~\ref{pr:the_classic_S-procedure_is_valid} will be given
in~\S\ref{Proof_of_Proposition_pr:the_classic_S-procedure_is_valid}.
It uses the notion of {closed convex hull}~$\closedconvexhull f$ of a
function~$f$ as recalled in~\eqref{eq:closed_convex_hull_of_a_function},
and of {conic hull}~$f_c$ as recalled in~\eqref{eq:conic_hull} (see also~\eqref{eq:1-homogeneous_envelope}).

\begin{proposition}
  \label{pr:the_classic_S-procedure_is_valid}
  The following statements are equivalent:
  \begin{enumerate}
  \item
    \label{it:the_classic_S-procedure_is_valid}
    the {classic S-procedure}
    is valid, that is, (I)~$\implies$~(C) in~\eqref{eq:classic_S-procedure},
  \item
    \label{it:the_classic_S-procedure_is_valid_RR_++^constraintdim_StrictEpigraph}
    \( \closedconvexhull\Bp{
      \BInfCond{\fonctiondepart^{0}}{\Epigraph_{\RR_{+}^{\constraintdim}}
        \np{-\fonctiondepart^{1}, \ldots,-\fonctiondepart^{\constraintdim} }}_{c}}\np{0} > -\infty \),
  \item
    \label{it:the_classic_S-procedure_is_valid_StrictEpigraph}  
    \( \closedconvexhull\Bp{
      \BInfCond{\fonctiondepart^{0}}{\np{-\fonctiondepart^{1},
          \ldots,-\fonctiondepart^{\constraintdim} }}_{c} \UppPlus \Indicator{\RR_{-}^{\constraintdim}} }\np{0} > -\infty \).
  \end{enumerate}
\end{proposition}

\subsection{The S-procedure for quadratic functions}
\label{The_S-procedure_for_quadratic_functions}

Now, we specialize Proposition~\ref{pr:the_classic_S-procedure_is_valid} to the
quadratic case, using the study of the conditional infimum of a quadratic
function \wrt\ the square mapping, done
in~\S\ref{Conditional_infimum_of_quadratic_functions}. 

\begin{proposition}
  \label{pr:the_classic_S-procedure_is_valid_quadratic}
  Let $\bar{\MatriceObjective}$ be a $\spacedim\times\spacedim$ matrix,
  \( \bar{\VecteurObjective} \in \RR^{\spacedim} \) be a vector,
  $\bar\scalaire\in \RR$ be a scalar,
  and $\bar{\MatriceConstraint}^{1}, \ldots, \bar{\MatriceConstraint}^{\constraintdim}$ be
  $\spacedim\times\spacedim$ diagonal matrices.
  We define the quadratic function
  \begin{subequations}
    \begin{align}
      \fonctiondepart^{0}\np{\depart} 
      &=
        \depart\transp\bar{\MatriceObjective}\depart + \depart\transp\bar{\VecteurObjective}  + \bar{\scalaire}
        \eqsepv \forall \depart \in \RR^{\spacedim}
        \eqfinv    
        \label{eq:the_classic_S-procedure_is_valid_quadratic_data^{0}}
        \intertext{the quadratic forms}
        \fonctiondepart^{\LocalIndex}\np{\depart} 
      &=
        \depart\transp\bar{\MatriceConstraint}^{\LocalIndex}\depart
        = \sum_{\LocalIndexbis\in\ic{1,\spacedim}}
        \bar{\MatriceConstraint}^{\LocalIndex}_{\LocalIndexbis,\LocalIndexbis}\depart_{\LocalIndexbis}^2
        \eqsepv \forall \depart \in \RR^{\spacedim}
        \eqsepv \forall \LocalIndex\in\ic{1,\constraintdim}
        \eqfinv    
        \label{eq:the_classic_S-procedure_is_valid_quadratic_data_constraintdim}
        \intertext{and the $p\times\spacedim$ matrix $\bar{\MatriceConstraint}$ with}
        \bar{\MatriceConstraint}_{i,j}
      &=
        \bar{\MatriceConstraint}^{\LocalIndex}_{\LocalIndexbis,\LocalIndexbis}
        \eqsepv
        \forall \LocalIndex\in\ic{1,\constraintdim} \eqsepv \forall \LocalIndexbis\in\ic{1,\spacedim}
        \eqfinp 
        \label{eq:the_classic_S-procedure_is_valid_quadratic_data_MatriceConstraint}
    \end{align}
    \label{eq:the_classic_S-procedure_is_valid_quadratic_data}
  \end{subequations}
  Suppose that
  \begin{itemize}
  \item
the (criterion) quadratic function~\( \fonctiondepart^{0} \) 
  in~\eqref{eq:the_classic_S-procedure_is_valid_quadratic_data^{0}}
  either
  is a quadratic form (that is, \( \bar{\scalaire}=0 \) and \( \bar{\VecteurObjective}=0 \))
  or satisfies \( \fonctiondepart^{0}\np{0}>0 \) 
  (that is, \( \bar{\scalaire}> 0 \)),
\item
the (constraints) matrix $\bar{\MatriceConstraint}$ either is invertible or
  satisfies
  $\bar{\MatriceObjective}_{\LocalIndexbis,\LocalIndexbis} > 0$ for all $\LocalIndexbis \in \ic{1,\spacedim}$.
  \end{itemize}
%
  Then, the {classic S-procedure} is valid for the functions
  \( \fonctiondepart^{0}, \fonctiondepart^{1}, \ldots, \fonctiondepart^{\constraintdim} \colon
  \RR^{\spacedim} \to \RR \)
  in~\eqref{eq:the_classic_S-procedure_is_valid_quadratic_data^{0}}--\eqref{eq:the_classic_S-procedure_is_valid_quadratic_data_constraintdim}.
\end{proposition}

\begin{proof}
  We introduce the function
  \( \widetilde{\fonctiondepart}^{0} \colon \RR^{\spacedim} \to \RR \cup \na{+\infty} \) defined by
  (see~\eqref{eq:ConditionalInfimum_quadratic_convex_LB})
  \begin{equation}
    \label{eq:ConditionalInfimum_quadratic-example_bis}
    \widetilde{\fonctiondepart}^{0} = 
    \begin{cases}
      +\infty & \text{ if } \depart\not\in \RR_+^{\spacedim}
           \eqfinv
      \\
      \displaystyle 
      \sum_{\LocalIndexbis=1}^{\spacedim} \bar{\MatriceObjective}_{\LocalIndexbis,\LocalIndexbis} \depart_{\LocalIndexbis} 
      -
      \sum_{\LocalIndexbis\neq \LocalIndexter} \module{\bar{\MatriceObjective}_{\LocalIndexbis,\LocalIndexter}}\sqrt{\depart_{\LocalIndexbis} \depart_\LocalIndexter}
      -
      \sum_{\LocalIndexbis=1}^{\spacedim} \module{\bar{\VecteurObjective}_{\LocalIndexbis}} \sqrt{\depart_{\LocalIndexbis} }
      + \bar{\scalaire}
         & \text{ if } \depart\in \RR_+^{\spacedim} 
           \eqfinv
    \end{cases}      
  \end{equation}
  and the functions \( \fonctionprimalbis \colon \RR^{\constraintdim} \to \barRR \),
  \( \bar\fonctionprimalbis \colon \RR^{\constraintdim} \to \barRR \),
  defined by 
  \begin{subequations}
    \begin{align*}
      \fonctionprimalbis
      &=
        \InfCond{\fonctiondepart^{0}}{\np{-\fonctiondepart^{1},
        \ldots,-\fonctiondepart^{\constraintdim} }}
        \UppPlus \Indicator{\RR_{-}^{\constraintdim}}
        \tag{see Item~\ref{it:the_classic_S-procedure_is_valid_StrictEpigraph}
        in Proposition~\ref{pr:the_classic_S-procedure_is_valid}}
      \\
      &=
        \InfCond{\fonctiondepart^{0}}{-\bar{\MatriceConstraint}\SquareMapping}
        \UppPlus \Indicator{\RR_{-}^{\constraintdim}}    
        \intertext{
        by definitions~\eqref{eq:the_classic_S-procedure_is_valid_quadratic_data_constraintdim}
        of \( \fonctiondepart^{1}, \ldots, \fonctiondepart^{\constraintdim} \),
        \eqref{eq:the_classic_S-procedure_is_valid_quadratic_data_MatriceConstraint}
        of~\( \bar{\MatriceConstraint} \), and~\eqref{eq:SquareMapping}
        of the square mapping~\( \SquareMapping \)}
      &=
        \bInfCond{ \InfCond{\fonctiondepart^{0}}{\SquareMapping} }{-\bar{\MatriceConstraint}}
        \UppPlus \Indicator{\RR_{-}^{\constraintdim}}
        \tag{by the tower property~\eqref{eq:tower_property}}
      \\
      &\geq
        \InfCond{ \widetilde{\fonctionprimal}^0 }{-\bar{\MatriceConstraint}}
        \UppPlus \Indicator{\RR_{-}^{\constraintdim}} 
        \intertext{as \(\InfCond{\fonctiondepart^{0}}{\SquareMapping}\geq\widetilde{\fonctionprimal}^0 \)
        by~\eqref{eq:ConditionalInfimum_quadratic_convex_LB_inequality},
        and as the conditional infimum preserves inequalities by~\eqref{eq:correspondence_conditional_infimum_properties_monotonicity}}
      &:=
        \bar\fonctionprimalbis
        \tag{a definition}
        \eqfinp
    \end{align*}
  \end{subequations}
  Then, from the just obtained inequality
  \(\fonctionprimalbis\geq\bar\fonctionprimalbis\),
  we deduce that
  \begin{align} \fonctionprimalbis_{c}
    &\geq \bar\fonctionprimalbis_{c}\nonumber \\
    &= \Bp{\InfCond{ \tilde{\fonctionprimal}^0 }{-\bar{\MatriceConstraint}} \UppPlus
      \Indicator{\RR_{-}^{\constraintdim}}}_{c}
      \tag{by definition of $\bar\fonctionprimalbis$}\\
    &= \Bp{\InfCond{ \tilde{\fonctionprimal}^0 }{-\bar{\MatriceConstraint}}}_{c} \UppPlus
      \Indicator{\RR_{-}^{\constraintdim}}\nonumber \\
    &= \InfCond{ \tilde{\fonctionprimal}^0_{c} }{-\bar{\MatriceConstraint}} \UppPlus
      \Indicator{\RR_{-}^{\constraintdim}}
      \eqfinv
      \tag{by~\eqref{eq:conic_hull_of_correspondence_conditional_supremum_infimum}}
  \end{align}
  hence that \( \closedconvexhull \fonctionprimalbis_{c}\np{0}
  \geq 
  \closedconvexhull\bar\fonctionprimalbis_{c}\np{0} \).
  As the function
  \( \tilde{\fonctionprimal}^0 \colon \RR^{\spacedim} \to \RR \cup \na{+\infty} \) defined
  by~\eqref{eq:ConditionalInfimum_quadratic-example_bis} is convex, by
  Theorem~\ref{th:Conditional_infimum_of_a_quadratic_function_knowing_squares},
  then so is the function
  \( \tilde{\fonctionprimal}^0_{c} \colon \RR^{\spacedim} \to \barRR \)
  by~\eqref{eq:StrictEpigraph_conic_hull}, and so is the function
  \( \InfCond{ \tilde{\fonctionprimal}^0_{c} }{-\bar{\MatriceConstraint}} \colon
  \RR^{\spacedim} \to \barRR \) by Corollary~\ref{cor:composition_convexity}.
  Therefore, our strategy of proof is to determine when the function
  \( \InfCond{ \tilde{\fonctionprimal}^0_{c} }{-\bar{\MatriceConstraint}} \) is \lsc, 
  then evaluate
  \(
  \closedconvexhull\bar\fonctionprimalbis_{c}\np{0}=\bar\fonctionprimalbis_{c}\np{0}
  \), and finally check when
  \( \bar\fonctionprimalbis_{c}\np{0} > -\infty \).
  Indeed, in that case, we will get that
\(   \closedconvexhull\Bp{
    \BInfCond{\fonctiondepart^{0}}{\np{-\fonctiondepart^{1},
    \ldots,-\fonctiondepart^{\constraintdim} }}_{c} \UppPlus \Indicator{\RR_{-}^{\constraintdim}}
    }\np{0}
    =  \closedconvexhull\fonctionprimalbis_{c}\np{0}
    \geq 
    \closedconvexhull\bar\fonctionprimalbis_{c}\np{0}  > -\infty \),
    and conclude, by means of Item~\ref{it:the_classic_S-procedure_is_valid_StrictEpigraph}
    in Proposition~\ref{pr:the_classic_S-procedure_is_valid},
    that the {classic S-procedure} is valid.
    
    For this purpose, we will use the property that (see~\eqref{eq:conic_hull})
    \begin{subequations}
      \label{eq:ConditionalInfimum_quadratic-example_bis_conic_hull}
      \begin{align}
        \label{eq:ConditionalInfimum_quadratic-example_bis_conic_hull_one}
        \tilde{\fonctionprimal}^0_{c}\np{\depart}
        &= 
          \begin{cases}
            +\infty & \text{ if } \depart\not\in \RR_+^{\spacedim}
                 \eqfinv
            \\
            \displaystyle 
            \sum_{\LocalIndexbis=1}^{\spacedim} \bar{\MatriceObjective}_{\LocalIndexbis,\LocalIndexbis} \depart_{\LocalIndexbis} 
            -
            \sum_{\LocalIndexbis\neq \LocalIndexter} \module{\bar{\MatriceObjective}_{\LocalIndexbis,\LocalIndexter}}\sqrt{\depart_{\LocalIndexbis} \depart_\LocalIndexter}
            +\inf_{\lambda > 0}   \lambda \Bp{ -\frac{1}{\sqrt{\lambda}}
            \sum_{\LocalIndexbis=1}^{\spacedim} \module{\bar{\VecteurObjective}_{\LocalIndexbis}} \sqrt{\depart_{\LocalIndexbis} }
            + \bar{\scalaire} }
               & \text{ if } \depart\in \RR_+^{\spacedim} 
                 \eqfinv
          \end{cases}
          \intertext{ hence, if \( \bar{\scalaire} = 0 \) and \( \bar{\VecteurObjective}=0 \), we get  that }
          \label{eq:ConditionalInfimum_quadratic-example_bis_conic_hull_two}
          \tilde{\fonctionprimal}^0_{c}\np{\depart}
        &=
          \begin{cases}
            +\infty & \text{ if } \depart\not\in \RR_+^{\spacedim}
                 \eqfinv
            \\
            \displaystyle 
            \sum_{\LocalIndexbis=1}^{\spacedim} \bar{\MatriceObjective}_{\LocalIndexbis,\LocalIndexbis} \depart_{\LocalIndexbis} 
            -
            \sum_{\LocalIndexbis\neq \LocalIndexter} \module{\bar{\MatriceObjective}_{\LocalIndexbis,\LocalIndexter}}\sqrt{\depart_{\LocalIndexbis} \depart_\LocalIndexter}
               & \text{ if } \depart\in \RR_+^{\spacedim} 
                 \eqfinv
          \end{cases}          
          \intertext{and if \( \bar{\scalaire} > 0 \), we get that}
          \label{eq:ConditionalInfimum_quadratic-example_bis_conic_hull_three}
          \tilde{\fonctionprimal}^0_{c}\np{\depart}
        &=
          \begin{cases}
            +\infty & \text{ if } \depart\not\in \RR_+^{\spacedim}
                 \eqfinv
            \\
            \displaystyle 
            \sum_{\LocalIndexbis=1}^{\spacedim} \bar{\MatriceObjective}_{\LocalIndexbis,\LocalIndexbis} \depart_{\LocalIndexbis} 
            -
            \sum_{\LocalIndexbis\neq \LocalIndexter} \module{\bar{\MatriceObjective}_{\LocalIndexbis,\LocalIndexter}}\sqrt{\depart_{\LocalIndexbis} \depart_\LocalIndexter}
            -
            \frac{ \bp{ \sum_{\LocalIndexbis=1}^{\spacedim} \module{\bar{\VecteurObjective}_{\LocalIndexbis}} \sqrt{\depart_{\LocalIndexbis} } }^2 }%
            { 4 \bar{\scalaire} }
               & \text{ if } \depart\in \RR_+^{\spacedim} 
                 \eqfinv
          \end{cases}          
      \end{align}
    \end{subequations}
    We conclude that, in the two considered cases, the function~$\tilde{\fonctionprimal}^0_{c}$ is \lsc .
    
    Now, for any $\arrivee\in\RR^{\constraintdim}$, we have that
    \begin{align*}
      \bar\fonctionprimalbis_{c}\np{\arrivee}
      &=\InfCond{ \tilde{\fonctionprimal}^0_{c} }{-\bar{\MatriceConstraint}}\np{\arrivee}
        \UppPlus \Indicator{\RR_{-}^{\constraintdim}}\np{\arrivee}
      \\
      &=
      \inf\defset{\tilde{\fonctionprimal}^0_{c} \np{\depart}}{\depart\in \RR_+^{\spacedim}
        \mtext{ and } \bar{\MatriceConstraint}\depart=-\arrivee}
        \UppPlus \Indicator{\RR_{-}^{\constraintdim}}\np{\arrivee}
        \tag{by~\eqref{eq:ConditionalInfimum_quadratic-example_bis_conic_hull}
        \( \depart\not\in \RR_+^{\spacedim} \implies
        \tilde{\fonctionprimal}^0_{c} \np{\depart}=+\infty \)}
      \\
    &= \inf_{u} \tilde{\fonctionprimal}^0_{c}(u)
      + \Indicator{\defset{(u,v)\in\RR_+^{\spacedim}\times\RR_{-}^{\constraintdim} }{{\bar{\MatriceConstraint}}u=-v }}
      \eqfinp
    \end{align*}
    We have to prove that $\bar\fonctionprimalbis_{c}(0) > -\infty$ and that the function~$\bar\fonctionprimalbis_{c}$ is \lsc\ at $0$ in
    the two assumed possible cases regarding the (constraints) matrix~$\bar{\MatriceConstraint}$.
    First, we assume that the matrix $\bar{\MatriceConstraint}$ is invertible. Then,
  the function~$\bar\fonctionprimalbis_{c}$ reduces to
  \begin{equation}
    \bar\fonctionprimalbis_{c}\np{\arrivee} = \inf_{u} F(u,v) =
    \tilde{\fonctionprimal}^0_{c} \bp{\bar{\MatriceConstraint}^{-1}(-v)}
    \UppPlus \Indicator{\RR_+^{\spacedim}}\bp{{{\bar{\MatriceConstraint}}^{-1}(-v) }}
    \UppPlus \Indicator{\RR_{-}^{\constraintdim} }(v)
    \eqfinv
  \end{equation}
  hence we have that $\bar\fonctionprimalbis_{c}(0)=0$, and that the function~$\bar\fonctionprimalbis_{c}$ is \lsc\ at $0$.

  Second, we assume that, for all $\LocalIndexbis \in \ic{1,\spacedim}$,
  $\bar{\MatriceObjective}_{\LocalIndexbis,\LocalIndexbis} > 0$.  Then, the
  function $\tilde{\fonctionprimal}^0_{c}$ is coercive
  by~\eqref{eq:ConditionalInfimum_quadratic-example_bis_conic_hull}.
  We deduce that, on the one hand, 
  $\bar\fonctionprimalbis_{c}(0) > -\infty$ as the infimum is achieved,
  and that, on the other hand, the function~$\bar\fonctionprimalbis_{c}$ is \lsc\ at~$0$
  using~\cite[Theorem~1.17 p.~16]{Rockafellar-Wets:1998}.

  Finally, in the two cases, we obtain that
  \(
  \closedconvexhull\fonctionprimalbis_{c}\np{0}
  \geq
  \closedconvexhull\bar\fonctionprimalbis_{c}\np{0}=\bar\fonctionprimalbis_{c}\np{0} > -\infty \).
  We conclude, by means of Item~\ref{it:the_classic_S-procedure_is_valid_StrictEpigraph}
  in Proposition~\ref{pr:the_classic_S-procedure_is_valid},
  that the {classic S-procedure} is valid.
\end{proof}

\section{Conclusion}
\label{Conclusion}

Detecting hidden convexity is one of the tools to address nonconvex
minimization problems.
In this paper, we have contributed to this research program by
putting forward the notion of conditional infimum. 
Building upon a well-known parallelism between optimization and probability
theories, we have established a list of properties of the conditional infimum,
among which a tower formula, relevant for minimization problems. 
Thus equipped, we have provided a new sufficient condition for 
hidden convexity in nonconvex quadratic minimization problems,
as well as new sufficient conditions for the S-procedure.
We have also pointed out perspectives in obtaining lower bound convex programs,
using a new class of one-sided homogeneous sublinear couplings.
\bigskip


\appendix

\section{More on block-signed form}
\label{Appendix}

\subsubsubsection{Another characterization of block-signed form (see Definition~\ref{de:block-signed form})}

\begin{proposition}
  \label{pr:varepsilon_sign_appendix}
  Let $\spacedim \in \NN^*$ be a positive integer,
  $\vecteur \in \RR^{\spacedim}$ be a vector, and $\matrice$ be a
  $\spacedim{\times}\spacedim$ symmetric matrix.
  The following assertions are equivalent.
  \begin{subequations}
    \begin{enumerate}
    \item
      \label{it:varepsilon_appendix}
      There exists \( \varepsilon=\np{\varepsilon_{1},\ldots,\varepsilon_{\spacedim}} \in \na{-1,1}^{\spacedim} \)
      such that
      \begin{equation}
        \begin{cases}
          \varepsilon_i\vecteur_i \leq 0 \eqsepv & \forall i=1,\ldots,\spacedim 
                                      \eqfinv
          \\
          \text{and} &
          \\
          \varepsilon_i\varepsilon_j\matrice_{i,j}  \leq 0 \eqsepv 
                                    & 
                                      \forall i,j=1,\ldots,\spacedim \eqsepv i\neq j
                                      \eqfinp 
        \end{cases}
        \label{eq:varepsilon_appendix}
      \end{equation}
    \item
      \label{it:varepsilon_sign_bis_appendix}
      There exists \( \varepsilon=\np{\varepsilon_{1},\ldots,\varepsilon_{\spacedim}} \in \na{-1,1}^{\spacedim} \)
      such that
      \begin{equation}
        \begin{cases}
          \sign{\vecteur_i} \in \na{-\varepsilon_i,0} \eqsepv & \forall i=1,\ldots,\spacedim 
                                                    \eqfinv
          \\
          \text{and} &
          \\
          \sign{\matrice_{i,j}} \in \na{-\varepsilon_i\varepsilon_j,0} \eqsepv 
                                                  & 
                                                    \forall i,j=1,\ldots,\spacedim \eqsepv i\neq j
                                                    \eqfinp 
        \end{cases}
        \label{eq:varepsilon_sign_bis_appendix}
      \end{equation}

    \item
      The pair \( \np{\matrice,\vecteur} \) can be put in {block-signed} form
      (see Definition~\ref{de:block-signed form}), and
      \( \varepsilon=\np{\varepsilon_{1},\ldots,\varepsilon_{\spacedim}} \in \na{-1,1}^{\spacedim} \)
      is given by Definition~\ref{de:block-signed form}.
    \end{enumerate}
  \end{subequations}
\end{proposition}

\begin{proof}
  The equivalence between Item~\ref{it:varepsilon_appendix}
  and Item~\ref{it:varepsilon_sign_bis_appendix} is obvious.
  Item~\ref{it:varepsilon_sign_bis_appendix} is a rephrasing of
  Definition~\ref{de:block-signed form}.
\end{proof}

\subsubsubsection{A greedy $O(n+m)$ algorithm to obtain $\epsilon$ satisfying Equation~\eqref{eq:varepsilon}}

We consider a simple labeled graph $G=(V,\matrice)$ with vertices
$V=\ic{1,\spacedim}$ and adjacency matrix
$A^d_{i,j}=\abs{\mathrm{sign}(\matrice_{i,j})}$ for $i\not=j$ and $A^d_{i,i}=0$.
Thus, there is an edge between vertex $i$ and vertex $j$ if $i\not=j$ and
$\matrice_{i,j}\not=0$ and the label of the edge is
$\mathrm{sign}(\matrice_{i,j})\in \na{-1,1}$. The number of edges of the graph is $m$.

Thus, finding a vector $\epsilon$ satisfying
$\varepsilon_i\varepsilon_j\matrice_{i,j} \leq 0$ forall $i\not=j$ in $V$, amounts to finding a function
$\phi \colon \ic{1,\spacedim} \to \na{-1,1}$ such that $\phi(i)=\phi(j)$ if there is an edge
joining $i$ and $j$ with a label $-1$ and $\phi(i)\not=\phi(j)$ if there is an edge
joining $i$ and $j$ with a label $1$.

Now, let $v \in V$ be given and denote by $C(v,G)$ the connected component of
$G$ which contains the vertex $v$.  If we fix the value of $\phi(v)$, then a walk
on the subgraph associated to the connected component $C(v,G)$ starting from
vertex $v$, will fix all the values of $\phi(c')$ for $c'\in C(v,G)$ or will find an
incompatibility rejecting the value $\phi(v)$ for vertex $v$.

Then, the following greedy algorithm will find a vector $\epsilon$
satisfying~\eqref{eq:varepsilon} or will return that there is no
solution. First, we consider the subset of indices
$I=\nset{i \in \ic{1,\spacedim}}{b_i\not=0}$ for which the values
$\nset{\phi(i)}{i\in I}$ are imposed by the nonzero values in the components of
vector $b$. Now, for each $i \in I$, we explore $C(i,G)$ and fix the values of
$\phi$ on $C(i,G)$ or stop if an incompatiblity is met.

Second, until all the values of $\phi$ are fixed or an incompatibility is found, we
loop as follows.  We pick a vertex $v\in G$, for which the value of $\phi(v)$ is not
set (indeed, we are done if they are all set).  We test successively the two
possible values $\phi(v)=\pm 1$, and explore as above $C(v,G)$ and stop in case of
incompatibility.  The algorithm is greedy as fixing the values of $\phi$ on a
connected component does not affect the remaining connected components. Thus a
successful choice for $\phi$ on a vertex $v$ does not need to be questioned later
in the algorithm. The complexity of the algorithm is similar to finding the
connected component of the graph~$G$ which is known to be of complexity
$O(n+m)$.

\begin{algorithm}[H]
  \KwResult{$\epsilon$ solution or $\mathrm{Fail}$ if there is no solution}
  \textbf{Initialization}: $\epsilon = 0$ \\
  \While{$(\exists i, \epsilon_i = 0)$}
  {
    $K \leftarrow \text{\bf{if} } (b_i \not= 0) \text{ \bf{then} } \na{-b_i} \text{ \bf{else} } \na{-1,1}$ \\
    \For{$k \in K$}
    {
      $\epsilon_i \leftarrow k$ \textbf{;} $(\epsilon,R) \leftarrow \texttt{ExploreC}(i, \epsilon, A, b)$  \\
      \lIf{$(R = \mathrm{Success})$}{\textbf{break};\texttt{// quit the loop}}
    }
    \lIf{$(R = \mathrm{Fail})$}{\Return $(\epsilon,\mathrm{Fail})$}
  }
  
  \SetKwFunction{Fun}{ExploreC}
  \SetKwProg{Fn}{Function}{:}{}
  \Fn{\Fun{$i, \epsilon, A, b$}}
  {
    \For{$j \in \nset{j \in \ic{1,n}}{j\not=i \wedge A_{i,j} \not=0}$}%
    {
      $\overline{\epsilon} \leftarrow (- \epsilon_i {A_{i,j}})$\\
      \lIf{$(\epsilon_j \not=0 \wedge \epsilon_j \not= \overline{\epsilon}) \vee ((b_j \not= 0) \wedge \overline{\epsilon} \not= -b_j)$}{\Return $(\epsilon,\mathrm{Fail})$}
      $\epsilon_j \leftarrow \overline{\epsilon}$ \textbf{;}$(\epsilon,R) \leftarrow \texttt{ExploreC}(j, \epsilon, A, b)$ \\
      \lIf{$(R = \mathrm{Fail})$}{\Return $(\epsilon,\mathrm{Fail})$}
    }%
    \Return $(\epsilon,\mathrm{Success})$
  }
  \caption{If $R =  \mathrm{Fail}$ there is no $\epsilon$ solution to Equation~\eqref{eq:varepsilon}
    else if $R =  \mathrm{Success}$ there is a solution returned in vector $\epsilon$}
  \label{tts:alg:2tssdp}
\end{algorithm}

\section{More on the general S-procedure}


In~\S\ref{New_necessary_and_sufficient_conditions_for_the_S-procedure_to_be_valid},
we briefly present the so-called general S-procedure \cite{Volle-Barro:2024},
and we show new necessary and sufficient conditions for its validity.
Thus equipped, we prove
in~\S\ref{Proof_of_Proposition_pr:the_classic_S-procedure_is_valid} our result
exposed in~\S\ref{The_S-procedure_revisited_with_the_conditional_infimum}.

\subsection{Validity of the general S-procedure}
\label{New_necessary_and_sufficient_conditions_for_the_S-procedure_to_be_valid}

\subsubsubsection{Definition of the general S-procedure}

We outline the \emph{general S-procedure} as introduced in~\cite{Volle-Barro:2024}.
Let \( \DEPART \) be a nonempty set,
\( \PRIMAL, \DUAL \) be paired\footnote{%
  In fact, in~\cite{Volle-Barro:2024} \( \PRIMAL \) is a locally convex Hausdorff topological vector space
  and \( \DUAL=\PRIMAL' \) is the topological dual vector space of~\( \PRIMAL \) with the standard bilinear 
  form \( \nscal{\primal}{\dual} = \dual\np{\primal} \). 
}
vector spaces (see~\S\ref{Conditional_infimum_and_one-sided_homogeneous_sublinear_couplings}),
\( \fonctiondepart \colon \DEPART \times \PRIMAL \to \barRR \) be a function,
and consider the statements
\begin{subequations}
  \begin{align}
    \textrm{(I)}\qquad
    &
      \fonctiondepart\np{\depart,0} \geq 0 \eqsepv \forall \depart\in\DEPART 
      \eqfinv
      \label{eq:general(I)}
    \\
    \textrm{(C)}\qquad
    &
      \exists\, \dual\in\DUAL \eqsepv \fonctiondepart\np{\depart,\primal}
      \geq \nscal{\primal}{\dual} \eqsepv \forall \depart\in\DEPART \eqsepv \forall \primal\in\PRIMAL
      \eqfinp
      \label{eq:general(C)}
  \end{align}
\end{subequations}
We have that (C) $\implies$ (I) because \( \fonctiondepart\np{\depart,0} \geq \nscal{0}{\dual} = 0 \).
The \emph{S$_{\fonctionprimal}$-procedure} is \emph{valid} \cite[Definition~1.1]{Volle-Barro:2024}
when there are sufficient conditions which ensure that (I) $\implies$ (C), that is, to obtain 
(I) + \emph{suitable conditions} $\implies$ (C).

\subsubsubsection{Recalls on characterization of the validity of the {S$_\fonctionprimal$-procedure}}

We recall the approach in \cite[Sect.~3]{Volle-Barro:2024} that characterizes
the validity of the {S$_\fonctionprimal$-procedure} in terms of the value function
\( \fonctionprimalbis \colon \PRIMAL \to \barRR \) defined by
\begin{equation}
  \fonctionprimalbis\np{\primal}
  =
  \inf_{\depart \in \DEPART} \fonctiondepart\np{\depart,\primal} 
  \eqsepv \forall \primal\in\PRIMAL
  \eqfinp
  \label{eq:S_fonctionprimal-procedure_value_function} 
\end{equation}


\noindent\emph{\textbf{Excerpt from \cite[Theorem~3.1]{Volle-Barro:2024}}
  The following statements are equivalent\footnote{%
    In \cite[Theorem~3.1]{Volle-Barro:2024}, the statement is done under the assumption that \( \fonctionprimalbis\np{0}\geq
    0 \), but we have not seen the role played by this assumption.}
  \begin{enumerate}
  \item
    the {S$_\fonctionprimal$-procedure} is valid,
  \item
    the value function~\( \fonctionprimalbis \) in~\eqref{eq:S_fonctionprimal-procedure_value_function}
    has a (continuous) linear minorant,
  \item
    \( \ba{ \LFM{\fonctionprimalbis} \leq 0 } \neq \emptyset \).
  \end{enumerate}
}

\subsubsubsection{New necessary and sufficient conditions for the general
  S-procedure to be valid}

The following
Proposition~\ref{pr:New_necessary_and_sufficient_conditions_for_the_S-procedure_to_be_valid}
builds upon \cite[Theorem~3.1]{Volle-Barro:2024} and
\cite[Proposition~2.1]{Volle-Barro:2024}.  We provide characterizations of when
the {S$_{\fonctionprimal}$-procedure} is valid.

\begin{proposition}
  \label{pr:New_necessary_and_sufficient_conditions_for_the_S-procedure_to_be_valid}
  The following statements are equivalent:
  \begin{enumerate}
  \item
    \label{it:the_S_fonctionprimal-procedure_is_valid}
    the {S$_{\fonctionprimal}$-procedure} is valid, 
  \item
    \label{it:ba_LFM_fonctionprimalbis_leq_{0}}
    \( \ba{ \LFM{\fonctionprimalbis} \leq 0 } \neq \emptyset\),
  \item
    \label{it:the_function_fonctionprimalbis_has_a_linear_minorant}
    the function~\( \fonctionprimalbis 
    \) has a continuous linear minorant,
    that is, there exists \( \dual\in\DUAL \) such that \( \fonctionprimalbis\np{\primal}
    \geq \nscal{\primal}{\dual} \), for all \( \primal\in\PRIMAL \),
  \item
    \label{it:the_function_LFMbifonctionprimalbis_has_a_linear_minorant}
    the function~\( \LFMbi{\fonctionprimalbis} \) has a continuous linear minorant,
  \item
    \label{it:the_function_closedconvexhull_fonctionprimalbis_has_a_linear_minorant}
    the function~\( \closedconvexhull\fonctionprimalbis \) has a
    continuous linear minorant,
  \item
    \label{it:ba_LFM_fonctionprimalbis_c_leq_{0}}
    \( \ba{ \LFM{\fonctionprimalbis}_{c} \leq 0 } \neq \emptyset\),
  \item
    \label{it:the_function_fonctionprimalbis_c_has_a_linear_minorant}
    the conic hull function~\( \fonctionprimalbis_{c} \) (see~\eqref{eq:conic_hull})
    has a continuous linear minorant,
  \item
    \label{it:the_function_LFMbifonctionprimalbis_c_has_a_linear_minorant}
    the function~\( \LFMbi{\fonctionprimalbis_{c}} \) has a
    continuous linear minorant,
  \item
    \label{it:the_function_closedconvexhull_fonctionprimalbis_c_has_a_linear_minorant}
    the closed convex conic hull function~\( \closedconvexhull\np{\fonctionprimalbis_{c}} \) has a
    continuous linear minorant,
  \item
    \label{it:closedconvexhull_RR++_StrictEpigraph_fonctionprimalbis_cap_=}
    \( \np{0,-1} \not\in \closedconvexhull\np{\RR_{++}\StrictEpigraph\fonctionprimalbis} \),
  \item
    \label{it:closedconvexhullnpfonctionprimalbis_cnp0inna0,+infty}
    \( \closedconvexhull\np{\fonctionprimalbis_{c}}\np{0} \in \na{0,+\infty} \),
  \item
    \label{it:LFMbifonctionprimalbis_cnp0inna0,+infty}
    \( \LFMbi{\fonctionprimalbis_{c}}\np{0} \in \na{0,+\infty} \).
    %
  \end{enumerate}
\end{proposition}

\begin{proof}
  
  \noindent$\bullet$   
  Item~\ref{it:the_S_fonctionprimal-procedure_is_valid}
  $\iff$ Item~\ref{it:ba_LFM_fonctionprimalbis_leq_{0}}
  $\iff$ Item~\ref{it:the_function_fonctionprimalbis_has_a_linear_minorant}
  is proven in \cite[Theorem~3.1]{Volle-Barro:2024}.
  \medskip

  \noindent$\bullet$ 
  Item~\ref{it:ba_LFM_fonctionprimalbis_leq_{0}}
  $\iff$ Item~\ref{it:the_function_LFMbifonctionprimalbis_has_a_linear_minorant}
  follows from 
  Item~\ref{it:ba_LFM_fonctionprimalbis_leq_{0}}
  $\iff$ Item~\ref{it:the_function_fonctionprimalbis_has_a_linear_minorant},
  because \( \LFM{\np{ \LFMbi{\fonctionprimalbis} }}=
  \LFM{\fonctionprimalbis} \).
  \medskip

  \noindent$\bullet$ 
  Item~\ref{it:the_function_fonctionprimalbis_has_a_linear_minorant}
  $\iff$ Item~\ref{it:the_function_LFMbifonctionprimalbis_has_a_linear_minorant}
  $\iff$ Item~\ref{it:the_function_closedconvexhull_fonctionprimalbis_has_a_linear_minorant}
  because, by the very definition~\eqref{eq:closed_convex_hull_of_a_function}
  of the function~\( \closedconvexhull\fonctionprimalbis \), we have the inequalities
  \( \LFMbi{\fonctionprimalbis} \leq 
  \closedconvexhull\fonctionprimalbis \leq
  \fonctionprimalbis \)
  (see Equation~\eqref{eq:LFMbi_fonctionprimal_leq_closedconvexhull_fonctionprimal_leq_fonctionprimal}).
  \medskip

  \noindent$\bullet$ 
  Item~\ref{it:ba_LFM_fonctionprimalbis_leq_{0}} $\iff$ 
  Item~\ref{it:ba_LFM_fonctionprimalbis_c_leq_{0}} because
  \( \LFM{\fonctionprimalbis_{c}} = \Indicator{\na{ \LFM{\fonctionprimalbis} \leq
      0 }} \) by~\eqref{eq:LFM_conic_hull},
  hence \( \ba{ \LFM{\fonctionprimalbis_{c}} \leq 0 } = \ba{ \LFM{\fonctionprimalbis} \leq 0 } \).
  \medskip

  \noindent$\bullet$
  Item~\ref{it:ba_LFM_fonctionprimalbis_c_leq_{0}} 
  $\iff$ Item~\ref{it:the_function_fonctionprimalbis_c_has_a_linear_minorant}
  $\iff$ Item~\ref{it:the_function_LFMbifonctionprimalbis_c_has_a_linear_minorant}
  $\iff$ Item~\ref{it:the_function_closedconvexhull_fonctionprimalbis_c_has_a_linear_minorant}
  is proved in the same way as we have proved that 
  Item~\ref{it:ba_LFM_fonctionprimalbis_leq_{0}} 
  $\iff$ Item~\ref{it:the_function_fonctionprimalbis_has_a_linear_minorant}
  $\iff$ Item~\ref{it:the_function_LFMbifonctionprimalbis_has_a_linear_minorant}
  $\iff$ Item~\ref{it:the_function_closedconvexhull_fonctionprimalbis_has_a_linear_minorant}.

  
  \medskip

  \noindent$\bullet$
  We prove that Item~\ref{it:the_function_closedconvexhull_fonctionprimalbis_c_has_a_linear_minorant}
  $\implies$ Item~\ref{it:closedconvexhull_RR++_StrictEpigraph_fonctionprimalbis_cap_=}.
  %
  The proof is by contradiction.
  Suppose that \( \np{0,-1} \in \closedconvexhull\np{\RR_{++}\StrictEpigraph\fonctionprimalbis}\).
  As the function~\( \closedconvexhull\np{\fonctionprimalbis_{c}} \) has a linear minorant,
  by Item~\ref{it:the_function_closedconvexhull_fonctionprimalbis_c_has_a_linear_minorant}, there exists
  \( \dual\in\DUAL \) such that \( \closedconvexhull\np{\fonctionprimalbis_{c}}\np{\primal} \geq -\nscal{\primal}{\dual} \),
  for all \( \primal\in\PRIMAL \).
  As \( \np{0,-1} \in \closedconvexhull\np{\RR_{++}\StrictEpigraph\fonctionprimalbis}=
  \closedconvexhull\np{\StrictEpigraph\fonctionprimalbis_{c}}=
  \Epigraph~\closedconvexhull\np{\fonctionprimalbis_{c}} \)
  by~\eqref{eq:Epigraph_closedconvexhull_fonctionprimal_c}, we get that 
  \( -1 \geq \closedconvexhull\np{\fonctionprimalbis_{c}}\np{0} \geq -\nscal{0}{\dual}=0 \),
  hence $-1 \geq 0$ which is impossible.
  As a consequence, we get that
  \( \np{0,-1} \not\in \closedconvexhull\np{\RR_{++}\StrictEpigraph\fonctionprimalbis}\),
  that is, we have proved that
  Item~\ref{it:closedconvexhull_RR++_StrictEpigraph_fonctionprimalbis_cap_=} holds true.

  \medskip

  \noindent$\bullet$
  We prove that Item~\ref{it:closedconvexhull_RR++_StrictEpigraph_fonctionprimalbis_cap_=}
  $\implies$ Item~\ref{it:ba_LFM_fonctionprimalbis_leq_{0}}.
  Suppose that \( \np{0,-1} \not\in \closedconvexhull\np{\RR_{++}\StrictEpigraph\fonctionprimalbis}\).
  As \( \np{0,-1} \) is compact and \( \closedconvexhull\np{\RR_{++}\StrictEpigraph\fonctionprimalbis}\)
  is closed convex, there exists a nonzero \( \np{\dual,r}\in\DUAL\times\RR \) such that
  for all \( \np{\primal,t} \in \closedconvexhull\np{\RR_{++}\StrictEpigraph\fonctionprimalbis}\),
  \( \nscal{\primal}{\dual} +rt < \nscal{0}{\dual} +r\times (-1) = -r \)
  (see for example \cite[Theorem~1.1.5]{Zalinescu:2002}, \cite[Theorem~5.79, Corollary~5.80]{Aliprantis-Border:2006}).
  As \( \closedconvexhull\np{\RR_{++}\StrictEpigraph\fonctionprimalbis}
  = \closedconvexhull\np{\RR_{+}\StrictEpigraph\fonctionprimalbis}\)
  by~\eqref{eq:Epigraph_closedconvexhull_fonctionprimal_c}, 
  this implies that,
  for all \( \np{\primal,t} \in \StrictEpigraph\fonctionprimalbis\) and for all \( \lambda\geq 0\),
  \( \lambda\np{\nscal{\primal}{\dual} +rt} < -r \).
  With \( \lambda=0 \), we get \( 0 <-r \), that is, \( r < 0 \) and, without loss of generality, \( r=-1\). 
  By leeting \( \lambda \to +\infty \) in the inequality
  \( \nscal{\primal}{\dual} -t < -1/\lambda \), we get that \( \nscal{\primal}{\dual} -t \leq 0 \)
  for all \( \np{\primal,t} \in \StrictEpigraph\fonctionprimalbis\).
  This gives
  \[
    \LFM{\fonctionprimalbis}\np{\dual}
    = \sup_{\primal\in\PRIMAL} \nscal{\primal}{\dual} -\fonctionprimalbis\np{\primal}
    = \sup_{\primal\in\PRIMAL, \fonctionprimalbis\np{\primal} < t} \nscal{\primal}{\dual} -t
    = \sup_{\np{\primal,t} \in \StrictEpigraph\fonctionprimalbis}\nscal{\primal}{\dual} -t \leq 0
    \eqfinv
  \]
  that is, Item~\ref{it:ba_LFM_fonctionprimalbis_leq_{0}} holds true.
  \medskip

  \noindent$\bullet$
  We prove that Item~\ref{it:closedconvexhull_RR++_StrictEpigraph_fonctionprimalbis_cap_=}
  is equivalent to
  Item~\ref{it:closedconvexhullnpfonctionprimalbis_cnp0inna0,+infty}
  by contraposition (where the symbol~$\neg$ means the negation of a statement) as follows:
  \begin{align*}
    \neg \bp{ \textrm{Item~\ref{it:closedconvexhullnpfonctionprimalbis_cnp0inna0,+infty}} }
    &\iff
      \closedconvexhull\np{\fonctionprimalbis_{c}}\np{0} \not\in \na{0,+\infty} 
    \\
    &\iff
      \closedconvexhull\np{\fonctionprimalbis_{c}}\np{0}=-\infty
      \tag{as \( \closedconvexhull\np{\fonctionprimalbis_{c}}\np{0} \in
      \na{-\infty,0,+\infty} \) by~\eqref{eq:1-homogeneous_function_zero}}       
    \\
    &\iff
      \na{0}\times\RR \subset \Epigraph~\closedconvexhull\np{\fonctionprimalbis_{c}} 
      \tag{by definition of \( \Epigraph~\closedconvexhull\np{\fonctionprimalbis_{c}} \)}
    \\
    &\iff
      \na{0}\times \OpenIntervalOpen{-\infty}{0} \subset \Epigraph~\closedconvexhull\np{\fonctionprimalbis_{c}} 
      \tag{because \( \Epigraph~\closedconvexhull\np{\fonctionprimalbis_{c}} \) is an epigraph}
    \\
    &\iff
      \na{0}\times \OpenIntervalOpen{-\infty}{0} \subset \closedconvexhull\np{\RR_{++}\StrictEpigraph\fonctionprimalbis}
      \tag{as \( \closedconvexhull\np{\RR_{++}\StrictEpigraph\fonctionprimalbis}=
      \Epigraph~\closedconvexhull\np{\fonctionprimalbis_{c}} \) by~\eqref{eq:Epigraph_closedconvexhull_fonctionprimal_c}}
    \\
    &\iff
      \np{0,-1} \in \closedconvexhull\np{\RR_{++}\StrictEpigraph\fonctionprimalbis}
      \tag{as \( \closedconvexhull\np{\RR_{++}\StrictEpigraph\fonctionprimalbis} \) is a cone}
    \\
    &\iff
      \neg \bp{ \textrm{Item~\ref{it:closedconvexhull_RR++_StrictEpigraph_fonctionprimalbis_cap_=}} }
      \eqfinp 
  \end{align*}
  \medskip

  \noindent$\bullet$
  We have that Item~\ref{it:LFMbifonctionprimalbis_cnp0inna0,+infty} implies
  Item~\ref{it:closedconvexhullnpfonctionprimalbis_cnp0inna0,+infty} 
  because, by the very definition~\eqref{eq:closed_convex_hull_of_a_function}
  of the function~\( \closedconvexhull\np{\fonctionprimalbis_{c}} \), we have the inequalities
  \( \LFMbi{\fonctionprimalbis}_{c}\np{0} \leq 
  \closedconvexhull\np{\fonctionprimalbis_{c}}\np{0} \)
  (see Equation~\eqref{eq:LFMbi_fonctionprimal_leq_closedconvexhull_fonctionprimal_leq_fonctionprimal}).
  \medskip

  \noindent$\bullet$
  Finally, Item~\ref{it:ba_LFM_fonctionprimalbis_c_leq_{0}} implies
  Item~\ref{it:LFMbifonctionprimalbis_cnp0inna0,+infty}
  because, \( \LFM{\fonctionprimalbis_{c}} = \Indicator{\na{ \LFM{\fonctionprimalbis} \leq
      0 }} \) by~\eqref{eq:LFM_conic_hull},
  hence  \( \LFMbi{\fonctionprimalbis_{c}}\np{0}
  = \LFMr{\Indicator{\na{ \LFM{\fonctionprimalbis}\leq 0}}}\np{0}
  =\SupportFunction{\na{ \LFM{\fonctionprimalbis}\leq 0}}\np{0} > -\infty \) --- where 
\( \SupportFunction{} \) denotes the support function --- 
  as \( \na{ \LFM{\fonctionprimalbis}\leq 0 } \neq\emptyset \) by
  Item~\ref{it:ba_LFM_fonctionprimalbis_c_leq_{0}}.
\end{proof}

\subsection{Proof of Proposition~\ref{pr:the_classic_S-procedure_is_valid}}
\label{Proof_of_Proposition_pr:the_classic_S-procedure_is_valid}

\begin{proof}
  
  \noindent$\bullet$
  We prove that
  Item~\ref{it:the_classic_S-procedure_is_valid}
  $\iff$
  Item~\ref{it:the_classic_S-procedure_is_valid_RR_++^constraintdim_StrictEpigraph}.

  For that purpose, we first show that the classic S-procedure
  in~\S\ref{The_S-procedure_revisited_with_the_conditional_infimum} is a special
  case of the general S-procedure
  in~\S\ref{New_necessary_and_sufficient_conditions_for_the_S-procedure_to_be_valid}.
  We follow \cite{Volle-Barro:2024}, and we take \( \PRIMAL= \DUAL = \RR^{\constraintdim}\) and 
  \begin{subequations}
    \begin{align}
      \fonctiondepart\np{\depart,\primal_{1},\ldots,\primal_{\constraintdim}}
      &=
        \fonctiondepart_{0}\np{\depart} \UppPlus \sum_{j=1}^{\constraintdim}  \Indicator{\na{-\fonctiondepart_j\np{\depart} \leq \primal_j}}
        \eqsepv   \forall \np{\depart,\primal_{1},\ldots,\primal_{\constraintdim}} \in \DEPART\times\RR^{\constraintdim} 
        \eqfinp
        \intertext{In that case, the value function~\eqref{eq:S_fonctionprimal-procedure_value_function}
        is given, for any \( \np{\primal_{1},\ldots,\primal_{\constraintdim}} \in \RR^{\constraintdim} \), by}
        \fonctionprimalbis\np{\primal_{1},\ldots,\primal_{\constraintdim}}
      & =
        \inf_{\depart \in \DEPART} \bp{
        \fonctiondepart_{0}\np{\depart} \UppPlus \sum_{j=1}^{\constraintdim}  \Indicator{\na{-\fonctiondepart_j\np{\depart} \leq \primal_j}} }
        \tag{by~\eqref{eq:S_fonctionprimal-procedure_value_function}}
      \\
      &=
        \inf \defset{ \fonctiondepart_{0}\np{\depart} }%
        { \depart\in\DEPART \eqsepv -\fonctiondepart_{1}\np{\depart} \leq\primal_{1}, \ldots,
        -\fonctiondepart_{\constraintdim}\np{\depart} \leq\primal_{\constraintdim} }
        \nonumber
      \\
      &=
        \InfCond{\fonctiondepart_{0}}{ \defset{ \depart\in\DEPART }%
        { -\fonctiondepart_{1}\np{\depart} \leq\primal_{1}, \ldots,
        -\fonctiondepart_{\constraintdim}\np{\depart} \leq\primal_{\constraintdim} } }
        \tag{by~\eqref{eq:subset_conditional_infimum}}
      \\
      &=
\BInfCond{\fonctiondepart_{0}}{\Epigraph_{\RR_{+}^{\constraintdim}}
        \np{-\fonctiondepart_{1}, \ldots,-\fonctiondepart_{\constraintdim} }}\np{\primal_{1},\ldots,\primal_{\constraintdim}}
        \tag{by~\eqref{eq:value_function_of_the_classic_mathematical_programming_minimization_problem}
        and definition~\eqref{eq:RR_+p-epigraph} 
        of \( \Epigraph_{\RR_{+}^{\constraintdim}}\np{-\fonctiondepart_{1}, \ldots,-\fonctiondepart_{\constraintdim} } \) }
    \end{align}
  \end{subequations}
We have obtained that \( \fonctionprimalbis = \BInfCond{\fonctiondepart_{0}}{\Epigraph_{\RR_{+}^{\constraintdim}}
      \np{-\fonctiondepart_{1}, \ldots,-\fonctiondepart_{\constraintdim} }} \).
  Then, we use the equivalence between
  Item~\ref{it:the_S_fonctionprimal-procedure_is_valid}
  and
  Item~\ref{it:closedconvexhullnpfonctionprimalbis_cnp0inna0,+infty}
  in
  Proposition~\ref{pr:New_necessary_and_sufficient_conditions_for_the_S-procedure_to_be_valid},
  and get that the {S$_{\fonctionprimal}$-procedure} is valid.
  \medskip

  \noindent$\bullet$
  We prove that
  Item~\ref{it:the_classic_S-procedure_is_valid}
  $\iff$
  Item~\ref{it:the_classic_S-procedure_is_valid_StrictEpigraph}.

  For that purpose, we first show that the classic S-procedure
  in~\S\ref{The_S-procedure_revisited_with_the_conditional_infimum}
  is a special case of the general S-procedure
  in~\ref{New_necessary_and_sufficient_conditions_for_the_S-procedure_to_be_valid}.
  We take \( \PRIMAL= \DUAL = \RR^{\constraintdim}\) and 
  \begin{subequations}
    \begin{align*}
      \fonctiondepart\np{\depart,\primal_{1},\ldots,\primal_{\constraintdim}}
      &=
        \fonctiondepart_{0}\np{\depart} \UppPlus \sum_{j=1}^{\constraintdim}
        \Indicator{\na{-\fonctiondepart_j\np{\depart}=\primal_j}}
        \UppPlus \Indicator{\RR_{-}^{\constraintdim}}\np{\primal_{1},\ldots,\primal_{\constraintdim}}
        \eqsepv   \forall \np{\depart,\primal_{1},\ldots,\primal_{\constraintdim}} \in \DEPART\times\RR^{\constraintdim} 
        \eqfinp
        \intertext{In that case, the value function~\eqref{eq:S_fonctionprimal-procedure_value_function} is given, for any \( \np{\primal_{1},\ldots,\primal_{\constraintdim}} \in \RR^{\constraintdim} \), by}
        \fonctionprimalbis\np{\primal_{1},\ldots,\primal_{\constraintdim}}
      & =
        \inf_{\depart \in \DEPART} \bp{
        \fonctiondepart_{0}\np{\depart} \UppPlus \sum_{j=1}^{\constraintdim}  \Indicator{\na{-\fonctiondepart_j\np{\depart} = \primal_j}} }
        \UppPlus \Indicator{\RR_{-}^{\constraintdim}}\np{\primal_{1},\ldots,\primal_{\constraintdim}}
        \tag{by~\eqref{eq:S_fonctionprimal-procedure_value_function}}
      \\
      &=
        \inf \defset{ \fonctiondepart_{0}\np{\depart} }%
        { \depart\in\DEPART \eqsepv -\fonctiondepart_{1}\np{\depart} =\primal_{1}, \ldots,
        -\fonctiondepart_{\constraintdim}\np{\depart} =\primal_{\constraintdim} }
        \UppPlus \Indicator{\RR_{-}^{\constraintdim}}\np{\primal_{1},\ldots,\primal_{\constraintdim}}
        \nonumber
      \\
      &=
        \InfCond{\fonctiondepart_{0}}{ \defset{ \depart\in\DEPART }%
        { -\fonctiondepart_{1}\np{\depart} =\primal_{1}, \ldots,
        -\fonctiondepart_{\constraintdim}\np{\depart} =\primal_{\constraintdim} } }
        \UppPlus \Indicator{\RR_{-}^{\constraintdim}}\np{\primal_{1},\ldots,\primal_{\constraintdim}}
        \tag{by~\eqref{eq:subset_conditional_infimum}}
      \\
      &=
        \InfCond{\fonctiondepart_{0}}{\Epigraph_{}
        \np{-\fonctiondepart_{1}, \ldots,-\fonctiondepart_{\constraintdim} }}\np{\primal_{1},\ldots,\primal_{\constraintdim}}
        \UppPlus \Indicator{\RR_{-}^{\constraintdim}}\np{\primal_{1},\ldots,\primal_{\constraintdim}}
        \eqfinp
        \tag{by~\eqref{eq:mapping_ConditionalInfimum}}
    \end{align*}
  \end{subequations}
We have obtained that \( \fonctionprimalbis =
        \InfCond{\fonctiondepart_{0}}{\np{-\fonctiondepart_{1},
        \ldots,-\fonctiondepart_{\constraintdim} }}
        \UppPlus \Indicator{\RR_{-}^{\constraintdim}} \).
  Then, we use the equivalence between
  Item~\ref{it:the_S_fonctionprimal-procedure_is_valid}
  and
  Item~\ref{it:closedconvexhullnpfonctionprimalbis_cnp0inna0,+infty}
  in
  Proposition~\ref{pr:New_necessary_and_sufficient_conditions_for_the_S-procedure_to_be_valid},
  and get that the {S$_{\fonctionprimal}$-procedure} is valid.
\end{proof}

\section{Recalls on closed convex and conic hulls of a function}
\label{Recalls_on_closed_convex_and_conic_hulls_of_a_function}

Here, we gather well-known definitions and properties, and stress the use of the
strict epigraph of a function. For any set~$\UNCERTAIN$ and function
\( \fonctionuncertain \colon \UNCERTAIN \to \barRR \), its \emph{strict epigraph} is
\( \epigraph_{s}\fonctionuncertain= \defset{
  \np{\uncertain,t}\in\UNCERTAIN\times\RR}%
{\fonctionuncertain\np{\uncertain} < t} \).  

\subsubsubsection{Closed convex hull of a function}

Let \( \PRIMAL \) be a topological vector space and \( \fonctionprimal \colon \PRIMAL \to \barRR \) be a function.
The \emph{lower semicontinuous convex envelope},
or \emph{closed convex hull}, 
of the function 
\( \fonctionprimal \colon \PRIMAL \to \barRR \) is the function
\( \closedconvexhull{\fonctionprimal} \colon \PRIMAL \to \barRR \) given by
\begin{equation}
  \closedconvexhull{\fonctionprimal}
  =
  \sup \defset{ \fonctiontrois \leq \fonctionprimal }%
  { \fonctiontrois \colon \PRIMAL \to \barRR \textrm{ is lower semicontinuous convex}}
  \eqfinp
\label{eq:closed_convex_hull_of_a_function}  
\end{equation}

\begin{lemma}
  Let \( \PRIMAL, \DUAL \) be paired vector spaces
  (see~\S\ref{Conditional_infimum_and_one-sided_homogeneous_sublinear_couplings}),
  and \( \fonctionprimal \colon \PRIMAL \to \barRR \) be a function.
  We have that
  \begin{subequations}
    \begin{align}
      \LFMbi{\fonctionprimal}
      & \leq 
        \closedconvexhull\fonctionprimal \leq
        \fonctionprimal
        \eqfinv        
        \label{eq:LFMbi_fonctionprimal_leq_closedconvexhull_fonctionprimal_leq_fonctionprimal}        
      \\
      \Epigraph\np{\closedconvexhull\fonctionprimal}
      &=
        \closedconvexhull\np{\Epigraph\fonctionprimal}
        = \closedconvexhull\np{\StrictEpigraph\fonctionprimal}
        \eqfinp
        \label{eq:Epigraph_closedconvexhull_fonctionprimal_closedconvexhull_Epigraph_fonctionprimal}
    \end{align}
  \end{subequations}
\end{lemma}

\begin{proof}
  We prove~\eqref{eq:LFMbi_fonctionprimal_leq_closedconvexhull_fonctionprimal_leq_fonctionprimal}.
  By the very definitions of the functions~\( \LFMbi{\fonctionprimal} \)
  and~\( \closedconvexhull\fonctionprimal \),
  we have the inequalities~\eqref{eq:LFMbi_fonctionprimal_leq_closedconvexhull_fonctionprimal_leq_fonctionprimal}
  \cite[Theorem~2.3.1 (iv)]{Zalinescu:2002}.

  We prove~\eqref{eq:Epigraph_closedconvexhull_fonctionprimal_closedconvexhull_Epigraph_fonctionprimal}.
  The equality 
  \( \Epigraph\np{\closedconvexhull\fonctionprimal}
  = \closedconvexhull\np{\Epigraph\fonctionprimal} \) follows from the very definition
  of the function~\( \closedconvexhull\fonctionprimal \)
  \cite[p.~63]{Zalinescu:2002}.
  As \( \StrictEpigraph\fonctionprimal \subset \Epigraph\fonctionprimal \),
  hence \( \closedconvexhull\np{\StrictEpigraph\fonctionprimal}
  \subset \closedconvexhull\np{\Epigraph\fonctionprimal} \),
  there remains to show that
  \( \closedconvexhull\np{\Epigraph\fonctionprimal} \subset
  \closedconvexhull\np{\StrictEpigraph\fonctionprimal} \).
  Let \( \np{\uncertain,t} \in \closedconvexhull\np{\Epigraph\fonctionprimal} \) be the limit of a sequence
  \( \sequence{\np{\uncertain_{n},t_{n}}}{n\in\NN^*} \subset \convexhull\np{\Epigraph\fonctionprimal} \),
  that is, for all \( n\in\NN^* \), \( \uncertain_{n}=\sum_{k\in\Lambda_{n}} \lambda_{n}^{k}\uncertain_{n}^{k}\),
  \( t_{n}=\sum_{k\in\Lambda_{n}} \lambda_{n}^{k}t_{n}^{k}\),
  where the \( \sequence{\lambda_{n}^{k}}{k\in\Lambda_{n}} \) are the coefficients of a convex combination,
  and \( \fonctionprimal\np{\uncertain_{n}^{k}} \leq t_{n}^{k} \) for all \( k\in\Lambda_{n} \).
  It is clear that \( \np{\uncertain,t} \) is the limit of the sequence
  \( \sequence{\np{\uncertain_{n},t_{n}+1/n}}{n\in\NN^*} \subset \convexhull\np{\StrictEpigraph\fonctionprimal} \),
  because \( t_{n}+1/n=\sum_{k\in\Lambda_{n}} \lambda_{n}^{k}\np{t_{n}^{k}+1/n} \),
  where \( \fonctionprimal\np{\uncertain_{n}^{k}} \leq t_{n}^{k} < t_{n}^{k}+1/n\) for all \( k\in\Lambda_{n} \).
  This ends the proof.
\end{proof}

\subsubsubsection{Conic hull of a function}

\begin{subequations}
  Let $\PRIMAL$ be a (real) vector space. 
  We say that a function \( \fonctionprimal \colon \PRIMAL \to \barRR \) is \emph{$1$-homogeneous} if 
  \begin{equation}
    {\fonctionprimal\np{\lambda\primal} = \lambda \fonctionprimal\np{\primal}}
    \eqsepv \forall \lambda \in {\RR_{++}}
    \eqsepv \forall \primal \in \PRIMAL
    \eqfinv
    \label{eq:1-homogeneous_function}
  \end{equation}
  or, equivalently, if \( \Epigraph\fonctionprimal \) (or \( \StrictEpigraph\fonctionprimal \)) is a cone.
  Then, necessarily, we have that
  \begin{equation}
    \fonctionprimal\np{0} \in \na{-\infty,0,+\infty}
    \eqfinp
    \label{eq:1-homogeneous_function_zero}
  \end{equation}
  
  The \emph{$1$-homogeneous envelope},
  or \emph{conic hull},
  of the function \( \fonctionprimal \colon \PRIMAL \to \barRR \) is the function
  \( \fonctionprimal_{c} \colon \PRIMAL \to \barRR \) given by
  \begin{equation}
    \fonctionprimal_{c}
    =
    \sup \defset{ \fonctiontrois \leq \fonctionprimal }%
    { \fonctiontrois \colon \PRIMAL \to \barRR \textrm{ is $1$-homogeneous}}
    \eqfinp
    \label{eq:1-homogeneous_envelope}   
  \end{equation}  
\end{subequations}

\begin{subequations}
  \begin{lemma}
    For any {function}~$\fonctionprimal \colon \PRIMAL \to {\barRR}$,
    we have that 
    \begin{equation}
      \StrictEpigraph\fonctionprimal_{c}
      = \RR_{++}\StrictEpigraph\fonctionprimal
      \eqfinp
      \label{eq:StrictEpigraph_fonctionprimal_c_RR++_StrictEpigraph_fonctionprimal}
    \end{equation}
    Its \emph{closed convex conic hull} \( \closedconvexhull\fonctionprimal_{c} \colon \PRIMAL \to \barRR \)
    is $1$-homogeneous, with
    \begin{equation}
      \Epigraph\np{\closedconvexhull\fonctionprimal_{c}}
      =
      \closedconvexhull\np{\RR_{++}\StrictEpigraph\fonctionprimal}
      = \closedconvexhull\np{\RR_{+}\StrictEpigraph\fonctionprimal}
      \eqfinp
      \label{eq:Epigraph_closedconvexhull_fonctionprimal_c}
    \end{equation}
    %
    Moreover, we have that
    \begin{equation}
      \LFM{\fonctionprimal_{c}} = \Indicator{\na{ \LFM{\fonctionprimal} \leq 0}}
      \eqfinp 
      \label{eq:LFM_conic_hull}  
    \end{equation}
  \end{lemma}
\end{subequations}

\begin{proof}
  Let us call \( \hat\fonctionprimal_{c} \) the function given by~\eqref{eq:conic_hull}.
  On the one hand, it is easy to see that \( \hat\fonctionprimal_{c} \) 
  is $1$-homogeneous and below~\( \fonctionprimal \). By definition~\eqref{eq:1-homogeneous_envelope},
  we deduce that \( \hat\fonctionprimal_{c} \leq \fonctionprimal_{c} \).
  On the other hand, consider a $1$-homogeneous function $\fonctiontrois \colon \PRIMAL \to \barRR$
  such that $\fonctiontrois(x) \le \fonctionprimal(x)$. Then, for any $\alpha >0$ and any $x\in\PRIMAL$ we have that
  $\fonctiontrois(x) = \alpha \fonctiontrois(x/\alpha) \le  \alpha \fonctionprimal(x/\alpha)$ which implies that
  $\fonctiontrois(x) \le \inf_{\alpha >0}\fonctionprimal(x/\alpha) = \hat\fonctionprimal_{c}(x)$.
  We deduce from definition~\eqref{eq:1-homogeneous_envelope} that \( \fonctionprimal_{c} \leq \hat\fonctionprimal_{c} \).
  We conclude that \( \fonctionprimal_{c}=\hat\fonctionprimal_{c} \), that is,
  the definitions~\eqref{eq:conic_hull} and~\eqref{eq:1-homogeneous_envelope} coincide.
  \medskip

  We have that \( \StrictEpigraph\fonctionprimal_{c}
  = \RR_{++}\StrictEpigraph\fonctionprimal \) because
  \begin{align*}
    \np{\primal,t} \in\StrictEpigraph\fonctionprimal_{c}
    &\iff
      \inf_{\rho >0} \rho\fonctionprimal\np{\primal/\rho}= \fonctionprimal_{c}\np{\primal} < t 
      \tag{by~\eqref{eq:conic_hull} as \( \fonctionprimal_{c}=\hat\fonctionprimal_{c} \)}
    \\
    &\iff
      \exists \rho >0 \eqsepv \rho\fonctionprimal\np{\primal/\rho} < t
    \\
    &\iff
      \exists \rho >0 \eqsepv \frac{1}{\rho} \np{\primal,t} \in\StrictEpigraph\fonctionprimal
      \iff
      \np{\primal,t} \in \RR_{++} \StrictEpigraph\fonctionprimal
      \eqfinp
  \end{align*}
  \medskip


  We prove~\eqref{eq:Epigraph_closedconvexhull_fonctionprimal_c}.
  The equality \( \Epigraph\np{\closedconvexhull\fonctionprimal_{c}}=
  \closedconvexhull\np{\RR_{++}\StrictEpigraph\fonctionprimal} \)
  follows from~\eqref{eq:Epigraph_closedconvexhull_fonctionprimal_closedconvexhull_Epigraph_fonctionprimal}
  and~\eqref{eq:StrictEpigraph_fonctionprimal_c_RR++_StrictEpigraph_fonctionprimal}.
  As \( \closedconvexhull\np{\RR_{++}\StrictEpigraph\fonctionprimal} \subset
  \closedconvexhull\np{\RR_{+}\StrictEpigraph\fonctionprimal} \),
  there remains to show that
  \( \closedconvexhull\np{\RR_{+}\StrictEpigraph\fonctionprimal} \subset
  \closedconvexhull\np{\RR_{++}\StrictEpigraph\fonctionprimal} \).
  We set \( \Cone=\RR_{++}\StrictEpigraph\fonctionprimal \), a cone.
  If \( \Cone=\emptyset \), the above inclusion is true.
  Else, we observe that
  \( \RR_{+}\StrictEpigraph\fonctionprimal= \Cone\cup\na{0} \).
  Then, we have that \( \Cone\cup\na{0} \subset \overline{\Cone}\cup\na{0}= \overline{\Cone}\) 
  because \( 0\in \overline{\Cone} \) since \( \Cone \) is a nonempty cone.
  As \( \overline{\Cone} \subset \overline{\convexhull{\Cone}}=\closedconvexhull{\Cone} \),
  we get that \( \Cone\cup\na{0} \subset \closedconvexhull{\Cone} \), hence that
  \( \closedconvexhull\np{\Cone\cup\na{0}} \subset \closedconvexhull{\Cone} \).
  Thus, we have shown that \( \closedconvexhull\np{\RR_{+}\StrictEpigraph\fonctionprimal} \subset
  \closedconvexhull\np{\RR_{++}\StrictEpigraph\fonctionprimal} \),
  hence that \( \closedconvexhull\np{\RR_{+}\StrictEpigraph\fonctionprimal} =
  \closedconvexhull\np{\RR_{++}\StrictEpigraph\fonctionprimal} \).
  \medskip

  Finally Equation~\eqref{eq:LFM_conic_hull} follows from:  
  for any \( \dual\in\DUAL \), we have that
  \begin{align*}
    \LFM{\fonctionprimal_{c}}\np{\dual}
    &=
      \sup_{\primal\in\PRIMAL} \bp{ \nscal{\primal}{\dual} - \fonctionprimal_{c}\np{\primal} }
    \\
    &=
      \sup_{\primal\in\PRIMAL} \sup_{\rho >0} 
      \bp{ \nscal{\primal}{\dual} - \rho\fonctionprimal\np{\primal/\rho} }
      \tag{by~\eqref{eq:conic_hull}}
    \\
    &=
      \sup_{\primal\in\PRIMAL} \sup_{\rho >0} 
      \bp{ \rho\nscal{\primal/\rho}{\dual} - \rho\fonctionprimal\np{\primal/\rho} }
    \\
    &=
      \sup_{\rho >0} \rho\sup_{\primal\in\PRIMAL} 
      \bp{ \nscal{\primal}{\dual} - \fonctionprimal\np{\primal} }
    \\
    &=
      \sup_{\rho >0} \rho \LFM{\fonctionprimal}\np{\dual}
    \\
    &=
      \Indicator{\na{ \LFM{\fonctionprimal} \leq 0 }}\np{\dual}
      \eqfinp 
  \end{align*}
  This ends the proof.
\end{proof}

\newcommand{\noopsort}[1]{} \ifx\undefined\allcaps\def\allcaps#1{#1}\fi

\end{document}